\documentclass[12pt,a4paper]{amsart}
\usepackage[czech,english]{babel}
\usepackage{amssymb,amscd,eucal}
\usepackage{tikz}
\usepackage{hyperref}

\calclayout

\newcommand{\+}{\nobreakdash-}
\renewcommand{\:}{\colon}
\renewcommand{\.}{\mskip.5\thinmuskip\relax}

\newcommand{\rarrow}{\longrightarrow}
\newcommand{\ot}{\otimes}
\newcommand{\ocn}{\odot}
\newcommand{\st}{\star}
\newcommand{\tim}{\rightthreetimes}

\newcommand{\A}{{\mathsf A}}
\newcommand{\B}{{\mathsf B}}
\newcommand{\C}{{\mathsf C}}
\newcommand{\D}{{\mathsf D}}
\newcommand{\E}{{\mathsf E}}
\newcommand{\F}{{\mathsf F}}
\newcommand{\K}{{\mathsf K}}
\renewcommand{\L}{{\mathsf L}}
\newcommand{\R}{{\mathsf R}}

\newcommand{\cC}{{\mathcal C}}
\newcommand{\cE}{{\mathcal E}}
\newcommand{\cF}{{\mathcal F}}

\newcommand{\cM}{{\mathcal M}}
\newcommand{\cN}{{\mathcal N}}

\newcommand{\cR}{{\mathsf R}}

\newcommand{\bcK}{{\boldsymbol{\mathcal K}}}
\newcommand{\bcL}{{\boldsymbol{\mathcal L}}}
\newcommand{\bcM}{{\boldsymbol{\mathcal M}}}

\newcommand{\bcS}{{\boldsymbol{\mathcal S}}}

\newcommand{\Z}{{\mathbb Z}}
\newcommand{\boL}{{\mathbb L}}
\newcommand{\boR}{{\mathbb R}}
\newcommand{\boT}{{\mathbb T}}
\newcommand{\boQ}{{\mathbb Q}}

\newcommand{\Sets}{\mathsf{Sets}}

\newcommand{\fA}{{\mathfrak A}}
\newcommand{\fC}{{\mathfrak C}}
\newcommand{\fD}{{\mathfrak D}}
\newcommand{\fI}{{\mathfrak I}}
\newcommand{\fJ}{{\mathfrak J}}

\newcommand{\fP}{{\mathfrak P}}
\newcommand{\fQ}{{\mathfrak Q}}
\newcommand{\fR}{{\mathfrak R}}
\newcommand{\fS}{{\mathfrak S}}
\newcommand{\fU}{{\mathfrak U}}
\newcommand{\fV}{{\mathfrak V}}

\newcommand{\fp}{{\mathfrak p}}

\DeclareMathOperator{\Hom}{Hom}
\DeclareMathOperator{\Ext}{Ext}
\DeclareMathOperator{\Tor}{Tor}
\DeclareMathOperator{\Ann}{Ann}
\DeclareMathOperator{\coker}{coker}
\DeclareMathOperator{\colim}{colim}

\DeclareMathOperator{\Hot}{\mathsf{Hot}}
\DeclareMathOperator{\Acycl}{\mathsf{Acycl}}
\DeclareMathOperator{\Add}{\mathsf{Add}}
\DeclareMathOperator{\Prod}{\mathsf{Prod}}
\DeclareMathOperator{\Spec}{\mathsf{Spec}}

\newcommand{\modl}{{\operatorname{\mathsf{--mod}}}}
\newcommand{\comodl}{{\operatorname{\mathsf{--comod}}}}

\newcommand{\simodl}{{\operatorname{\mathsf{--simod}}}}

\newcommand{\contra}{{\operatorname{\mathsf{--contra}}}}
\newcommand{\discr}{{\operatorname{\mathsf{discr--}}}}
\newcommand{\discrl}{{\operatorname{\mathsf{--discr}}}}
\newcommand{\smooth}{{\operatorname{\mathsf{--smooth}}}}
\newcommand{\sicntr}{{\operatorname{\mathsf{--sicntr}}}}

\newcommand{\tors}{{\operatorname{\mathsf{-tors}}}}
\newcommand{\ctra}{{\operatorname{\mathsf{-ctra}}}}

\renewcommand{\b}{{\mathsf{b}}}
\newcommand{\co}{{\mathsf{co}}}
\newcommand{\ctr}{{\mathsf{ctr}}}
\newcommand{\si}{{\mathsf{si}}}
\newcommand{\abs}{{\mathsf{abs}}}
\newcommand{\proj}{{\mathsf{proj}}}
\newcommand{\inj}{{\mathsf{inj}}}
\newcommand{\fdim}{{\mathsf{fdim}}}
\newcommand{\cs}{{\mathsf{cs}}}

\newcommand{\rop}{{\mathrm{op}}}

\newcommand{\bu}{{\text{\smaller\smaller$\scriptstyle\bullet$}}}
\newcommand{\lrarrow}{\.\relbar\joinrel\relbar\joinrel\rightarrow\.}

\newcommand{\oc}{\mathbin{\text{\smaller$\square$}}}

\theoremstyle{plain}
\newtheorem{thm}{Theorem}[section]

\newtheorem{lem}[thm]{Lemma}
\newtheorem{prop}[thm]{Proposition}
\newtheorem{cor}[thm]{Corollary}
\theoremstyle{definition}
\newtheorem{ex}[thm]{Example}
\newtheorem{exs}[thm]{Examples}
\newtheorem{rem}[thm]{Remark}

\newcommand{\Section}[1]{\bigskip\section{#1}\medskip}

\begin{document}

\title{The tilting-cotilting correspondence}

\author{Leonid Positselski and Jan \v S\v tov\'\i\v cek}

\address{Leonid Positselski, Institute of Mathematics, Czech Academy
of Sciences, \v Zitn\'a~25, 115~67 Prague~1, Czech Republic; and
\newline\indent Laboratory of Algebraic Geometry, National Research
University Higher School of Economics, Moscow 119048, Russia; and
\newline\indent Sector of Algebra and Number Theory, Institute for
Information Transmission Problems, Moscow 127051, Russia; and
\newline\indent Department of Mathematics, Faculty of Natural Sciences,
University of Haifa, Mount Carmel, Haifa 31905, Israel}

\email{posic@mccme.ru}

\address{Jan {\v S}{\v{t}}ov{\'{\i}}{\v{c}}ek, Charles University
in Prague, Faculty of Mathematics and Physics, Department of Algebra,
Sokolovsk\'a 83, 186 75 Praha, Czech Republic}

\email{stovicek@karlin.mff.cuni.cz}

\begin{abstract}
 To a big $n$\+tilting object in a complete, cocomplete abelian
category $\A$ with an injective cogenerator we assign a big
$n$\+cotilting object in a complete, cocomplete abelian category $\B$
with a projective generator, and vice versa.
 Then we construct an equivalence between the (conventional or
absolute) derived categories of $\A$ and~$\B$.
 Under various assumptions on $\A$, which cover a wide range of examples
(for instance, if $\A$ is a module category or, more generally,
a locally finitely presentable Grothendieck abelian category),
we show that $\B$ is the abelian category of contramodules
over a topological ring and that the derived equivalences
are realized by a~contramodule-valued variant of the usual
derived Hom-functor.
\end{abstract}

\maketitle

\setcounter{tocdepth}{2} % sets the "depth" of the table of contents 
\tableofcontents

\section{Introduction}
\medskip

Tilting theory has its roots in representation theory of finite-dimensional algebras and has evolved into a powerful derived Morita theory with numerous applications in various fields of algebra and algebraic geometry~\cite{Handbook}. Our motivation in this paper stems from a beautiful correspondence between finite-dimensional tilting and finite-dimensional cotilting modules which goes back to Brenner and Butler~\cite{BB79} and Miyashita~\cite{Miy}.

We exhibit an equally symmetric and easy-to-state correspondence (which we call the tilting-cotilting correspondence) in the context of very general abelian categories. This puts several recent generalizations of the Brenner--Butler correspondence (see for instance \cite{Ba-eq,BMT,NSZ,PV,St}) into a unified framework.

%However, we also follow another important motive in this paper.
As we illustrate on various classes of examples, the resulting derived equivalences which we obtain are often as concrete as in the classical finite-dimensional situation. They are often induced by a tilting bimodule ${_\fA T_\fR}$ over two topological rings $\fA$ and $\fR$ which is discrete as a module over either of the rings. In this situation, one side of our tilting equivalence is the Grothendieck category $\A$ of all discrete left $\fA$\+modules, while the other side $\B$ is the category of so-called left $\fR$\+contramodules (a concept related to but different from a topological module). Roughly speaking, contramodules are $\fR$\+modules where well-behaved infinite linear combinations of elements are defined, as long as the coefficients from the ring converge to zero.

%It turns out that one side of our tilting-cotilting correspondence can be rather often described as a category of models of an additive infinitary algebraic theory in the sense of \cite{Wr} (see also \cite[Introduction]{PR}) and, in this context, a derived equivalence with the other end of the tilting-cotilting correspondence is provided by a usual derived Hom-functor. In a wide range of algebraic examples, the infinitary aspect is simply captured by a topology on a ring and the models of the corresponding algebraic theory are so-called contramodules over that topological ring (a concept related to but different from a topological module).

The results which we present here are also strongly inspired by the fact that examples of categories of discrete modules and contramodules over topological rings naturally arise as the categories of comodules and contramodules over coalgebras, respectively.
In this case, the derived equivalences induced by tilting and cotilting 
objects sometimes coincide with a comodule-contramodule correspondence 
introduced by the first-named author
in~\cite[\S0.2--3 and Chapters~5--6]{Psemi}.
This direction is further pursued
%Following this direction, further aspects of the connection between the tilting theory and comodule-contramodule correspondence are also studied
in \cite{Pperp,PStinfty}.

\medskip

Let us explain our results and their context more in detail. The nowadays classical theorem of Brenner and Butler~\cite{BB79}, viewed from the perspective of tilting derived equivalences by Happel~\cite{Hap87} and Cline, Parshall and Scott \cite{CPS86}, can be stated as follows. If $A$ is a finite-dimensional algebra, $T\in A\modl_\fdim$ is a finite-dimensional tilting module in the sense of Miyashita~\cite{Miy} and $B=\Hom_A(T,T)^\rop$ is its (finite-dimensional) endomorphism algebra, there is a triangle equivalence $\D^\b(A\modl_\fdim) \simeq \D^\b(B\modl_\fdim)$ which sends $T \in A\modl_\fdim$ to the projective generator $B \in B\modl_\fdim$. Moreover, the vector-space dual $W = A^* \in A\modl_\fdim$, which is an injective cogenerator there, is sent to a cotilting module in $B\modl_\fdim$ (which is defined formally dually to tilting modules). Hence the situation is completely self-dual and is illustrated in Figure~\ref{the-correspondence-fig}.

\begin{figure}
  \begin{center}
    \begin{tikzpicture}
      % The ambient derived category
      \draw (0,0) rectangle (5,3.7);
      \draw (1.5,3.3) node {$\D^\st(\A) \simeq \D^\st(\B)$};
    
      % \A, \B and the intersection
      \fill[gray!20!white, draw=black] (1.75, 1.5) ellipse (1.25 and 1);
      \draw (.8,2.55) node {$\A$};
      \fill[gray!20!white, draw=black] (3.25, 1.5) ellipse (1.25 and 1);
      \draw (4.25,2.55) node {$\B$};
      \begin{scope}
      \clip (3.25, 1.5) ellipse (1.25 and 1);
      \fill[gray!40!white, draw=black] (1.75, 1.5) ellipse (1.25 and 1);
      \end{scope}
    
      % Distinguished objects
      \draw (2.1,1.5) node{$\bullet$};
      \draw (1.82,1.53) node{$\scriptstyle{T}$};
      \draw (2.9,1.5) node{$\bullet$};
      \draw (3.23,1.53) node{$\scriptstyle{W}$};
    \end{tikzpicture}
  \end{center}
  \caption{The tilting-cotilting correspondence: $T\in\A$ is a tilting
	object and $W\in\A$ an injective cogenerator, while $T\in\B$ is
	a projective generator and $W\in\B$ is a cotilting object.}
  \label{the-correspondence-fig}
\end{figure}
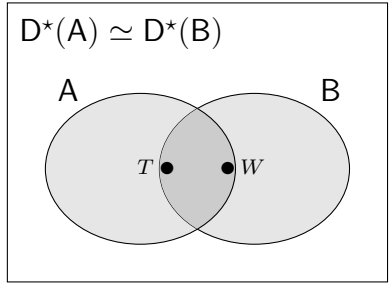

 In the last decade, there have been several attempts to recover a part
of this picture outside of the realm of finite-dimensional algebras.
Based on an existing theory of infinitely generated (co)tilting modules,
similar results were obtained in~\cite{St} in the case where $\A$ is a 
Grothendieck abelian category and $W \in \B$ is a big (i.~e., infinitely
generated) $n$\+cotilting module. The heart of a t\+structure associated with
a big $n$\+tilting module is considered in the recent paper~\cite{Ba}.
 A general discussion of (co)tilting objects in triangulated categories
can be found in the papers by Psaroudakis and Vit\'{o}ria~\cite{PV}
and Nicol\'{a}s, Saor\'{\i}n and Zvonareva~\cite{NSZ}, and of big $n$\+(co)tilting
objects in abelian categories, in~\cite[Section~6]{NSZ}.
% In particular, the introduction to~\cite{PV} emphasizes the importance
%of the abelian hearts of the (co)tilting t\+structures, as opposed to
%simply the module categories over the endomorphism rings of
%the (co)tilting objects, in the context of the (co)tilting derived
%equivalences.

 In this paper, we construct a one-to-one correspondence between
the two dual settings of complete, cocomplete abelian categories $\A$
with an injective cogenerator $W$ and a big $n$\+tilting object $T$,
and complete, cocomplete abelian categories $\B$ with a projective
generator $T$ and a big $n$\+cotilting object~$W$.
 The correspondence assigns to an abelian category $\A$ with
a tilting object $T$ the heart $\B={}^T\D^{\b,\le0}\cap{}^T\D^{\b,\ge0}$
of the tilting t\+structure on the derived category $\D^\b(\A)$, and
to an abelian category $\B$ with a cotilting object $W$ the heart
$\A={}^W\D^{\b,\le0}\cap {}^W\D^{\b,\ge0}$ of the cotilting t\+structure
on $\D^\b(\B)$.
 In addition, we proceed to construct triangulated equivalences
$\D^\st(\A)\simeq\D^\st(\B)$ between the (bounded or unbounded,
conventional or absolute) derived categories of the abelian
categories $\A$ and~$\B$. See again Figure~\ref{the-correspondence-fig}.

At this point, we have encountered a certain lack of symmetry in the available
literature on abelian categories, however. The abelian categories $\A$ in
the one-to-one correspondence above are often Grothendieck abelian categories and
as such well-understood. The other end of the correspondence involves
cocomplete abelian categories $\B$ with a projective generator, and their
structure does not seem to be so widely known. Here we show that in general, they
are described as the categories of modules over additive monads
on the category of sets.

If $\B$ is locally presentable (which always happens if $\A$ is such
in the context of the tilting-cotilting correspondence),
we turn out to land in the realm of models of infinitary algebraic
theories in the sense of \cite{Wr} (see also \cite[Introduction]{PR}).
% We also show that the categories of models of additive $\kappa$\+ary algebraic theories
Such categories $\B$ relate to module categories in a formally dual way as compared
to the Popescu--Gabriel Theorem for Grothendieck abelian categories.
 Indeed, Grothendieck abelian categories can be described as reflective full
subcategories in the categories of modules over associative rings
with exact reflection functors.
 On the other hand, the categories of models of additive $\kappa$\+ary
algebraic theories can be presented as reflective full subcategories in
the categories of modules over associative rings with
exact embedding functors.

 %Furthermore, we consider various restrictions that can be imposed
%on the abelian category $\A$ and discuss the related properties
%of the abelian category~$\B$.
 %In particular, whenever the category $\A$ is locally presentable,
%the category $\B$ is locally presentable, too.
 %Locally presentable abelian categories with a projective generator
%can be described as the \emph{categories of models of additive
%$\kappa$\+ary algebraic theories} for some cardinal numbers~$\kappa$~\cite{PR}.
 %Formally, $\B$ is up to equivalence the category of algebras/modules
%over an additive monad $\boT$ on the category of sets, where $\boT$ is
%the \emph{endomorphism monad} of the tilting object $T\in\A$ (as opposed
%to a mere endomorphism ring). When $\A$ is a Grothendieck abelian category, 
%$\boT$ has an additional property that the map 
%$\boT(X)\rarrow\prod_{x\in X}\boT(\{x\})$ is injective for all sets~$X$.

In a wide range of algebraic examples, however, the infinitary aspect is simply
captured by a topology on a ring and the models of the corresponding algebraic
theory are so-called \emph{contramodules} over that topological ring.
It turns out that, with the language of contramodules, the tilting
derived equivalences become even more concrete and transparent.

In many cases, the situation is even as follows: The category $\A$ is equivalent to the category
$\fA\discrl$ of discrete left modules (also known as a pretorsion class
in $\fA\modl$, \cite[\S VI.4]{Sten}) over a left topological ring~$\fA$.
The other side of the derived equivalence is occupied by the category
$\B$ of left contramodules over a right topological ring~$\fR$. A tilting
object in $\A$ is then in fact an $\fA$\+$\fR$\+bimodule ${_\fA T_\fR}$
which is discrete from either side, and the derived equivalences
are obtained by deriving the adjunction
\[ T\ocn_\fR{-}\:\fR\contra \;\rightleftarrows\; \fA\discrl \;\:\!\!\Hom_\fA(T,{-}). \]
Here, $T\ocn_\fR{-}$ is the so-called \emph{contratensor product} functor,
which is left adjoint to the contramodule valued Hom\+functor.

Upon developing basic Morita theory for contramodule categories, we obtain
that, up to this kind of Morita equivalence of $\fR$,
the exact forgetful functor $\fR\contra\rarrow\fR\modl$ is fully
faithful; so the category $\B=\fR\contra$ is a full subcategory in $\fR\modl$. The latter
adjunction is then simply a restriction of the usual adjunction
\[ T\ot_\fR{-}\:\fR\modl \;\rightleftarrows\; \fA\modl \;\:\!\!\Hom_\fA(T,{-}). \]
to corresponding full subcategories of the module categories.
%
 %In particular, when $\A=A\modl$ is the abelian category of modules
%over an associative ring $A$, we show that the tilting heart $\B$
%is equivalent to the abelian category $\fR\contra$ of contramodules,
%in the sense of~\cite[\S2.1]{Prev},
%over the \emph{topological} ring $\fR=\Hom_A(T,T)^\rop$ of endomorphisms
%of the tilting module~$T$.
 %The same conclusion holds in some other cases, e.~g., when $\A$
%is a full subcategory closed under infinite direct sums in $A\modl$,
%or when $\A$ is a locally finitely presentable Grothendieck category.
 %These assertions are deduced from a series of much more general
%results claiming that for any object $M\in\A$ the full additive
%subcategory $\Add(M)\subset\A$ consisting of all the direct summands
%of infinite direct sums of copies of $M$ in $\A$ is equivalent to
%the additive category of projective objects $\fR\contra_\proj$ in
%the abelian category of contramodules $\fR\contra$ over
%the topological ring $\fR=\Hom_\A(M,M)^\rop$, i.~e.,
%$\Add(M)\simeq\fR\contra_\proj$.
%
This phenomenon was already observed in~\cite{Ba-eq,BMT} in the special case
where $\A=A\modl$ for a discrete ring $A$, when
$T\in A\modl$ is a ``good'' $n$\+tilting module in the sense of~\cite{Ba-eq,BMT}
(every infinitely generated tilting module is Morita equivalent to one such).
When $n=1$, the $\fR$\+modules in the image of the forgetful functor $\fR\contra\rarrow\fR\modl$ were called
``costatic'' modules in the paper~\cite{GT} (specifically, in
the context of~\cite[Theorems~5.7 and~6.3]{GT}).

 %In the general case of a good $n$\+tilting module $T$, a triangulated
%equivalence between the (unbounded) derived category $\D(A\modl)$ and
%a certain full subcategory or Verdier quotient category of
%$\D(\fR\modl)$ was constructed in the paper~\cite{BMT}.
 %As a particular case of our results, we obtain a triangulated
%equivalence between the (bounded or unbounded) derived categories of
%two abelian categories $\D^\st(\A)\simeq\D^\st(\B)$, where $\A=A\modl$
%and $\B=\fR\contra\subset\fR\modl$.
 %This generalizes much further: for any $n$\+tilting object $T$ in
%a locally presentable abelian category $\A$ with an injective
%cogenerator $W$, replacing $T$ by the coproduct of its copies $T^{(Y)}$
%indexed over a large enough set $Y$ identifies the related abelian
%category $\B$ with a reflective full subcategory in the category
%of modules $B\modl$ over the associative ring
%$B=\Hom_\A(T^{(Y)},T^{(Y)})^\rop$, with an exact embedding functor
%$\B\rarrow B\modl$.

We conclude the paper with developing general tools to recognize the situation
where the tilting-cotilting correspondence is encoded in a discrete
tilting bimodule ${_\fA T_\fR}$ as above, even when the categories $\A$ and $\B$ are not \emph{a priori}
given in this form. We illustrate the tools on various classes of Gorenstein
or relative Gorenstein Grothendieck categories, which is the case closely
related to above-mentioned comodule-contramodule
correspondence~\cite[\S0.2--3 and Chapters~5--6]{Psemi}.
In particular, we characterize the situation where an injective cogenerator
of a locally Noetherian category is a tilting object, generalizing results
of Angeleri H\"ugel, Herbera and Trlifaj~\cite{AHT}.

\medskip

\textbf{Acknowledgement.}
 The authors are grateful to Luisa Fiorot, Jorge Vit\'oria, and
Jan Trlifaj for helpful discussions and communications.
 Leonid Positselski's research is supported by the Israel Science
Foundation grant~\#\,446/15 and research plan RVO:~67985840.
 Jan \v{S}\v{t}ov\'{\i}\v{c}ek's research is supported by
the Neuron Fund for Support of Science.

\Section{\texorpdfstring{Tilting $\mathrm{t}$-Structures}{Tilting t-Structures}}
%\Section{Tilting $\mathrm{t}$-Structures}
\label{tilting-t-structure-secn}

 Given an abelian category~$\C$, we denote by $\D^\b(\C)$ the derived
category of bounded complexes over~$\C$.
 Let $(\D^{\b,\le0}(\C),\.\D^{\b,\ge0}(\C))$ denote the standard
t\+structure on $\D^\b(\C)$, i.~e., $\D^{\b,\le0}(\C)\subset\D^\b(\C)$
is the full subcategory of complexes with the cohomology objects
concentrated in the nonpositive cohomological degrees and
$\D^{\b,\ge0}(\C)\subset\D^\b(\C)$ is the full subcategory of complexes
with the cohomology objects concentrated in the nonnegative
cohomological degrees. We use the notation $\tau_{\le 0}\:\D^\b(\C)\rarrow \D^{\b,\le0}(\C)$ and $\tau_{\ge 0}\:\D^\b(\C)\rarrow \D^{\b,\ge0}(\C)$ for the corresponding truncation functors, which are the right and the left adjoints to the inclusions $\D^{\b,\le0}(\C)\subset\D^\b(\C)$ and $\D^{\b,\ge0}(\C)\subset\D^\b(\C)$, respectively.
 A similar notation is used for the standard t\+structures on
the derived categories of bounded below, bounded above, and unbounded
complexes $\D^+(\C)$, \,$\D^-(\C)$, and $\D(\C)$
\,\cite[n$^\circ$\,1.3]{BBD}.

 The constructions of derived categories (and, more generally,
Verdier quotient categories) involve a well-known difficulty in
that the collection of all morphisms between a fixed pair of objects
in the derived category may turn out to be a proper class rather
than a set (see~\cite{CN} for an example).
 We will say that a certain derived category $\D^\st(\C)$
\emph{has Hom sets} if this complication does not arise, that is
morphisms between any two fixed objects in $\D^\st(\C)$ form a set.
 Even when this is not the case, one can still work with $\D^\st(\C)$
as a ``very large category'', but one has to be cautious.
 In any event, all the abelian categories in this paper will be
presumed or proved to have Hom sets.

 The following useful lemma can be found in~\cite[Lemma~10]{NSZ}.
We include a proof for the reader's convenience.

\begin{lem} \label{high-ext-long-complex}
 Let\/ $\C$ be an abelian category and $T\in\C$ an object of
projective dimension~$\le n$.
 Then for every complex $X^\bu\in\D^{\le -n-1}(\C)$ one has\/
$\Hom_{\D(\C)}(T,X^\bu)=0$.
\end{lem}

\begin{proof}
 Any morphism $T\rarrow X^\bu$ in $\D(\C)$ can be represented as
a fraction formed by a morphism $T\rarrow Y^\bu$ and
a quasi-isomorphism $X^\bu\rarrow Y^\bu$ of complexes over~$\C$.
 Applying the canonical truncation, we can assume that the terms
of the complex $Y^\bu$ are concentrated in the degrees~$\le0$.
 Let $Z$ denote the kernel of the differential $Y^{-n-1}\rarrow Y^{-n}$
and let $\sigma_{\ge -n-1}Y^\bu\subset Y$ be the silly truncation
of the complex~$Y^\bu$; so one has $H^i(\sigma_{\ge-n-1}Y^\bu)=0$ for
$i\ne-n-1$ and $H^{-n-1}(\sigma_{\ge-n-1}Y^\bu)=Z$.
 Then the morphism of complexes $T\rarrow Y^\bu$ factorizes as
$T\rarrow\sigma_{\ge -n-1}Y^\bu\rarrow Y^\bu$ and the morphism
$T\rarrow\sigma_{\ge -n-1}Y^\bu$ represents an extension class in
$\Ext^{n+1}_\C(T,Z)$, which vanishes by the assumption.
\end{proof}

 Let $\A$ be an abelian category with set-indexed products and
an injective cogenerator.
 It follows from Freyd's adjoint functor existence theorem that
any complete abelian category with a cogenerator is
cocomplete~\cite[Proposition~6.4]{Fa}.
 Furthermore, in any abelian category with enough injective
objects the coproducts are exact~\cite[Exercise~III.2]{Mit}.
 Thus set-indexed coproducts exist and are exact in~$\A$.

 It follows that both the unbounded derived category $\D(\A)$ and
the unbounded homotopy category $\Hot(\A)$ of complexes over $\A$
have set-indexed coproducts, and the canonical Verdier localization
functor $\Hot(\A)\rarrow\D(\A)$ preserves set-indexed coproducts.
 Therefore, so do the cohomology functors $H^i\:\D(\A)\rarrow\A$.
 Notice also that the bounded below derived category $\D^+(\A)$ has Hom
sets, because it is equivalent to the bounded below homotopy category
$\Hot^+(\A_\inj)$ of complexes of injective objects in~$\A$
(see e.~g.~\cite[Proposition I.4.7]{Hart66}).

 Let $n\ge0$ be an integer.
 Let us say that an object $T\in\A$ is (\emph{big}) \emph{$n$\+tilting}
if the following three conditions are satisfied:
\begin{enumerate}
\renewcommand{\theenumi}{\roman{enumi}}
\item the projective dimension of $T$ in $\A$ does not exceed~$n$,
that is $\Ext^i_\A(T,X)=0$ for all $i>n$ and all $X\in\A$;
\item $\Ext^i_\A(T,T^{(I)})=0$ for all $i>0$ and all sets $I$, where
$T^{(I)}$ denotes the coproduct of $I$~copies of $T$ in~$\A$;
\item every complex $X^\bu\in\D(\A)$ such that $\Hom_{\D(\A)}
(T,X^\bu[i])=0$ for all $i\in\Z$ is acyclic.
\end{enumerate}

 Denote by $\Add(T)\subset\A$ the full subcategory formed by
the direct summands of infinite coproducts of copies of the object
$T\in\A$.
 Condition~(ii) of the above definition allows us to show that
the triangulated subcategory of $\D^\b(\A)$ generated by $\Add(T)$
is equivalent to a full subcategory of $\Hot(\A)$.

\begin{lem} \label{tria-Add-T-via-complexes}
 Suppose that $T\in\A$ satisfies $\Ext^i_\A(T,T^{(I)})=0$ for all $i>0$
and all sets $I$. Then the composition
$$
 \Hot^\b(\Add(T))\stackrel{\subset}\rarrow\Hot^\b(\A)\rarrow\D^\b(\A)
$$
is fully faithful and the essential image is the triangulated
subcategory of $\D^\b(\A)$ generated by $\Add(T)$.
\end{lem}

\begin{proof}
 This is completely analogous to~\cite[Lemma 1.1]{Hap87}.
 Since $\A$ has exact coproducts, we have
$\Ext^i_\A(T^{(I)},T^{(J)})=0$ for all $i>0$ and all sets $I,J$.

 If now $X^\bu$, $Y^\bu \in \Hot^\b(\Add(T))$, then
$\Hom_{\Hot^\b(\A)}(X^\bu,Y^\bu) \cong\Hom_{\D^\b(\A)}(X^\bu,Y^\bu)$
is shown by induction on the sum of the widths of $X^\bu$ and $Y^\bu$.
 If $X^\bu$, $Y^\bu$ have widths one, then $X^\bu \cong X'[i]$ and
$Y^\bu \cong Y'[j]$ for some $X',Y'\in\Add(T)$ and $i,j \in \Z$.
 If $i\ne j$, then $\Hom_{\Hot^\b(\A)}(X^\bu,Y^\bu) = 0
= \Hom_{\D^\b(\A)}(X^\bu,Y^\bu)$ by the assumption on $T$,
and if $i=j$, then $\Hom_{\Hot^\b(\A)}(X^\bu,Y^\bu)
\cong \Hom_\A(X',Y') \cong \Hom_{\D^\b(\A)}(X^\bu,Y^\bu)$.

 If $X^\bu$ has width greater than one, we use the silly truncation
to find a triangle $X_1^\bu \rarrow X^\bu \rarrow X_2^\bu
\rarrow X_1^\bu[1]$ such that $X_1^\bu,X_2^\bu\in\Hot^\b(\Add(T))$
have widths smaller than the width of $X^\bu$.
 Applying $\Hom(-,Y^\bu)$ to this triangle and using the 5-lemma,
we deduce that $\Hom_{\Hot^\b(\A)}(X^\bu,Y^\bu) \cong
\Hom_{\D^\b(\A)}(X^\bu,Y^\bu)$.
 If the width of $Y^\bu$ is greater than one, we proceed similarly.
\end{proof}

The following theorem is the main result of this section.

\begin{thm} \label{tilting-t-structure-thm}
 Let $T\in\A$ be an $n$\+tilting object.
 Then the pair of full subcategories
\begin{align*}
 {}^T\D^{\le 0}&=\{\,X^\bu\in\D(\A)\mid \Hom_{\D(\A)}(T,X^\bu[i])=0
 \text{ for all $i>0$}\,\}, \\
 {}^T\D^{\ge 0}&=\{\,X^\bu\in\D(\A)\mid \Hom_{\D(\A)}(T,X^\bu[i])=0
 \text{ for all $i<0$}\,\}
\end{align*}
is a t\+structure on the unbounded derived category\/ $\D(\A)$.
\end{thm}

\begin{proof}
 Let $X^\bu\in\D(\A)$ be a complex.
 We start with constructing an approximation triangle
$$
 \tau_{\le0}^TX^\bu\lrarrow X^\bu\lrarrow\tau_{\ge1}^TX^\bu\lrarrow
 (\tau_{\le0}^TX^\bu)[1]
$$
with $\tau_{\le0}^TX^\bu\in{}^T\D^{\le0}$ and
$\tau_{\ge1}^TX^\bu\in{}^T\D^{\ge1}$.
It in fact follows automatically that these triangles, where $X^\bu$ runs over
the objects of $\D(\A)$, define
the truncation functors $\tau_{\le0}^T\:\D(\A)\rarrow{}^T\D^{\le0}$ and 
$\tau_{\ge0}^T\:\D(\A)\rarrow{}^T\D^{\ge0}$ as the corresponding
adjoints to the inclusions $\D(\A)\subset{}^T\D^{\le0}$ and
$\D(\A)\subset{}^T\D^{\ge0}$, respectively;
see~\cite[Proposition 1.3.3]{BBD}.

 Notice that for every complex $Y^\bu\in\D(\A)$ the collection of
all morphisms $T\rarrow Y^\bu$ in $\D(\A)$ is a set, because
$\Hom_{\D(\A)}(T,Y^\bu)=\Hom_{\D^+(\A)}(T,\tau_{\ge-n}Y^\bu)$ by
Lemma~\ref{high-ext-long-complex} and the category $\D^+(\A)$ has
Hom sets.
 Proceeding by induction, put $X^\bu_0=X^\bu$ and for every $i\ge0$
consider a distinguished triangle
$$
 T[i]^{(H_i)}\lrarrow X^\bu_i\lrarrow X^\bu_{i+1}\lrarrow T[i+1]^{(H_i)},
$$
where $H_i=\Hom_{\D(\A)}(T[i],X^\bu_i)$ and the first map in the triangle
is the canonical one.
 Using the condition~(ii), one easily proves by induction on~$i$ that
\begin{equation} \label{inductive-approximation-contstruction-hom}
 \Hom_{\D(\A)}(T[j],X^\bu_{i+1})=0 \qquad\text{for $j=0$,~\dots,~$i$}.
\end{equation}

 The complexes $X^\bu_i$ form an inductive system $X^\bu_0\rarrow
X^\bu_1\rarrow X^\bu_2 \rarrow\dotsb$.
 Since countable coproducts exist in $\D(\A)$, we can construct
a homotopy colimit of this system, defined as the third object in
a distinguished triangle
$$
 \coprod_{i=0}^\infty X^\bu_i\lrarrow\coprod_{i=0}^\infty X^\bu_i
 \lrarrow\operatorname{hocolim}\nolimits_{i\ge0} X^\bu_i
 \lrarrow\coprod_{i=0}^\infty X^\bu_i[1],
$$
where the first map in the triangle is $\mathrm{id}-\mathit{shift}\:
\coprod_i X^\bu_i\rarrow\coprod_i X^\bu_i$.
 Put
$$
 \tau_{\ge1}^TX^\bu=\operatorname{hocolim}\nolimits_{i\ge0} X^\bu_i.
$$

 For any $0\le k\le i$, denote by $S^\bu_{ki}$ a cone of the morphism
$X^\bu_k\rarrow X^\bu_i$.
 Then $S^\bu_{ki}$ is an iterated extension of the objects
$T[k+1]^{(H_k)}$,~\dots, $T[i]^{(H_{i-1})}$ in $\D(\A)$.
By~\cite[Lemma 1.7.1]{Nee01} we have
$$
 \operatorname{hocolim}\nolimits_{i\ge0}X^\bu_i=
 \operatorname{hocolim}\nolimits_{i\ge k}X^\bu_i.
$$
 By the octahedron axiom (or more precisely,
by~\cite[Proposition~1.1.11]{BBD}), a cone of the natural morphism
$X^\bu_k\rarrow\operatorname{hocolim}_{i\ge k} X^\bu_i$ is at the same
time a cone of a certain morphism $\coprod_{i\ge k}S^\bu_{ki}\rarrow
\coprod_{i\ge k}S^\bu_{ki}$.

 Since the cohomology functors $\D(\A)\rarrow\A$ preserve
countable coproducts, we can conclude that a cone of the morphism
$X_k^\bu\rarrow\tau_{\ge1}^TX^\bu$ belongs to $\D^{\le -k-1}(\A)$.
 Taking $j\ge0$ and $k>n+j$, by Lemma~\ref{high-ext-long-complex}
and~\eqref{inductive-approximation-contstruction-hom} we have
$$
 \Hom_{\D(\A)}(T[j],\tau_{\ge1}^TX^\bu)=\Hom_{\D(\A)}(T[j],X^\bu_k)=0.
$$
 Hence $\Hom_{\D(\A)}(T[j],\tau_{\ge1}^TX^\bu)=0$ for all $j\ge0$
and $\tau_{\ge1}^TX^\bu\in{}^T\D^{\ge1}$.

 On the other hand, a cocone $\tau_{\le0}^TX^\bu$ of the morphism
$X^\bu\rarrow\tau_{\ge1}^TX^\bu$ is at the same time a cocone of
a morphism $\coprod_{i\ge 0}S^\bu_{0,i}\rarrow\coprod_{i\ge 0}S^\bu_{0,i}$.
Using Lemma~\ref{tria-Add-T-via-complexes}, one shows that 
the object $S^\bu_{0,i}\in\D(\A)$ can be represented by a complex
of the form
$$
 T^{(H_{i-1})}\lrarrow\dotsb\lrarrow T^{(H_1)}\lrarrow T^{(H_0)}
$$
with terms in the abelian category~$\A$, sitting in the cohomological gradings
from $-i$ to~$-1$.
 Hence the object $\coprod_{i\ge0}S^\bu_{0,i}$ can be represented by
a complex sitting in the cohomological degrees~$\le-1$ whose terms
are copowers of the object~$T$.
 Applying again the condition~(ii) and
Lemma~\ref{high-ext-long-complex}, we conclude that
$$
 \Hom_{\D(\A)}(T[j],\coprod_{i\ge0}S^\bu_{0,i})=0
 \qquad\text{for $j\le0$},
$$
hence $\Hom_{\D(\A)}(T[j],\tau_{\le0}^TX^\bu)=0$ for $j\le -1$ and
$\tau_{\le0}^TX^\bu\in{}^T\D^{\le0}$.

 Now let $X^\bu$ be a complex belonging to ${}^T\D^{\le0}$, so
$\Hom_{\D(\A)}(T[j],X^\bu)=0$ for $j\le-1$.
 Then, similarly to~\eqref{inductive-approximation-contstruction-hom},
we have
$$
 \Hom_{\D(\A)}(T[j],X^\bu_{i+1})=0
 \qquad\text{for all $j\le i$, \,$j\in\Z$}.
$$
 Arguing as above, we can conclude that
$\Hom_{\D(\A)}(T[j],\tau_{\ge1}^TX^\bu)=0$ for all $j\in\Z$.
 By the condition~(iii), it follows that $\tau_{\ge1}^TX^\bu=0$ and
$X^\bu\cong\tau_{\le0}^TX^\bu$ in $\D(\A)$.

 Therefore, every object $X^\bu\in{}^T\D^{\le0}$ can be obtained from
the objects $T$, $T[1]$, $T[2]$, $T[3]$,~\dots\ using extensions and
coproducts in $\D(\A)$, or in other words, ${}^T\D^{\le0}\subset
\D(\A)$ is the suspended subcategory generated by~$T$.
 Hence $\Hom_{\D(\A)}(X^\bu,Y^\bu)=0$ for all $X^\bu\in{}^T\D^{\le0}$
and $Y^\bu\in{}^T\D^{\ge1}$.
 The theorem is proved.
\end{proof}

\begin{cor} \label{bounded-tilting-t-structures}
 Let $T\in\A$ be an $n$\+tilting object.  Then\par
\textup{(a)} the pair of full subcategories
\begin{align*}
 {}^T\D^{-,\le 0}&=\{\,X^\bu\in\D^-(\A)\mid \Hom_{\D^-(\A)}(T,X^\bu[i])=0
 \text{ for all $i>0$}\,\}, \\
 {}^T\D^{-,\ge 0}&=\{\,X^\bu\in\D^-(\A)\mid \Hom_{\D^-(\A)}(T,X^\bu[i])=0
 \text{ for all $i<0$}\,\}
\end{align*}
is a t\+structure on the bounded above derived category\/ $\D^-(\A)$;
\par
\textup{(b)} the pair of full subcategories
\begin{align*}
 {}^T\D^{+,\le 0}&=\{\,X^\bu\in\D^+(\A)\mid \Hom_{\D^+(\A)}(T,X^\bu[i])=0
 \text{ for all $i>0$}\,\}, \\
 {}^T\D^{+,\ge 0}&=\{\,X^\bu\in\D^+(\A)\mid \Hom_{\D^+(\A)}(T,X^\bu[i])=0
 \text{ for all $i<0$}\,\}
\end{align*}
is a t\+structure on the bounded below derived category\/ $\D^+(\A)$;
\par
\textup{(c)} the pair of full subcategories
\begin{align*}
 {}^T\D^{\b,\le 0}&=\{\,X^\bu\in\D^\b(\A)\mid \Hom_{\D^\b(\A)}(T,X^\bu[i])=0
 \text{ for all $i>0$}\,\}, \\
 {}^T\D^{\b,\ge 0}&=\{\,X^\bu\in\D^\b(\A)\mid \Hom_{\D^\b(\A)}(T,X^\bu[i])=0
 \text{ for all $i<0$}\,\}
\end{align*}
is a t\+structure on the bounded derived category\/ $\D^\b(\A)$.
\end{cor}

\begin{proof}
 By the definition of ${}^T\D^{\ge0}$, we have $\D^{\ge0}(\A)\subset
{}^T\D^{\ge0}$.
 Since $({}^T\D^{\le0},\.{}^T\D^{\ge0})$ is a t\+structure on $\D(\A)$
by Theorem~\ref{tilting-t-structure-thm} and
$(\D^{\le0}(\A),\.\D^{\ge0}(\A))$ is also a t\+structure on $\D(\A)$,
it follows that $\D^{\le0}(\A)\supset{}^T\D^{\le0}$.
 By Lemma~\ref{high-ext-long-complex}, we have $\D^{\le -n}(\A)\subset
{}^T\D^{\le0}$, and consequently $\D^{\ge-n}(\A)\supset{}^T\D^{\ge0}$.
 To sum up, the inclusions of full subcategories
\begin{alignat}{2}
\D^{\le-n}(\A)&\subset{}^T\D^{\le0}&&\subset
\D^{\le0}(\A), \label{d-unbounded-le-inclusions} \\
\D^{\ge0}(\A)&\subset{}^T\D^{\ge0}&&\subset
\D^{\ge-n}(\A) \label{d-unbounded-ge-inclusions}
\end{alignat}
hold in the derived category $\D(\A)$.

 Now the assertions~(a--c) are easily deduced.
 For example, let us prove~(c).
 Notice that we have ${}^T\D^{\b,\le0}=\D^\b(\A)\cap{}^T\D^{\le0}
\subset{}^T\D^{\le0}$ and ${}^T\D^{\b,\ge0}=\D^\b(\A)\cap{}^T\D^{\ge0}
\subset{}^T\D^{\ge0}$, so
$\Hom_{\D^\b(\A)}(X^\bu,Y^\bu)=\Hom_{\D(\A)}(X^\bu,Y^\bu)=0$ for all
$X^\bu\in{}^T\D^{\b,\le0}$ and $Y^\bu\in{}^T\D^{\b,\ge1}$.

 Furthermore, let $X^\bu\in\D^\b(\A)\subset\D(\A)$ be a bounded complex
and $\tau_{\le0}^TX^\bu$, \,$\tau_{\ge1}^TX^\bu$ be its truncations with
respect to the t\+structure $({}^T\D^{\le0},\.{}^T\D^{\ge0})$
on $\D(\A)$.
 Then $\tau^T_{\le0}X^\bu\in{}^T\D^{\le0}\subset\D^{\le0}(\A)\subset
\D^-(\A)$ by~\eqref{d-unbounded-le-inclusions}
and $\tau^T_{\ge1}X^\bu\in{}^T\D^{\ge1}\subset
\D^{\ge-n+1}(\A)\subset\D^+(\A)$
by~\eqref{d-unbounded-ge-inclusions}.
 It follows that the complexes $\tau_{\le0}^TX^\bu$, \,$\tau_{\ge1}^TX^\bu$
belong to $\D^\b(\A)$.
\end{proof}

 The following proposition provides the converse implication to
Corollary~\ref{bounded-tilting-t-structures}(c).
 In fact, it shows that an object $T \in \A$ of finite projective
dimension is tilting in our sense if and only if it is a tilting object
in $\D^\b(\A)$ in the sense of~\cite[Definition 4.1]{PV}.

\begin{prop} \label{bounded-tilting-t-structure-implies-generation}
 Let $T\in\A$ be an object satisfying the conditions~\textup{(i--ii)}.
 Suppose that the pair of full subcategories $({}^T\D^{\b,\le0},\.
{}^T\D^{\b,\ge0})$ is a t\+structure on the bounded derived category\/
$\D^\b(\A)$.
 Then the condition~\textup{(iii)} is satisfied.
\end{prop}

\begin{proof}
 Consider an unbounded complex $X^\bu\in\D(\A)$ such that
$\Hom_{\D(\A)}(T,X^\bu[i])=0$ for $0\le i\le n$.
 We will show that $H^0(X^\bu)=0$.

 If we denote by $\tau_{\le 0}$, \,$\tau_{\ge0}$ the truncation functors
in the standard t\+structure on $\D(\A)$,
 we have by Lemma~\ref{high-ext-long-complex} that
$\Hom_{\D(\A)}(T,(\tau_{\le-n-1}X^\bu)[i])=0$. Therefore
$$
 \Hom_{\D(\A)}(T,(\tau_{\ge-n}X^\bu)[i])=\Hom_{\D(\A)}(T,X^\bu[i])
$$
for $i\ge0$, hence $\Hom_{\D(\A)}(T,(\tau_{\ge-n}X^\bu)[i])=0$ for
$0\le i\le n$. 
 Similarly, we have
$$
 \Hom_{\D(\A)}(T,(\tau_{\le n}\tau_{\ge-n}X^\bu)[i])=
 \Hom_{\D(\A)}(T,(\tau_{\ge-n}X^\bu)[i])=0
$$
for $0\le i\le n$.
 Clearly, $H^0(X^\bu)=H^0(\tau_{\le n}\tau_{\ge-n}X^\bu)$, so it
suffices to show that $H^0(\tau_{\le n}\tau_{\ge-n}X^\bu)=0$.

 This reduces the question to the case of a bounded complex
$\tau_{\le n}\tau_{\ge-n}X^\bu$.
 Thus it will be sufficient if we assume that
$X^\bu\in\D^\b(\A)$ is a bounded complex satisfying
$\Hom_{\D^\b(\A)}(T,X^\bu[i])=0$ for $0\le i\le n$, and
deduce that $H^0(X^\bu)=0$.

 By the definition of ${}^T\D^{\b,\ge0}$, we have $\D^{\b,\ge0}(\A)\subset
{}^T\D^{\b,\ge0}$.
 By the condition~(i), we have $\D^{\b,\le-n}(\A)\subset
{}^T\D^{\b,\le0}$.
 Since $({}^T\D^{\b,\le0},\.{}^T\D^{\b,\ge0})$ is presumed to be
a t\+structure on $\D^\b(\A)$, we come to the inclusions of
full subcategories
\begin{alignat}{2}
\D^{\b,\le-n}(\A)&\subset{}^T\D^{\b,\le0}&&\subset
\D^{\b,\le0}(\A), \label{d-le-inclusions} \\
\D^{\b,\ge0}(\A)&\subset{}^T\D^{\b,\ge0}&&\subset
\D^{\b,\ge-n}(\A) \label{d-ge-inclusions}
\end{alignat}
in the bounded derived category $\D^\b(\A)$.

 Let $\tau^T_{\le0}$, \,$\tau^T_{\ge0}$ denote the truncation functors
with respect to the t\+structure $({}^T\D^{\b,\le0},\.{}^T\D^{\b,\ge0})$
on $\D^\b(\A)$.
 Then for any complex $X^\bu\in\D^\b(\A)$ the natural maps
$$
 \Hom_{\D^\b(\A)}(T,(\tau^T_{\le0}X^\bu)[i])\lrarrow
 \Hom_{\D^\b(\A)}(T,X^\bu[i])
$$
are isomorphisms for all $i\le0$, while the natural maps
$$
 \Hom_{\D^\b(\A)}(T,X^\bu[i])\lrarrow
 \Hom_{\D^\b(\A)}(T,(\tau^T_{\ge0}X^\bu)[i])
$$
are isomorphisms for all $i\ge0$.

 Let $X^\bu\in\D^\b(\A)$ be a complex such that $\Hom_{\D^\b(\A)}
(T,X^\bu[i])=0$ for $0\le i\le n$.
 Then, on the one hand,
$$
 \Hom_{\D^\b(\A)}(T,(\tau^T_{\ge0}X^\bu)[i])=\Hom_{\D^\b(\A)}(T,X^\bu[i])=0
 \qquad\text{for $0\le i\le n$}
$$
and, on the other hand,
$$
 \Hom_{\D^\b(\A)}(T,(\tau^T_{\ge0}X^\bu)[i])=0
 \qquad\text{for $i<0$},
$$
since $\tau^T_{\ge0}X^\bu\in{}^T\D^{\b,\ge0}$.
 Therefore, $\Hom_{\D^\b(\A)}(T,(\tau^T_{\ge0}X^\bu)[i])=0$ for all $i\le n$,
and it follows that
$$
 \tau^T_{\ge0}X^\bu\in{}^T\D^{\b,\ge n+1}\subset\D^{\b,\ge1}(\A).
$$
 Besides, $\tau^T_{\le-1}X^\bu\in{}^T\D^{\b,\le-1}\subset\D^{\b,\le-1}(\A)$.
 Thus $H^0(\tau^T_{\le-1}X^\bu)=0$ and $H^0(\tau^T_{\ge0}X^\bu)=0$,
implying that $H^0(X^\bu)=0$.
\end{proof}

 The t\+structures on the derived categories $\D(\A)$, \,$\D^+(\A)$,
\,$\D^-(\A)$, and $\D^\b(\A)$ provided by
Theorem~\ref{tilting-t-structure-thm}
and Corollary~\ref{bounded-tilting-t-structures} are called
the \emph{tilting t\+structures} associated with an $n$\+tilting
object $T\in\A$.

 We denote by $\B={}^T\D^{\b,\le0}\cap{}^T\D^{\b,\ge0}$ the heart of
the related tilting t\+structure on $\D^\b(\A)$.
 According to~(\ref{d-unbounded-le-inclusions}\+%
\ref{d-unbounded-ge-inclusions}), \,$\B$ coincides with the heart
${}^T\D^{\le0}\cap{}^T\D^{\ge0}$ of the tilting t\+structure on $\D(\A)$.
 The category $\B$ has Hom sets, since $\B\subset\D^\b(\A)\subset
\D^+(\A)$.
 By the definition, $\B$ is an abelian category.

 The following proposition shows that the category $\B$ has some
properties dual to those of the category~$\A$.

\begin{prop} \label{tilting-heart}
 The object $T\in\A\subset\D(\A)$ belongs to\/ $\B$ and is a projective
generator of the abelian category\/~$\B$.
 Coproducts of copies of $T$ in\/ $\B$ coincide with such coproducts
in\/ $\D(\A)$ and in\/~$\A$.
 The projective objects of\/ $\B$ are precisely the direct summands of
these coproducts.
 Set-indexed coproducts of arbitrary objects exist in\/~$\B$.
\end{prop}

\begin{proof}
 The coproduct $T^{(I)}$ of $I$ copies of $T$ in $\A$ is also
the coproduct of $I$ copies of $T$ in $\D(\A)$.
 The object $T^{(I)}$ belongs to ${}^T\D^{\ge0}$ by definition and
to ${}^T\D^{\le0}$ by the condition~(ii).
 Thus $T^{(I)}\in\B\subset\D(\A)$.
 Being the coproduct of $I$ copies of $T$ in $\D(\A)$, this object is
also the coproduct of $I$ copies of $T$ in $\B\subset\D^\b(\A)
\subset\D(\A)$.

 The object $T$ is projective in $\B$, because $\Ext^1_\B(T,X)=
\Hom_{\D^\b(\A)}(T,X[1])=0$ for every $X\in\B$.
 The object $T$ is a projective generator of $\B$, since
$\Hom_\B(T,X)=0$ for some $X\in\B$ implies $\Hom_{\D^\b(\A)}(T,X[i])=0$
for all $i\in\Z$, so $X=0$ by the condition~(iii).
 It follows that the projective objects of $\B$ are precisely
the direct summands of the objects~$T^{(I)}$.

 Finally, we have shown that set-indexed coproducts of projective
objects exist in~$\B$, and that there are enough projective objects.
 Set-indexed coproducts of arbitrary objects can be constructed in
terms of the coproducts of projective objects.
 More explicitly, the category $\B$ is equivalent to the additive
quotient of the category of morphisms in $\Add(T)$ modulo
an ideal which is closed under coproducts of morphisms;
see~\cite[Proposition IV.1.2]{ARS96}.
\end{proof}

\Section{Tilting Classes and Tilting Cotorsion Pairs}
\label{tilting-cotorsion-pair-secn}

 The aim of this section is to work out elementary homological algebra related to tilting and cotilting objects. This in particular allows us in Theorem~\ref{tilting-cotorsion-pair-thm} to characterize tilting objects $T\in\A$ directly in the category $\A$, without using the unbounded derived category as in condition~\textup{(iii)} in the previous section. The arguments are of purely homological nature in the spirit of~\cite{AR91} and so they easily dualize to the setting of cotilting modules.
 Our exposition is also informed by that in~\cite[Section~6]{NSZ}.

 Let $\A$ be a complete, cocomplete abelian category with an injective
cogenerator $W$, and let $T\in\A$ be an object of projective dimension
not exceeding~$n$.
 Denote by $\E$ the following full subcategory in~$\A$:
\[
 \E = \{ E\in\A \mid\Ext^i_\A(T,E)=0 \textrm{ for all } i>0 \}.
\]

\begin{lem} \label{tilting-class-coresolving-finite-dim}
\textup{(a)} The full subcategory\/ $\E\subset\A$ is closed under
direct summands, extensions, and cokernels of monomorphisms in
the abelian category\/~$\A$.
 In addition, $\E\subset\A$ contains the full subcategory of injective
objects\/ $\A_\inj\subset\A$. \par
\textup{(b)} Each object $X\in\A$ admits an exact sequence in\/ $\A$
of the form
\[
 0 \lrarrow X \lrarrow E^0 \lrarrow E^1 \lrarrow \dotsb
 \lrarrow E^n \lrarrow 0,
\]
where $E^0$, $E^1$,~\dots, $E^n \in \E$.
\end{lem}

\begin{proof}
 Part~(a): the closedness under direct summands is obvious, as is
the assertion that $\A_\inj\subset\E$.
 Furthermore, given a short exact sequence $0\rarrow E\rarrow F
\rarrow G\rarrow0$ in $\A$ with $E\in\E$, one immediately concludes
from the corresponding long exact sequence of $\Ext_\A^*(T,{-})$ that
$F\in\E$ if and only if $G\in\E$.

 Part~(b): given an object $X\in\A$, we use the fact that $\A$ has
an injective cogenerator $W$ and consider a short exact sequence
$0 \rarrow X \rarrow W^I \rarrow X' \rarrow 0$.
 Applying the same to $X'$, etc., we obtain an exact sequence
$0\rarrow X\rarrow E^0\rarrow\dotsb\rarrow E^n\rarrow0$, where
$E^1$,~\dots, $E^{n-1}$ are direct powers of~$W$.

 Denoting by $X^k$ the image of the morphism $E^{k-1}\rarrow E^k$,
we have short exact sequences $0\rarrow X^k\rarrow E^k\rarrow X^{k+1}
\rarrow 0$, \ $0\le k\le n-1$, where $X^0=X$ and $X^n=E^n$.
 Then $\Ext^i_\A(T,E^n)=\Ext_\A^{i+1}(T,X^{n-1})=\dotsb=
\Ext^{i+n}_\A(T,X)=0$ for $i>0$.
 Hence $E^n\in\E$.
\end{proof}

 The assertion of Lemma~\ref{tilting-class-coresolving-finite-dim}(a)
can be rephrased by saying that the full subcategory $\E\subset\A$
is \emph{coresolving} in the sense of~\cite[Section~2]{St}.
 Then Lemma~\ref{tilting-class-coresolving-finite-dim}(b) says that
the \emph{coresolution dimension} of $\A$ with respect to $\E$
is bounded by the number~$n$.

 For each integer $m\ge0$, we denote by $\L_m$ the full subcategory
of all objects $L\in\A$ for which there exists an exact sequence in $\A$
of the form
$$
 0\lrarrow L\lrarrow T^0\lrarrow T^1\lrarrow\dotsb\lrarrow T^m\lrarrow0,
$$
where $T^0$, $T^1$,~\dots, $T^m\in\Add(T)$.
 Clearly, one has $\Add(T)=\L_0\subset\L_1\subset\L_2\subset\dotsb
\subset\A$.

\begin{lem} \label{left-class-tilting-cotorsion-pair}
 Assume that the object $T\in\A$ satisfies
the conditions~\textup{(i--ii)}.
 Then \par
\textup{(a)} for any objects $L\in\L_m$ and $E\in\E$, one has\/
$\Ext_\A^i(L,E)=0$ for all $i>0$; \par
\textup{(b)} the intersection\/ $\L_m\cap\E$ coincides with the full
subcategory\/ $\Add(T)\subset\A$; \par
\textup{(c)} for each integer $m\ge n$, one has\/ $\L_m=\L_{m+1}$.
\end{lem}

\begin{proof}
 Part~(a): by the definition, for any object $L\in\L_m$ there exists
a short exact sequence $0\rarrow L\rarrow T'\rarrow M\rarrow 0$ in $\A$
with $T'\in\Add(T)$ and $M\in\L_{m-1}$.
 Since coproducts are exact in $\A$, we have
$\Ext^i_\A(T',E)=0$ for all $E\in\E$ and $i>0$.
 Arguing by induction on~$m$, we can assume that $\Ext_\A^i(M,E)=0$
for $i>0$.
 Applying the long exact sequence of $\Ext_\A^*({-},E)$, we obtain
the desired Ext vanishing.

 Part~(b): by the definition we have $\Add(T)\subset\L_m$, and
 by the condition~(ii) we have \,$\Add(T)\subset\E$.
 Conversely, given an object $K\in\L_m$, there exists a short exact
sequence $0\rarrow K\rarrow T'\rarrow M\rarrow0$ with
$T'\in\Add(T)$ and $M\in\L_{m-1}$.
 Now if $K\in\E$, then by part~(a) we have
$\Ext_\A^1(M,K)=0$, hence $K$ is a direct summand of~$T'$.

 Part~(c): let $L\in\L_{m+1}$, where $m\ge n$.
 Then there exists an exact sequence
$0\rarrow L\rarrow T^0\rarrow\dotsb\rarrow T^m\rarrow T^{m+1} \rarrow0$
in $\A$ with $T^k\in\Add(T)$.
 Denoting the image of the morphism $T^{k-1}\rarrow T^k$ by $M^k$ and
using the condition~(ii), we compute that
$\Ext^1_\A(T,M^m)=\Ext^2_\A(T,M^{m-1})=\dotsb=\Ext^{m+1}_\A(T,L)$.
 The latter Ext group vanishes by the condition~(i), so we have
$\Ext^1_\A(T,M^m)=0$.
 It follows that $\Ext^1_\A(T^{m+1},M^m)=0$, hence the short exact
sequence $0\rarrow M^m\rarrow T^m\rarrow T^{m+1}\rarrow0$ splits and
$M^m\in\Add(T)$.
 Now the exact sequence $0\rarrow L\rarrow T^0\rarrow\dotsb\rarrow
T^{m-1}\rarrow M^m\rarrow 0$ shows that $L\in\L_m$.
\end{proof}

 Assuming that the object $T\in\A$ satisfies the conditions~(i--ii),
we will denote the full subcategory $\L_n=\L_{n+1}=\L_{n+2}=\dotsb$
simply by $\L\subset\A$.

 Now let us discuss the situation when the object $T\in\A$ is big
$n$\+tilting.
 In this case, using the notation $\B={}^T\D^{\b,\le0}\cap
{}^T\D^{\b,\ge0}\subset\D^\b(\A)$ for the heart of the tilting
t\+structure on $\D^\b(\A)$, the full subcategory $\E\subset\A$ can be
described as the intersection $\E=\A\cap\B$ of the hearts of the standard
and tilting t\+structures on $\D^\b(\A)$.
 Hence the same category $\E$ can be also considered as a full
subcategory in the abelian category~$\B$.

 It turns out that $\E$ as a subcategory of $\B$ satisfies dual properties to those which $\E$ has as a subcategory of $\A$.
 In the terminology of~\cite[Section~2]{St}, the next lemma says that
the full subcategory $\E\subset\B$ is \emph{resolving} and
the \emph{resolution dimension} of (the objects of) $\B$ with respect
to $\E$ does not exceed~$n$.
 
\begin{lem} \label{cotilting-class-resolving-finite-dim}
\textup{(a)} The full subcategory\/ $\E\subset\B$ is closed under
direct summands, extensions, and kernels of epimorphisms in
the abelian category\/~$\B$.
 In addition, $\E\subset\B$ contains the full subcategory of projective
objects\/ $\B_\proj\subset\B$. \par
\textup{(b)} Each object $Y\in\B$ admits an exact sequence in\/ $\B$
of the form
\[
 0 \lrarrow E_n \lrarrow \dotsb \lrarrow E_1\lrarrow E_0
 \lrarrow Y \lrarrow 0,
\]
where $E_0$, $E_1$,~\dots, $E_n \in \E$.
\end{lem}

\begin{proof}
 Part~(a): the full subcategory $\E$ is closed under direct summands and extensions in $\B$, since both the full subcategories $\A$ and $\B\subset
\D^\b(\A)$ are closed under direct summands and extensions (in
the triangulated category sense) in $\D^\b(\A)$.
 According to Proposition~\ref{tilting-heart}, we have $\B_\proj=
\Add(T)\subset\A\cap\B$.

 To show that $\E$ is closed under kernels of epimorphisms in $\B$,
consider a short exact sequence $0\rarrow G\rarrow F\rarrow E\rarrow0$
in $\B$ with $E$ and $F\in\E$.
 Then there is a distinguished triangle $G\rarrow F\rarrow E\rarrow
G[1]$ in $\D^\b(\A)$.
 Now we have $E,F\in\E\subset\A$, hence $G\in\D^{\b,\ge0}(\A)$.
 On the other hand, $G\in\B\subset{}^T\D^{\b,\le0}\subset
\D^{\b,\le0}(\A)$ according to~\eqref{d-le-inclusions}.
 Hence $G\in\A\cap\B$.

 For the proof of part~(b) we use essentially the same argument
as in~\cite[Proposition~5.20]{St}.
 We denote for this proof for each $j\in\{0,1,\dotsc,n\}$ by $\E_j$
the class $\E_j=\B\cap\D^{\b,\ge-j}(\A)\subset\D^\b(\A)$.
 By~(\ref{d-le-inclusions}\+\ref{d-ge-inclusions}), we have
$\E=\E_0\subset\E_1\subset\dotsb\subset\E_n=\B$.

 For any $j\in\{1,\dotsc,n\}$ and $Y\in\E_j$, by
Proposition~\ref{tilting-heart} there is a short exact
sequence $0\rarrow Y'\rarrow T^{(I)}\rarrow Y\rarrow0$ in $\B$;
hence a related distinguished triangle $Y'\rarrow T^{(I)}\rarrow Y
\rarrow Y'[1]$ in $\D^\b(\A)$.
 Then $\Hom_{\D^\b(\A)}(Y',W[i-1])\cong
\Hom_{\D^\b(\A)}(Y,W[i])$ for all $i>1$, and therefore
$Y'\in\E_{j-1}$.

 Starting from an arbitrary $Y\in\B$, we construct by induction
the desired exact sequence, even with
$E_0$,~\dots, $E_{n-1}$ being copowers of~$T$.
\end{proof}

 The following definitions are standard; see for instance~\cite{SaSt11}, \cite[\S5]{St-model} or \cite[\S\S3 and 4]{PR} and the references there (the concept of a cotorsion pair goes back to~\cite{Sal77} and have been extensively used in representation and module theory~\cite{KrSo,GTbook}). 
 A pair of full subcategories $\K$ and $\F\subset\A$ in an abelian
category $\A$ is called a \emph{cotorsion pair} if $\K$ consists
precisely of all the objects $K\in\A$ such that $\Ext_\A^1(K,F)=0$
for all $F\in\F$, and $\F$ consists precisely of all the objects
$F\in\A$ such that $\Ext^1_\A(K,F)=0$ for all $K\in\K$.

 A cotorsion pair $(\K,\F)$ in $\A$ is called \emph{hereditary} if
$\Ext^i_\A(K,F)=0$ for all $K\in\K$, $F\in\F$, and $i\ge1$.
In a hereditary cotorsion pair, the class $\K$ is also closed under
the kernels of epimorphisms, and the class $\F$ is closed under
the cokernels of monomorphisms. Under mild conditions, these closure properties in fact characterize hereditary cotorsion pairs, see~\cite[Lemma 6.17]{St-model} or~\cite[Lemma 4.25]{SaSt11}.

 A cotorsion pair $(\K,\F)$ is called \emph{complete} if for every
object $X\in\A$ there exist short exact sequences in $\A$
(called the \emph{approximation sequences}) of the form
\begin{gather}
0\lrarrow F'\lrarrow K\lrarrow X\lrarrow 0 
\label{special-precover-sequence}\\
0\lrarrow X\lrarrow F\lrarrow K'\lrarrow 0
\label{special-preenvelope-sequence}
\end{gather}
with $K$, $K'\in\K$ and $F$, $F'\in\F$.

 For any cotorsion pair $(\K,\F)$ in $\A$, the full subcategory
$\K\subset\A$ is closed under coproducts, direct summands, and
extensions, while the full subcategory $\F\subset\A$ is closed under
products, direct summands, and extensions.
 This is true even if products or coproducts are not exact, see~\cite[Proposition 8.1]{CF07}.
 The following partial converse assertion to these observations holds:
if $(\K,\F)\subset\A$ is a pair of full subcategories such that
$\Ext_\A^1(K,F)=0$ for all $K\in\K$ and $F\in\F$,
the approximation sequences~(\ref{special-precover-sequence}\+%
\ref{special-preenvelope-sequence}) exist for all objects $X\in\A$,
and the full subcategories $\K$, $\F\subset\A$ are closed under
direct summands, then $(\K,\F)$ is a (complete) cotorsion pair in~$\A$.

 The following theorem, which provides the promised characterization of tilting objects without using the unbounded derived category, is the main result of this section. It among other things shows that the tilting objects in our sense agree with the definition of Nicol\'{a}s, Saor\'{\i}n and Zvonareva~\cite[Section 6]{NSZ}.

\begin{thm} \label{tilting-cotorsion-pair-thm}
 Let\/ $\A$ be an abelian category with set-indexed products and
an injective cogenerator, and let $T\in\A$ be an object satisfying
the conditions~\textup{(i--ii)} of Section~\ref{tilting-t-structure-secn}.
 Then the following three conditions are equivalent:

\begin{enumerate}
\item[(1)] the object $T\in\A$ satisfies
the condition~\textup{(iii)} as well;
\item[(2)] for every object
$E\in\E=\bigcap_{i>0}\ker\Ext^i_\A(T,-)$ there exists an object
$T'\in\Add(T)$ together with an epimorphism $T'\rarrow E$ in
the category\/~$\A$;
\item[(3)] for every object $X\in\A$ there exists an object
$L\in\L$ (i.~e., $L$ has a finite $\Add(T)$-coresolution)
together with an epimorphism $L\rarrow X$ in the category\/~$\A$.
\end{enumerate}
 If one of the conditions~\textup{(1--3)} is satisfied, then the pair
of full subcategories $(\L,\E)$ is a hereditary complete cotorsion pair
in the abelian category\/~$\A$.
\end{thm}

\begin{proof}
 (1)~$\Longrightarrow$~(2): according to
Proposition~\ref{tilting-heart}, for every object $Y\in\B$ there exists
an object $T'\in\B_\proj=\Add(T)$ together with an epimorphism
$T'\rarrow Y$ in the category~$\B$.
 In particular, this applies to the object $Y=E\in\E$.
 By Lemma~\ref{cotilting-class-resolving-finite-dim}(a),
the kernel $E'$ of the epimorphism $T'\rarrow E$ belongs
to $\E\subset\B$.
 Now the short exact sequence $0\rarrow E'\rarrow T'\rarrow E
\rarrow 0$ in $\B$ corresponds to a distinguished triangle
$E'\rarrow T'\rarrow E\rarrow E'[1]$ in $\D^\b(\A)$.
 Since $E'\in\E$ and $E\in\E\subset\A$, it follows that
the short sequence $0\rarrow E'\rarrow T'\rarrow E\rarrow0$
is exact in~$\A$.

 (2)~$\Longrightarrow$~(3): first of all, let us show that
the epimorphism $T'\rarrow E$ in~(2) can be chosen in such
a way that its kernel belongs to~$\E$.
 Indeed, set $T''=T^{(I)}$ to be the coproduct of copies of $T$ indexed
by the set $I=\Hom_\A(T,E)$ of all morphisms $T\rarrow E$, and
let $T''\rarrow\E$ be the natural morphism.
 The existence of an epimorphism $T'\rarrow E$ with $T'\in\Add(T)$
implies that the morphism $T''\rarrow E$ is an epimorphism, and
surjectivity of the map $\Hom_\A(T,T'')\rarrow\Hom_\A(T,E)$ together with~\textup{(ii)} implies
that the kernel $E''$ of the morphism $T''\rarrow E$ satisfies
$\Ext^i_\A(T,E'')=0$ for all $i>0$.

 Now we apply the dual version of~\cite[Theorem~1.1]{AB}, using
the facts that the full subcategory $\E\subset\A$ is closed under
extensions (Lemma~\ref{tilting-class-coresolving-finite-dim}(a))
and every object of $\A$ has a coresolution of length at most~$n$
by objects from~$\E$
(Lemma~\ref{tilting-class-coresolving-finite-dim}(b)).
 Proceeding by induction on the minimal length of such a coresolution
for a given object $X\in\A$, one constructs for it the approximation
exact sequences $0\rarrow X\rarrow E\rarrow L'\rarrow0$ and
$0\rarrow E'\rarrow L\rarrow X\rarrow0$ with $E$, $E'\in\E$ and
$L$, $L'\in\L$ (\ref{special-precover-sequence}\+%
\ref{special-preenvelope-sequence}).
 In particular, this proves~(3).
 (Cf.~\cite[proof of Theorem~3]{NSZ}.)

 (3)~$\Longrightarrow$~(1): let $X^\bu\in\D(\A)$ be a complex
such that $\Hom_{\D(\A)}(T,X^\bu[i])=0$ for all $i\in\Z$.
 Since infinite coproducts are exact in $\A$, it follows that
$\Hom_{\D(\A)}(T',X^\bu[*])=0$ for all $T'\in\Add(T)$.
 Consequently, one has $\Hom_{\D(\A)}(L,X^\bu[i])=0$ for
all $L\in\L$.

 Now let $Z^i\in\A$ denote the kernel of the differential $X^i\rarrow
X^{i+1}$.
 According to~(3), there exists an object $L\in\L$ together with
an epimorphism $L\rarrow Z^i$.
 Since the related morphism of complexes $L\rarrow X^\bu[i]$
vanishes in $\D(\A)$, it must induce a zero morphism on
the cohomology objects.
 But the induced morphism $L\rarrow H^i(X^\bu)$ is an epimorphism
by construction, so the object $H^i(X^\bu)$ has to vanish.
 
 In addition to the equivalence of the three conditions~(1--3), we have
already shown that the approximation
sequences~(\ref{special-precover-sequence}\+%
\ref{special-preenvelope-sequence}) exist under the assumption of
these conditions.
 In view of Lemmas~\ref{left-class-tilting-cotorsion-pair}(a)
and~\ref{tilting-class-coresolving-finite-dim}(a), in order to prove
that $(\L,\E)$ is a hereditary complete cotorsion pair in $\A$,
it only remains to check that the full subcategory $\L\subset\A$
is closed under direct summands.
 The following Lemma~\ref{Add(T)-E-preenvelope} implies that.
\end{proof}

In order to state and prove Lemma~\ref{Add(T)-E-preenvelope}, we need some terminology from~\cite[\S5.1]{GTbook} (the same concepts were studied in \cite{AR91}, even in the context of tilting theory, but using different terminology).
 A morphism $f\:X\rarrow F$ from an object $X\in\A$ to an object
$F\in\E$ is said to be an \emph{$\E$\+preenvelope} if the map of
abelian groups $\Hom_\A(f,E)\:\Hom_\A(F,E)\rarrow\Hom_\A(X,E)$ is
surjective for all $E\in\E$.
 Since any object of $\A$ is a subobject of an object of
$\A_\inj\subset\E$, any $\E$\+preenvelope in $\A$ is a monomorphism.
 A monomorphism~$f$ is said to be a \emph{special\/ $\E$\+preenvelope}
if its cokernel $M$ has the property that $\Ext^1_\A(M,E)=0$ for
all $E\in\E$.
 It is clear from the long exact sequence of $\Ext_\A^*({-},E)$ that
any special $\E$\+preenvelope is an $\E$\+preenvelope.
 Conversely, any $\E$\+preenvelope of the form $f\:X\rarrow T'$ with
the object $T'$ belonging to $\Add(T)$ is special, because
$\Ext^1_\A(T',E)=0$ for all $E\in\E$.

\begin{lem} \label{Add(T)-E-preenvelope}
 Let $T\in\A$ be a big $n$\+tilting object. Then the following conditions are
equivalent for $L\in\A$:
 
\begin{enumerate}
\item[(1)] $L \in \L$;
\item[(2)] $L$ is a direct summand of an object from $\L$;
\item[(3)] there exists an\/ $\E$\+preenvelope
 $L\rarrow T'$ with $T'\in\Add(T)$.
\end{enumerate}
\end{lem}

\begin{proof}
(1)~$\Longrightarrow$~(2): This is obvious.

(2)~$\Longrightarrow$~(3): Suppose that $L$ is a summand of $N\in\L$.
  Let $0\rarrow N\rarrow T^0\rarrow\dotsb \rarrow T^n\rarrow0$ be
an exact sequence in $\A$ with $T^k\in\Add(T)$.
 Set $T'=T^0$, and denote by $M$ the cokernel of the morphism
$N\rarrow T'$.
 Then $M\in\L$, hence by
Lemma~\ref{left-class-tilting-cotorsion-pair}(a) we have
$\Ext^1_\A(M,E)=0$ for each $E\in\E$, and therefore $N\rarrow T'$ is an (even special)
$\E$\+preenvelope.
The composition $L\rarrow N\rarrow T'$ is clearly an $\E$\+preenvelope too.

(3)~$\Longrightarrow$~(1): 
If $f\:L\rarrow T'$ is an $\E$\+preenvelope with
$T'\in\Add(T)$ and $\coker(f)=M$, then, as mentioned above, $\Ext^1_\A(M,E)=0$ for all $E\in\E$.
 From an approximation sequence $0\rarrow E\rarrow L'\rarrow M
\rarrow 0$ \eqref{special-precover-sequence} with $E\in\E$ and
$L'\in\L$, whose existence was established in the proof of (2)~$\Longrightarrow$~(3) in Theorem~\ref{tilting-cotorsion-pair-thm}, we see that $M$ is a direct summand of~$L'$.
Hence $M$ has an $\E$\+preenvelope $M\rarrow T''$ with $T''\in\Add(T)$
by the previous implication.

 Setting $T^0=T'$, \,$T^1=T''$, and proceeding further in this way,
we construct an exact sequence
$0\rarrow L\rarrow T^0\rarrow T^1\rarrow\dotsb\rarrow T^{n-1}
\rarrow K\rarrow0$ and an $\E$\+preenvelope $g\:K\rarrow T^n$ with
$T^j\in\Add(T)$ for all $0\le j\le n$.
 Denoting by $M^j$ the image of the morphism $T^{j-1}\rarrow T^j$
and using the conditions~(i--ii), we have $\Ext^i_\A(T,K)=
\Ext^{i+1}_\A(T,M^{n-1})=\dotsb=\Ext^{i+n}_\A(T,L)=0$ for $i>0$.
 So $K\in\E$.
 Then the map $\Hom_\A(g,K)$ surjective, and hence $K$ is a direct
summand of~$T^n$.
 Thus $K\in\Add(T)$ and $L\in\L$.
\end{proof}

 Using Theorem~\ref{tilting-cotorsion-pair-thm}, we can compare our
definition of a (big) $n$\+tilting object with the traditional
definition of an (infinitely generated) $n$\+tilting module.

 Let $A$ be an associative ring and $\A=A\modl$ be the category of
left $A$\+modules.
 According to the definition going back to the papers~\cite{AC,Ba0},
a left $A$\+module $T$ is called \emph{$n$\+tilting} if it satisfies
the conditions~(i--ii) from Section~\ref{tilting-t-structure-secn}
as an object of the category $A\modl$, and in addition,
the following condition holds:
\begin{enumerate}
\renewcommand{\theenumi}{\roman{enumi}$_\mathrm{m}$}
\setcounter{enumi}{2}
\item the free left $A$\+module $A$ has a finite coresolution
$$
 0\lrarrow A\lrarrow T^0\lrarrow\dotsb\lrarrow T^r\lrarrow0
$$
by $A$\+modules $T^i$ belonging to the subcategory
$\Add(T)\subset A\modl$.
\end{enumerate}

 According to~\cite[Proposition~3.5]{Ba0}, when the conditions~(i--ii)
and~(iii$_\mathrm{m}$) are satisfied, one can have $r=n$, but
$r$~cannot be made smaller than that (or more precisely, than
the projective dimension of~$T$).

\begin{cor} \label{tilting-objects-and-tilting-modules}
 Let $T$ be a left $A$\+module satisfying, as an object of
the category\/ $\A=A\modl$, the conditions~\textup{(i)}
and~\textup{(ii)}.
 Then $T$ satisfies the condition~\textup{(iii)} if and only if it
satisfies~\textup{(iii$_\mathrm{m}$)}.
\end{cor}

\begin{proof}
 It is obvious that (iii$_\mathrm{m}$) implies~(iii)
(cf.\ the proof of (3)~$\Longrightarrow$~(1) in
Theorem~\ref{tilting-cotorsion-pair-thm}).

%To prove the converse implication, note that
%Theorem~\ref{tilting-cotorsion-pair-thm}(3) provides us with
%a surjection $L\rarrow A$. As such a surjection must split,
%we have $A\in\L$ by Lemma~\ref{Add(T)-E-preenvelope}. The existence
%of~\textup{(iii$_\mathrm{m}$)} follows from the very definition
%of $\L$.

To prove the converse implication, use the assertion of
Theorem~\ref{tilting-cotorsion-pair-thm} claiming that $(\L,\E)$
is a cotorsion pair in~$\A$.
 Since $\Ext_A^i(A,E)=0$ for all $E\in A\modl$ and $i>0$,
the free left $A$\+module $A$ belongs to~$\L$.
This provides the exact sequence~(iii$_\mathrm{m}$) with
$r=n$.
% (Alternatively, one could compare the result of
%Lemma~\ref{tilting-cotilting-classes-lemma}(a) below with
%the characterization of $n$\+tilting modules provided
%by~\cite[Theorem~4.4]{AC} or~\cite[Proposition~3.10]{Ba0}
%in order to show that (iii) implies~(iii$_\mathrm{m}$).)
\end{proof}

\begin{rem} \label{rem:tilting-class}
 In the literature, the full subcategory $\E\subset\A$ is known as
the \emph{$n$\+tilting class} associated with an $n$\+tilting object
$T\in\A$ \cite[Section~13.1]{GTbook}.
 The cotorsion pair $(\L,\E)$ in $\A$ is called the \emph{$n$\+tilting
cotorsion pair}.
\end{rem}

\Section{The Tilting-Cotilting Correspondence}
\label{tilting-cotilting-secn}

 The aim of this section is to show that the assignment of the tilting
heart $\B$ to an abelian category $\A$ with an $n$\+tilting object $T$
extends to a bijective correspondence between certain natural classes
of abelian categories $\A$ with $n$\+tilting objects $T$ and abelian
categories $\B$ with $n$\+cotilting objects~$W$.
 
 We start with some standard observations on tilting t\+structures along the lines of~\cite[Remarque 3.1.17]{BBD},
 \cite[Corollary~A.17]{Partin}.
 For any t\+structure $(\D^{\le0},\D^{\ge0})$ on a triangulated category
$\D$ with an abelian heart $\C=\D^{\le0}\cap\D^{\ge0}$, there are
natural maps from the $\Ext$ groups in the abelian category $\C$ to
the $\Hom$ groups in the triangulated category~$\D$
\begin{equation} \label{maps-theta}
 \theta^i_{\C,\D}=\theta_{\C,\D}^i(X,Y)\:\Ext^i_\C(X,Y)
 \to\Hom_\D(X,Y[i]),
 \quad \text{for all $X$, $Y\in\C$, \,$i\ge0$}.
\end{equation}
 The maps $\theta^i_{\C,D}(X,Y)$ are functorial in $X$ and $Y$ and
transform the Yoneda multiplication of $\Ext$ classes into
the composition of morphisms in~$\D$.
 The maps~$\theta_{\C,\D}^0$ and $\theta_{\C,\D}^1$ are always
isomorphisms, and the maps $\theta_{\C,\D}^2$ are monomorphisms.
 The maps $\theta_{\C,\D}^{i+1}$ are monomorphisms whenever the maps
$\theta_{\C,\D}^i$ are isomorphisms.

 A t\+structure $(\D^{\le0},\D^{\ge0})$ on $\D$ is said to be
\emph{of the derived type} if the maps $\theta_{\C,\D}^i$ are
isomorphisms for all~$i\ge0$.
 A t\+structure is of the derived type if and only if for every
morphism $\xi\:X\rarrow Y[i]$ in $\D$ with $X$, $Y\in\C$ and $i>0$
there exists an epimorphism $\pi\:X'\rarrow X$ in the abelian
category $\C$ such that $\xi\pi=0$ in $\D$, and if and only if
for every~$\xi$ there exists a monomorphism $\iota\:Y\rarrow Y'$
in $\C$ such that $\iota[i]\circ\xi=0$.

\medskip

 Let $\A$ be an abelian category with set-indexed products and
an injective cogenerator, and let $T\in\A$ be an $n$\+tilting object.
 Let $({}^T\D^{\b,\le0},\.{}^T\D^{\b,\ge0})$ be the tilting t\+structure
on $\D^\b(\A)$, and let $\B$ be the heart of this t\+structure.

\begin{lem} \label{tilting-derived-type}
 The tilting t\+structure $({}^T\D^{\b,\le0},\.{}^T\D^{\b,\ge0})$ on\/
$\D^\b(\A)$ associated with any $n$\+tilting object $T\in\A$ is
a t\+structure of the derived type.
 Equivalently, the tilting t\+structure
$({}^T\D^{\le0},\.{}^T\D^{\ge0})$ on\/ $\D(\A)$ is of the derived type.
\end{lem}

\begin{proof}
 Let us check that for every pair of objects $X$, $Y\in\B$ and
a morphism $\xi\:X\rarrow Y[i]$, \ $i>0$ in $\D^\b(\A)$ there exists
an epimorphism $\pi\:X'\rarrow X$ in $\B$ such that $\xi\pi=0$.
 Set $I=\Hom_\B(T,X)$ and $X'=T^{(I)}$; let $\pi\:X'\rarrow X$
be the natural epimorphism.
 Then $\xi\pi=0$ in $\D^\b(\A)$ because $\Hom_{\D^\b(\A)}(T,Y[i])=0$
by the definition of ${}^T\D^{\b,\le0}$ and
$\Hom_{\D^\b(\A)}(T^{(I)},Y[i])=\Hom_{\D^\b(\A)}(T,Y[i])^I$.
\end{proof}

This allows us to prove that the categories $\A$ and $\B$ are derived equivalent (at the bounded level, see Section~\ref{derived-equivalences-secn} for a more thorough discussion of derived equivalences). In fact, we provide two proofs: the first one is short and elegant, but it relies on a nontrivial technical construction of a so-called realization functor, while the second one is elementary and only uses the results of our Section~\ref{tilting-cotorsion-pair-secn}.

\begin{prop} \label{bounded-derived-equivalence}
 There is an equivalence of triangulated categories\/
$\D^\b(\B)\simeq\D^\b(\A)$ identifying the standard t\+structure
$(\D^{\b,\le0}(\B),\.\D^{\b,\ge0}(\B))$ on\/ $\D^\b(\B)$ with
the tilting t\+structure $({}^T\D^{\b,\le0},\.{}^T\D^{\b,\ge0})$ on\/
$\D^\b(\A)$.
 The restriction of this triangulated equivalence to
the full subcategory\/ $\B\subset\D^\b(\B)$ is the identity
embedding\/ $\B\rarrow\D^\b(\A)$.
\end{prop}

\begin{proof}[First proof]
 According to~\cite[n$^\circ$\,3.1]{BBD} and~\cite[Appendix~A]{PV}
(for another approach, see~\cite[\S3]{KV}), for any t\+structure on
the derived category $\D^\b(\A)$ with the abelian heart $\B$ there
exists a triangulated ``realization'' functor $\D^\b(\B)\rarrow
\D^\b(\A)$ whose restriction to~$\B$ is the identity embedding
$\B\rarrow\D^\b(\A)$.
 Furthermore, the inclusions~(\ref{d-le-inclusions}\+%
\ref{d-ge-inclusions}) show that the tilting t\+structure on
$\D^\b(\A)$ is bounded, so the triangulated category $\D^\b(\A)$ is
classically generated by its full subcategory~$\B$. That is,
the smallest triangulated subcategory of $\D^\b(\A)$ containing $\B$
is $\D^\b(\A)$ itself.
 As the tilting t\+structure is also of the derived type by
Lemma~\ref{tilting-derived-type}, the functor
$\D^\b(\B)\rarrow\D^\b(\A)$ is a triangulated equivalence
(cf.~\cite[Proposition~3.1.16]{BBD}).
 Since a bounded t\+structure on a triangulated category is determined
by its heart, the equivalence of categories $\D^\b(\B)\simeq\D^\b(\A)$
identifies the standard t\+structure on $\D^\b(\B)$ with
the tilting t\+structure on $\D^\b(\A)$. 
\end{proof}

\begin{proof}[Second proof]
 According to Lemmas~\ref{tilting-class-coresolving-finite-dim}
and~\ref{cotilting-class-resolving-finite-dim}, the full subcategory
$\E=\A\cap\B\subset\D^\b(\A)$ is coresolving in $\A$ and
resolving in~$\B$.
 The (co)resolution dimension is bounded by~$n$ in both cases.
 The exact category structures inherited by $\E$ from the abelian categories
$\A$ and $\B$ coincide (e.~g., because they can be inherited directly
from the triangulated category structure on $\D^\b(\A)$;
see~\cite[Section~A.8]{Partin}).
 Hence it follows that both the triangulated functors $\D^\b(\E)\rarrow
\D^\b(\A)$ and $\D^\b(\E)\rarrow\D^\b(\B)$ induced by the identity
embeddings $\E\rarrow\A$ and $\E\rarrow\B$ are triangulated equivalences
between the bounded derived category $\D^\b(\E)$ of the exact category
$\E$ and the bounded derived categories of the abelian categories
$\A$ and~$\B$ (see~\cite[Proposition~13.2.2(ii)]{KS}
or~\cite[Proposition~A.5.6]{Pcosh}; cf.\ the discussion in
Section~\ref{derived-equivalences-secn} and in the proof of
Theorem~\ref{exotic-derived-resolution-equivalences} below). 
 Hence we obtain the triangulated equivalence $\D^\b(\B)\simeq\D^\b(\A)$.

 It remains to show that the identity embeddings $\B\rarrow\D^\b(\B)$
and $\B\rarrow\D^\b(\A)$ form a commutative diagram with
the equivalence $\D^\b(\B)\simeq\D^\b(\A)$.
 Indeed, let $Y\in\B$ be an object and $0\rarrow E^{-n}\rarrow\dotsb
\rarrow E^{-1}\rarrow E^0\rarrow Y\rarrow0$ be an exact sequence in
$\B$ with $E^{-k}\in\E$, as in
Lemma~\ref{cotilting-class-resolving-finite-dim}(b).
 We have to construct a natural isomorphism in $\D^\b(\A)$ between the object
represented by the complex $E^\bu$ in $\D^\b(\A)$ and
the object $Y\in\B\subset\D^\b(\A)$.
%
% For this purpose, one considers the distinguished triangles of silly
%filtration $E^0\rarrow E^\bu\rarrow\sigma_{\le-1}E^\bu\rarrow E^0[1]$
%and %$E^{-1}[-1]\rarrow\sigma_{\le-1}E^\bu\rarrow\sigma_{\le-2}E^\bu
%\rarrow E^{-1}$ in $\D^\b(\A)$.
% Since $E^{-k}\in\E\subset{}^T\D^{\b,\le0}$ for all $1\le k\le n$,
%we have $\sigma_{\le-1}E^\bu\in{}^T\D^{\b,\le-1}$ and
%$\sigma_{\le-2}E^\bu\in{}^T\D^{\b,\le-2}$, so
%$\Hom_{\D^\b(\A)}(\sigma_{\le-1}E^\bu,\B)=0=
%\Hom_{\D^\b(\A)}(\sigma_{\le-2}E^\bu[-1],\B)$.
%
Since $\E\subset{}^T\D^{\b,\le0}$ and ${}^T\D^{\b,\le0}$ is closed under cones, one has that $E^\bu\in{}^T\D^{\b,\le0}$.
More generally, we have $\sigma_{\le i}E^\bu\in{}^T\D^{\b,\le i}$ for each $i\le 0$, where $\sigma_{\le i}$ is the silly truncation of $E^\bu$.
Now consider the triangles
$(\sigma_{\le-1}E^\bu)[-1] \rarrow E^0\rarrow E^\bu\rarrow \sigma_{\le-1}E^\bu$
and
$E^{-1}\rarrow (\sigma_{\le-1}E^\bu)[-1]\rarrow(\sigma_{\le-2}E^\bu)[-1]\rarrow E^{-1}[1]$
in $\D^\b(\A)$.
If we apply $\Hom_{\D^\b(\A)}(-,Y)$ to these triangles and use that $\B\subset{}^T\D^{\b,\ge0}$,
it follows that the group $\Hom_{\D^\b(\A)}(E^\bu,Y)$ is
isomorphic to the group of all morphisms $E^0\rarrow Y$ in $\B$
for which the composition $E^{-1}\rarrow E^0\rarrow Y$ vanishes.
 In particular, from the exact sequence $0\rarrow E^{-n}\rarrow
\dotsb\rarrow E^0\rarrow Y\rarrow 0$ in $\B$ we get a natural morphism
$E^\bu\rarrow Y$ in $\D^\b(\A)$.
 To check that this morphism is an isomorphism, one computes that
it induces an isomorphism $\Hom_{\D^\b(\A)}(T,E^\bu[i])\cong
\Hom_{\D^\b(\A)}(T,Y[i])$ for all $i\in\Z$.
 Indeed, one has $\Hom_{\D^\b(\A)}(T,E^\bu[i])=
H^i\Hom_\A(T,E^\bu)=H^i\Hom_\B(T,E^\bu)$, since
$\Ext^i_\A(T,E^{-k})=0$ for all $0\le k\le n$ and $i>0$,
and applying $\Hom_\B(T,{-})$ preserves exactness of
the sequence $0\rarrow E^{-n}\rarrow\dotsb\rarrow E^0\rarrow
Y\rarrow0$, because the object $T$ is projective in~$\B$.
\end{proof}

 Now we proceed to define the cotilting objects.
 The setting for these is completely dual to the tilting one.
 Let $\B$ be an abelian category with set-indexed coproducts
and a projective generator.
 It follows from these conditions that set-indexed products exist and
are exact in~$\B$.
 
 Let us say that an object $W\in\B$ is (\emph{big})
\emph{$n$\+cotilting} if the following three conditions are satisfied:
\begin{enumerate}
\renewcommand{\theenumi}{\roman{enumi}*}
\item the injective dimension of $W$ in $\B$ does not exceed~$n$,
that is $\Ext^i_\B(Y,W)=0$ for all $i>n$ and all $Y\in\B$;
\item $\Ext^i_\A(W^I,W)=0$ for all $i>0$ and all sets $I$, where
$W^I$ denotes the product of $I$~copies of $W$ in~$\B$;
\item every complex $Y^\bu\in\D(\B)$ such that $\Hom_{\D(\A)}
(Y^\bu,W[i])=0$ for all $i\in\Z$ is acyclic.
\end{enumerate}

\begin{thm} \label{cotilting-t-structure-thm}
 Let $W\in\B$ be an $n$\+cotilting object.
 Then the pair of full subcategories
\begin{align*}
 {}^W\D^{\le 0}&=\{\,Y^\bu\in\D(\B)\mid \Hom_{\D(\B)}(Y^\bu,W[i])=0
 \text{ for all $i<0$}\,\}, \\
 {}^W\D^{\ge 0}&=\{\,Y^\bu\in\D(\B)\mid \Hom_{\D(\B)}(Y^\bu,W[i])=0
 \text{ for all $i>0$}\,\}
\end{align*}
is a t\+structure on the unbounded derived category\/ $\D(\B)$.
\end{thm}

\begin{proof}
 Dual to Theorem~\ref{tilting-t-structure-thm}.
\end{proof}

\begin{cor} \label{bounded-cotilting-t-structure}
\textup{(a)} Let $W\in\B$ be an $n$\+cotilting object.  Then
the pair of full subcategories
\begin{align*}
 {}^W\D^{\b,\le 0}&=\{\,Y^\bu\in\D^\b(\B)\mid \Hom_{\D^\b(\B)}
 (Y^\bu,W[i])=0 \text{ for all $i<0$}\,\}, \\
 {}^W\D^{\b,\ge 0}&=\{\,Y^\bu\in\D^\b(\B)\mid \Hom_{\D^\b(\B)}
 (Y^\bu,W[i])=0 \text{ for all $i>0$}\,\}
\end{align*}
is a t\+structure on the bounded derived category\/ $\D^\b(\B)$. \par
\textup{(b)} Conversely, if $W\in\B$ is an object
satisfying~\textup{(i*)} and~\textup{(ii*)}, and the pair of
full subcategories\/ $({}^W\D^{\b,\le0},\.{}^W\D^{\b,\ge0})$ is
a t\+structure on\/ $\D^\b(\B)$, then the object $W$ also satisfies
the condition~\textup{(iii*)}.
\end{cor}

\begin{proof}
 Part~(a) is dual to Corollary~\ref{bounded-tilting-t-structures}(c).
 Part~(b) is dual to
Proposition~\ref{bounded-tilting-t-structure-implies-generation}.
\end{proof}

 The t\+structures on the derived categories $\D(\B)$ and $\D^\b(\B)$
provided by Theorem~\ref{cotilting-t-structure-thm}
and Corollary~\ref{bounded-cotilting-t-structure}(a), as well as
the similar t\+structures on $\D^+(\B)$ and $\D^-(\B)$, are called
the \emph{cotilting t\+structures} associated with an $n$\+cotilting
object $W\in\B$.

 More generally, let $W\in\B$ be an object of injective dimension not
exceeding~$n$.
 Denote by $\E$ the following full subcategory in~$\B$:
\[
 \E = \{ E\in\B \mid\Ext^i_\B(E,W)=0 \textrm{ for all } i>0 \}.
\]
 In the way dual to Lemma~\ref{tilting-class-coresolving-finite-dim},
one shows that the full subcategory $\E\subset\B$ is resolving and
the corresponding resolution dimension of $\B$ is bounded by~$n$.

 Denote by $\Prod(W)\subset\B$ the full subcategory formed by
the direct summands of infinite products of copies of the object
$W\in\B$.
 For any integer $m\ge0$, we denote by $\R_m$ the full subcategory of
all objects $R\in\B$ for which there exists an exact sequence in $\B$
of the form
$$
 0\lrarrow W_m\lrarrow\dotsb\lrarrow W_1\lrarrow W_0\lrarrow R
 \lrarrow0,
$$
where $W_0$, $W_1$,~\dots, $W_m\in\Prod(W)$.
 Clearly, one has $\Prod(W)=\R_0\subset\R_1\subset\dotsb\subset\B$.

\begin{lem}
 Assume that the object $W\in\B$ satisfies
the conditions~\textup{(i*--ii*)}.
 Then \par
\textup{(a)} for any objects $E\in\E$ and $R\in\R_m$, one has\/
$\Ext_\B^i(E,R)=0$ for all $i>0$; \par
\textup{(b)} the intersection\/ $\R_m\cap\E$ coincides with the full
subcategory\/ $\Prod(W)\subset\B$; \par
\textup{(c)} for each integer $m\ge n$, one has\/ $\R_m=\R_{m+1}$.
\end{lem}

\begin{proof}
 Dual to Lemma~\ref{left-class-tilting-cotorsion-pair}.
\end{proof}

 Assuming that the object $W\in\B$ satisfies the conditions~(i*--ii*),
we will denote the full subcategory $\R_n=\R_{n+1}=\R_{n+2}=\dotsb$
simply by $\R\subset\B$.

\begin{thm} \label{cotilting-cotorsion-pair-thm}
 Let\/ $\B$ be an abelian category with set-indexed coproducts and
a projective generator, and let $W\in\B$ be an object satisfying
the conditions~\textup{(i*--ii*)}.
 Then the following three conditions are equivalent:
\begin{enumerate}
\item[(1*)]  the object $W\in\B$ satisfies
the condition~\textup{(iii*)};
\item[(2*)] for every object $E\in\E$ there exists an object
$W'\in\Prod(W)$ together with a monomorphism $E\rarrow W'$
in the category\/~$\B$;
\item[(3*)] for every object $Y\in\B$ there exists an object
$R\in\R$ together with a monomorphism $Y\rarrow R$ in
the category\/~$\B$.
\end{enumerate}
 If one of the conditions~\textup{(1*--3*)} is satisfied, then the pair
of full subcategories $(\E,\R)$ is a hereditary complete cotorsion pair
in the abelian category\/~$\B$.
\end{thm}

\begin{proof}
 Dual to Theorem~\ref{tilting-cotorsion-pair-thm}.
\end{proof}

\begin{rem} \label{cotilting-modules-remark}
 Similarly to Corollary~\ref{tilting-objects-and-tilting-modules},
one can show, using Theorem~\ref{cotilting-cotorsion-pair-thm}
(or alternatively, using Lemma~\ref{tilting-cotilting-classes-lemma}(b)
below), that the specialization of our definition of an $n$\+cotilting
object in an abelian category to the case of the category of
modules over an associative ring $\B=B\modl$ is equivalent to
the definition of an $n$\+cotilting module studied in
the papers~\cite{AC,Ba0}.
\end{rem}

 The full subcategory $\E\subset\B$ is known as the \emph{$n$\+cotilting
class} associated with an $n$\+cotilting object $W\in\B$
\cite[Section~15.1]{GTbook}.
 The cotorsion pair $(\E,\R)$ in $\B$ is called the \emph{$n$\+cotilting
cotorsion pair}.

 Let $W\in\B$ be an $n$\+cotilting object.
 Set $\A={}^W\D^{\b,\le0}\cap{}^W\D^{\b,\ge0}$ to be the heart of
the cotilting t\+structure on $\D^\b(\B)$.
 By the definition, $\A$ is an abelian category.

\begin{prop} \label{cotilting-heart}
 The object $W\in\B\subset\D(\B)$ belongs to\/ $\A$ and is an injective
cogenerator of the abelian category\/~$\A$.
 Products of copies of $W$ in\/ $\A$ coincide with such products
in\/ $\D(\B)$ and in\/~$\B$.
 The injective objects of\/ $\A$ are precisely the direct summands of
these products.
 Set-indexed products of arbitrary objects exist in\/~$\A$.
\end{prop}

\begin{proof}
 Dual to Proposition~\ref{tilting-heart}.
\end{proof}

\begin{lem} \label{cotilting-derived-type}
 The cotilting t\+structure $({}^W\D^{\b,\le0},\.{}^W\D^{\b,\ge0})$ on\/
$\D^\b(\B)$ associated with any $n$\+cotilting object $W\in\B$ is
a t\+structure of the derived type.
 Equivalently, the cotilting t\+structure
$({}^W\D^{\le0},\.{}^W\D^{\ge0})$ on\/ $\D(\B)$ is a t\+structure of
the derived type.
\end{lem}

\begin{proof}
 Dual to Lemma~\ref{tilting-derived-type}.
\end{proof}

 Let $\A$ be an abelian category with set-indexed products and
an injective cogenerator, and let $T\in\A$ be an $n$\+tilting object.
 Choose an injective cogenerator $W\in\A$.

 Let $({}^T\D^{\b,\le0},\.{}^T\D^{\b,\ge0})$ be the tilting t\+structure
on $\D^\b(\A)$ corresponding to the tilting object~$T$, and let
$\B$ be the heart of this t\+structure.
 For any set $I$, we have $W^I\in\B\subset\D^\b(\A)$.
 Furthermore, one has
$$
 \Hom_{\D(\A)}(X^\bu,W^I)=\Hom_{\Hot(\A)}(X^\bu,W^I)
$$
for every complex $X^\bu\in\Hot(\A)$.
 It follows that $W^I\in\A\subset\D(\A)$ is the product of $I$
copies of $W$ in $\D(\A)$, and therefore also in $\D^\b(\A)\subset
\D(\A)$ and in $\B\subset\D^\b(\A)$.

 The following theorem is the main result of this section.

\begin{thm} \label{tilting-cotilting-thm}
 Let $\A$ be a complete, cocomplete abelian category with an injective
cogenerator $W\in\A$, let $T\in\A$ be an $n$\+tilting object, and
let\/ $\B={}^T\D^{\b,\le0}\cap{}^T\D^{\b,\ge0}$ be the heart of
the tilting t\+structure on\/ $\D^\b(\A)$.
 Then $W\in\B\subset\D^\b(\A)$ is an $n$\+cotilting object in
the abelian category\/~$\B$.
\end{thm}

\begin{proof}[First proof]
 According to Proposition~\ref{tilting-heart}, \,$\B$ is a complete,
cocomplete abelian category with a projective generator~$T$.
 We have explained that the objects
$W^I\in\A\subset\D^\b(\A)$ belong to $\B\subset\D^\b(\A)$ and are
the products of $I$ copies of $W$ in~$\B$.
 By Proposition~\ref{bounded-derived-equivalence}, we
have an equivalence of triangulated categories $\D^\b(\A)\simeq
\D^\b(\B)$ which agrees with the identity embedding $\B\rarrow
\D^\b(\A)$ and transforms the tilting t\+structure on $\D^\b(\A)$
into the standard t\+structure on $\D^\b(\B)$.

 Let us check that the conditions~(i*--iii*) hold for the object
$W\in\B$.
 We have
$$
 \Ext^i_\B(W^I,W)=\Hom_{\D^\b(\B)}(W^I,W[i])=
 \Hom_{\D^\b(\A)}(W^I,W[i])=\Ext^i_\A(W^I,W)=0
$$
for $i>0$, so (ii*)~is satisfied.

 We have
\begin{align*}
 \D^{\b,\le 0}(\A)&=\{\,X^\bu\in\D^\b(\A)\mid \Hom_{\D^\b(\A)}(X^\bu,W[i])=0
 \text{ for all $i<0$}\,\}, \\
 \D^{\b,\ge 0}(\A)&=\{\,X^\bu\in\D^\b(\A)\mid \Hom_{\D^\b(\A)}(X^\bu,W[i])=0
 \text{ for all $i>0$}\,\},
\end{align*}
since $W$ is a injective cogenerator of~$\A$.
 So the pair of full subcategories $(\D^{\b,\le0}(\A),\.\allowbreak
\D^{\b,\ge0}(\A))$ in $\D^\b(\A)$ is transformed by the triangulated
equivalence $\D^\b(\A)\simeq\D^\b(\B)$ into the pair of full
subcategories $({}^W\D^{\b,\le0},\.{}^W\D^{\b,\ge0})$ in $\D^\b(\B)$.
 The former is a t\+structure on $\D^\b(\A)$, hence it follows that
the latter is a t\+structure on $\D^\b(\B)$.

 Given our identifications of the t\+structures on
$\D^\b(\A)$ and $\D^\b(\B)$,
the inclusions~(\ref{d-le-inclusions}\+\ref{d-ge-inclusions})
of full subcategories in $\D^\b(\A)$ imply the inclusions
\begin{alignat}{2}
\D^{\b,\le0}(\B)&\subset{}^W\D^{\b,\le0}&&\subset
\D^{\b,\le n}(\B), \label{cotilting-d-le-inclusions} \\
\D^{\b,\ge n}(\B)&\subset{}^W\D^{\b,\ge0}&&\subset
\D^{\b,\ge0}(\B) \label{cotilting-d-ge-inclusions}
\end{alignat}
of full subcategories in $\D^\b(\B)$.
 In particular, according to~\eqref{cotilting-d-ge-inclusions}
for every object $Y\in\B$ we have $Y\in\D^{\b,\ge0}(\B)\subset
{}^W\D^{\b,\ge-n}$, implying that
$$
 \Ext^i_\B(Y,W)=\Hom_{\D^\b(\B)}(Y,W[i])=0
 \quad\text{for $i>n$}.
$$
 Thus (i*)~is satisfied.
 Finally, it remains to apply
Corollary~\ref{bounded-cotilting-t-structure}(b)
in order to deduce the condition~(iii*).
\end{proof}

\begin{proof}[Second proof]
 The following argument avoids the use of
Proposition~\ref{bounded-derived-equivalence},
using Lemma~\ref{tilting-derived-type} and
Theorem~\ref{cotilting-cotorsion-pair-thm} instead.

 By Lemma~\ref{tilting-derived-type}, we have
$$
 \Ext^i_\B(W^I,W)=\Hom_{\D^\b(\A)}(W^I,W[i])=\Ext^i_\A(W^I,W)=0
$$
for $i>0$, so the condition~(ii*) holds.
 To prove~(i*), we notice that
$$
 \Ext^i_\B(Y,W)=\Hom_{\D^\b(\A)}(Y,W[i])=0
 \quad\text{for all $Y\in\B$ and $i>n$},
$$
since $Y\in{}^T\D^{\b,\ge0}\subset\D^{\b,\ge-n}(\A)$
by~(\ref{d-ge-inclusions}) and $W\in\A_\inj$.

 Finally, let us check the condition~(2*) of
Theorem~\ref{cotilting-cotorsion-pair-thm}.
 Since $\Ext^i_\B(E,W)=\Hom_{\D^\b(\A)}(E,W[i])$ for all $E\in\B$
and $i\ge0$ by Lemma~\ref{tilting-derived-type}, the full
subcategory $\E\subset\B$ consisting of all objects $E\in\B$
such that $\Ext^i_\B(E,W)=0$ for $i>0$ can be described as
the intersection $\B\cap\A\subset\D^\b(\A)$.
 As a full subcategory of $\A$, this intersection coincides with
the $n$\+tilting class $\E=\A\cap\B\subset\A$ discussed in
Section~\ref{tilting-cotorsion-pair-secn}.

 Given an object $E\in\E\subset\B$, we first consider it as
an object of~$\A$.
 Since $W$ is an injective cogenerator of~$\A$, there is
a short exact sequence $0\rarrow E\rarrow W^I\rarrow E'\rarrow0$
in~$\A$.
 According to Lemma~\ref{tilting-class-coresolving-finite-dim}(a),
we have $E'\in\E$.
 Now there is a distinguished triangle $E\rarrow W^I\rarrow E'
\rarrow E[1]$ in $\D^\b(\A)$ with the objects $E$, $W^I$, and $E'$
belonging to $\B$, hence a short exact sequence $0\rarrow E
\rarrow W^I\rarrow E'\rarrow0$ in~$\B$.
 So $E\rarrow W'=W^I$ is a monomorphism in~$\B$, as desired.

 Alternatively, now that we have seen that the definition of
the full subcategory $\E\subset\B$ discussed in
Section~\ref{tilting-cotorsion-pair-secn} agrees with the one
in the present section, we can deduce~(i*) from
Lemma~\ref{cotilting-class-resolving-finite-dim}(b).
 Given an object $Y\in\B$, we have an exact sequence
$0\rarrow E_n\rarrow\dotsb\rarrow E_0\rarrow Y\rarrow 0$
with $E_k\in\E$.
 Then $\Ext^i_\B(E_k,W)=0$ for all $0\le k\le n$ and $i>0$,
hence $\Ext^i_\B(Y,W)=0$ for all $i>n$.
\end{proof}

\begin{thm} \label{cotilting-tilting-thm}
 Let $\B$ be a complete, cocomplete abelian category with a projective
generator $T\in\B$, let $W\in\B$ be an $n$\+cotilting object, and
let\/ $\A={}^W\D^{\b,\le0}\cap{}^W\D^{\b,\ge0}$ be the heart of
the cotilting t\+structure on\/ $\D^\b(\B)$.
 Then $T\in\A\subset\D^\b(\B)$ is an $n$\+tilting object in
the abelian category\/~$\A$.
\end{thm}

\begin{proof}
 Dual to Theorem~\ref{tilting-cotilting-thm}.
\end{proof}

\begin{cor} \label{tilting-cotilting-correspondence-cor}
The constructions of Theorems\/~\textup{\ref{tilting-cotilting-thm}}
and\/~\textup{\ref{cotilting-tilting-thm}} establish a one-to-one
correspondence between the equivalence classes of
\begin{enumerate}
\renewcommand{\theenumi}{\alph{enumi}}
\item complete, cocomplete abelian categories\/ $\A$
with an injective cogenerator $W\in\A$ and an $n$\+tilting object
$T\in\A$, and
\item complete, cocomplete abelian categories\/ $\B$ with
a projective generator $T\in\B$ and an $n$\+cotilting object $W\in\B$.
\qed
\end{enumerate}
\end{cor}

\medskip\noindent
\textit{Note added in proofs.}
 As of April~2019, we were informed about results similar to
the previous corollary, which were obtained by Apostolos Beligiannis
and are expected to appear in his paper called ``Tilting theory in
abelian categories and related homological and homotopical
structures''.

\Section{Derived Equivalences} \label{derived-equivalences-secn}

 Let $\A$ be a complete, cocomplete abelian category with an injective
cogenerator $W$ and an $n$\+tilting object $T$, and let $\B$ be
the corresponding complete, cocomplete abelian category with
a projective generator $T$ and an $n$\+cotilting object~$W$.
 The aim of this section is to construct, for any conventional or
absolute derived category symbol $\st=\b$, $+$, $-$, $\varnothing$,
$\abs+$, $\abs-$, or~$\abs$, an equivalence of derived categories
\begin{equation} \label{derived-equivalence}
 \D^\st(\A)\simeq\D^\st(\B).
\end{equation}

 Here the \emph{absolute derived category} $\D^\abs(\C)$ of an abelian
category $\C$ is defined as the Verdier quotient of $\Hot(\C)$ by
the thick subcategory generated by all totalizations of short
exact sequences of complexes.
 In other words, one forces short exact sequences of complexes to induce 
triangles and nothing more.
 The bounded versions of the absolute derived category are defined
similarly (we refer to~\cite[Section~A.1]{Pcosh}
or~\cite[Appendix~A]{Pmgm} and the references therein for a detailed
discussion and the definitions of other exotic derived categories).

 In particular, it will follow that there are two t\+structures on
both $\D^\st(\A)$ on $\D^\st(\B)$: the standard t\+structure on
$\D^\st(\B)$ is transformed by
the equivalence~\eqref{derived-equivalence}
into (what will be called) the tilting t\+structure on $\D^\st(\A)$,
while the standard t\+structure on $\D^\st(\A)$ is transformed into
the cotilting t\+structure on $\D^\st(\B)$.

 Notice that for the bounded derived categories $\D^\b(\A)$ and
$\D^\b(\B)$ we already have the desired picture.
 According to Proposition~\ref{bounded-derived-equivalence}, there
is a triangulated equivalence $\D^\b(\A)\simeq\D^\b(\B)$ transforming
the tilting t\+structure $({}^T\D^{\b,\le0},\.{}^T\D^{\b,\ge0})$ on
$\D^\b(\A)$ into the standard t\+structure on $\D^\b(\B)$.
 As it was explained in the first proof of
Theorem~\ref{tilting-cotilting-thm}, the same triangulated equivalence
also transforms the cotilting t\+structure 
$({}^W\D^{\b,\le0},\.{}^W\D^{\b,\ge0})$ on $\D^\b(\B)$ into
the standard t\+structure on $\D^\b(\A)$.

 As above, we denote by $\E=\A\cap\B$ the intersection of the two
t\+structure hearts in the triangulated category
$\D^\b(\A)\simeq\D^\b(\B)$.
 The full subcategory $\E$ is coresolving in $\A$ and resolving in~$\B$.
 Being, in particular, closed under extensions in both $\A$ and $\B$,
the additive category $\E$ inherits the exact category structure from
either of the abelian categories $\A$ and $\B$; one can easily
see that the two exact category structures on $\E$ obtained in this
way coincide (as it was mentioned in the first paragraph of the second
proof of Proposition~\ref{bounded-derived-equivalence}).

 Next we will give an alternative characterization of
the $n$\+tilting class $\E\subset\A$ and the $n$\+cotilting class
$\E\subset\B$, but it is convenient to introduce some notation first.
 We consider the restriction $\Psi = {}^TH^0|_\A\: \A \rarrow \B$ of
the zero cohomology functor with respect to the tilting
t\+structure $({}^T\D^{\b,\le0},\.{}^T\D^{\b,\ge0})$
on the derived category $\D^\b(\A)$.
 Since $\A \subset {}^T\D^{\b,\ge 0}$ by~\eqref{d-ge-inclusions},
we in fact have $\Psi = \tau^T_{\le0}|_\A$ and the functor
$\Psi\: \A \rarrow \B$ is left exact.
 Moreover, the restriction of $\Psi$ to the full subcategory
$\E = \A \cap \B$ coincides with the inclusion $\E\subset\B$.

 Similarly, we consider an extension of the embedding $\E\subset\A$
to a right exact functor $\Phi\:\B\rarrow\A$ defined as
$\Phi = {}^WH^0|_\B = \tau^W_{\ge0}|_\B$.
 Here $\tau^W_{\ge0}$ is the truncation functor and ${}^WH^0$ is
the zero cohomology functor, respectively, with respect to
the cotilting t\+structure $({}^W\D^{\b,\le0},\.{}^W\D^{\b,\ge0})$ on
the derived category $\D^\b(\B)$.

 The functor $\Psi\:\A\rarrow\B$ is right adjoint to the functor
$\Phi\:\B\rarrow\A$.
 Indeed, we may in view of Proposition~\ref{bounded-derived-equivalence}
express $\Psi$ as the restriction to $\A$ of the following composition
of right adjoint functors
\[
\D^{\b,\ge 0}(\A) \overset{\subset}\lrarrow
\D^{\b}(\A) \simeq
\D^\b(\B) \overset{\tau_{\le0}}\lrarrow
\D^{\b,\le 0}(\B).
\]
 Dually, the restriction to $\B$ of the composition of the
corresponding left adjoints coincides with $\Phi$.

 As above, we denote by $\Add(T)\subset\A$ the full subcategory formed
by the direct summands of infinite coproducts of copies of the object
$T\in\A$.
 Similarly, $\Prod(W)\subset\B$ denotes the full subcategory formed
by the direct summands of infinite products of copies of the object
$W\in\B$.
 Then $\Psi$ restricts to a category equivalence $\A_\inj\simeq\Prod(W)$, where
$\A_\inj\subset\A$ is the full subcategory of injective objects in~$\A$.
 One can in fact construct the functor $\Psi\:\A\rarrow\B$ as the unique
left exact extension of the additive embedding functor
$\A_\inj\simeq\Prod(W)\rarrow\B$.
Similarly, $\Phi\:\B\rarrow\A$ restricts to an equivalence
$\B_\proj\simeq\Add(T)$, where $\B_\proj\subset\B$ is the full subcategory
of projective objects in $\B$. Moreover, $\Phi$ is the unique right
exact extension of the additive embedding functor
$\B_\proj\simeq\Add(T)\rarrow\A$.

 The following lemma is inspired by, and should be compared with,
the results of the theory developed in the papers~\cite{AC,Ba0}
(cf.\ Corollary~\ref{tilting-objects-and-tilting-modules}
and Remark~\ref{cotilting-modules-remark}).

\begin{lem} \label{tilting-cotilting-classes-lemma}
\textup{(a)} The full subcategory\/ $\E\subset\A$ consists precisely
of all the objects\/ $E\in\A$ for which there exists an exact sequence
of the form
\begin{equation} \label{tilting-class-generated}
 T^{(I_n)}\lrarrow \dotsb\lrarrow T^{(I_2)}\lrarrow T^{(I_1)}\lrarrow E
 \lrarrow0
\end{equation}
in the abelian category~$\A$, where $I_1$,~\dots, $I_n$ are some sets.
\par
\textup{(b)} The full subcategory\/ $\E\subset\B$ consists precisely
of all the objects\/ $E\in\B$ for which there exists an exact sequence
of the form
\begin{equation} \label{cotilting-class-cogenerated}
 0\lrarrow E\lrarrow W^{I_1}\lrarrow W^{I_2}\lrarrow\dotsb\lrarrow
 W^{I_n}
\end{equation}
in the abelian category~$\B$, where $I_1$,~\dots, $I_n$ are some sets.
\end{lem}

\begin{proof}
 It suffices to prove part~(a).
 Since the full subcategory $\E\subset\A$ is coresolving,
the coresolution dimension of the object $A=\ker(T^{(I_n)}\to
T^{(I_{n-1})})\in\A$ does not exceed~$n$ by
Lemma~\ref{tilting-class-coresolving-finite-dim},
and the objects $T^{(I)}\in\A$ belong to~$\E$ by the condition~(ii),
it follows from (the dual of) \cite[Proposition 2.3(1)]{St}
or~\cite[Corollary~A.5.2]{Pcosh} that the object~$E$
in the exact sequence~\eqref{tilting-class-generated} belongs to~$\E$.

 Conversely, given an object $E\in\E\subset\A$, consider
the object $E=\Psi(E)\in\E\subset\B$.
 Since the object $T\in\B$ is a projective generator, an exact
sequence of the form
\begin{equation} \label{projective-resolution-e-b}
 0\lrarrow B\lrarrow
 T^{(I_n)}\lrarrow \dotsb\lrarrow T^{(I_2)}\lrarrow T^{(I_1)}\lrarrow
 E\lrarrow0
\end{equation}
exists in the abelian category~$\B$.
 By Lemma~\ref{cotilting-class-resolving-finite-dim}(a), we have
$T^{(I)}\in\E$ and $B\in\E$, and moreover, all the objects of
cocycles in the exact sequence~\eqref{projective-resolution-e-b}
belong to $\E$, so~\eqref{projective-resolution-e-b} is an exact
sequence in the exact category $\E\subset\B$.
 Applying to~\eqref{projective-resolution-e-b} the exact functor
$\Phi\:\E\rarrow\A$, we obtain the desired exact
sequence~\eqref{tilting-class-generated} in the abelian category~$\A$
(and in fact, even in the exact category
$\E\subset\A$; so the exact sequence~\eqref{tilting-class-generated}
can be chosen in such a way that the object
$A=\ker(T^{(I_n)}\to T^{(I_{n-1})})$ belongs to~$\E$).
\end{proof}

\begin{rem}
 The above argument also shows that any object $E$ of the exact
category $\E=\A\cap\B$ admits an arbitrarily long, and even
an infinite left resolution 
\begin{equation} \label{infinite-tilting-resolution}
 \dotsb\lrarrow T^{(I_i)}\lrarrow \dotsb\lrarrow T^{(I_2)}
 \lrarrow T^{(I_1)}\lrarrow E\lrarrow0
\end{equation}
by copowers of the tilting object $T$ in the abelian
category~$\A$. We obtain such a resolution simply by applying the functor $\Phi\:\B\rarrow\A$ to a corresponding projective resolution of $E$ in the abelian category $\B$.
\end{rem}

\begin{lem} \label{tilting-cotilting-classes-co-product-closed}
\textup{(a)} The full subcategory\/ $\E\subset\A$ is closed under
infinite coproducts in the abelian category\/~$\A$. \par
\textup{(b)} The full subcategory\/ $\E\subset\B$ is closed under
infinite products in the abelian category\/~$\B$.
\end{lem}

\begin{proof}
 It suffices to prove part~(a), as part~(b) is the dual assertion.
In fact, the assertion of part~(a) follows immediately from
Lemma~\ref{tilting-cotilting-classes-lemma}(a), but we prefer to give
a direct argument as well.

 Notice that for any t\+structure $(\D^{\le0},\.\D^{\ge0})$ on
a triangulated category $\D$ the full subcategory $\D^{\le0}
\subset\D$ is closed under infinite coproducts in $\D$ (those 
infinite coproducts of objects from $\D^{\le0}$ that exist in~$\D$).
 Furthermore, infinite coproducts of objects of $\A$ taken in $\A$ are
at the same time their coproducts in $\D^\b(\A)$, as it was explained
in the beginning of the proof of Proposition~\ref{tilting-heart}.
 Finally, one has
$$
 \E=\A\cap\B=\A\cap{}^T\D^{\b,\le0}\subset\D^\b(\A),
$$
because $\A\subset\D^{\b,\ge0}(\A)\subset{}^T\D^{\b,\ge0}$
according to~\eqref{d-ge-inclusions}.
 Since the full subcategory ${}^T\D^{\b,\le0}$ is closed under infinite
coproducts in $\D^\b(\A)$, we are done.
\end{proof}

\begin{rem} \label{E-is-AB4-and-AB4*}
The proof of Lemma~\ref{tilting-cotilting-classes-co-product-closed}
reveals in fact more: $\E$ is closed under infinite coproducts in $\D^\b(\A)$.
Dually, $\E$ is closed under infinite products in $\D^\b(\B) \simeq \D^\b(\A)$.
Thus, in particular, $\E$ is closed under both the infinite coproducts and
products in both $\A$ and~$\B$.
Moreover, both the infinite coproducts and infinite products of
exact sequences in $\E$ remain exact (since the same is true for triangles
by~\cite[Proposition 1.2.1]{Nee01}), despite the fact that products
may not be exact in $\A$ and coproducts may not be exact in $\B$.
\end{rem}

 Now we can prove the main results of the section.
 The case of the conventional unbounded derived categories
($\st=\varnothing$, i.~e., $\D(\A)\simeq\D(\E)\simeq\D(\B)$)
can be also found in~\cite[Theorem~1.7 and Proposition~2.3]{FMS}.

\begin{thm} \label{exotic-derived-resolution-equivalences}
\textup{(a)} For every derived category symbol\/ $\st=\b$, $+$, $-$,
$\varnothing$, $\abs+$, $\abs-$, $\co$, or\/~$\abs$, the triangulated
functor\/ $\D^\st(\E)\rarrow\D^\st(\A)$ induced by the exact
embedding functor\/ $\E\rarrow\A$ is an equivalence of triangulated
categories.
 The same is true for\/ $\st=\ctr$ whenever infinite products are exact
in the abelian category\/~$\A$. \par
\textup{(b)} For every derived category symbol\/ $\st=\b$, $+$, $-$,
$\varnothing$, $\abs+$, $\abs-$, $\ctr$, or\/~$\abs$, the triangulated
functor\/ $\D^\st(\E)\rarrow\D^\st(\B)$ induced by the exact
embedding functor\/ $\E\rarrow\B$ is an equivalence of triangulated
categories.
 The same is true for\/ $\st=\co$ whenever infinite coproducts are exact
in the abelian category\/~$\B$.
\end{thm}

\begin{proof}
 The result of part~(b) is provided by~\cite[Proposition~A.5.6]{Pcosh},
and part~(a) is a dual assertion. 
 (The case $\star=\varnothing$ can be found
in~\cite[Proposition~13.2.6]{KS}.)
 However, we provide a more detailed argument concerning part~(b) for
the reader's convenience.

All the conventional or exotic derived categories of $\B$ are defined as
Verdier quotients $\D^\st(\B) = \Hot^\st(\B)/\Acycl^\st(\B)$, where
$\Hot^\st(\B)$ is the correspondingly bounded homotopy category of complexes
over $\B$ and $\Acycl^\st(\B)$ is a thick subcategory of acyclic complexes
which contains totalizations of all short exact sequences of complexes.
The derived categories for $\E$ are defined analogously,
$\D^\st(\E) = \Hot^\st(\E)/\Acycl^\st(\E)$.
In order to prove that the inclusion $\Hot^\st(\E) \subset \Hot^\st(\B)$
induces an equivalence $\D^\st(\E)\simeq\D^\st(\B)$, it suffices to show
(see~\cite[Lemma 1.6]{PKoszul}) that
\begin{enumerate}
\item[($\alpha$)] each $B^\bu \in \Hot^\st(\B)$ admits a morphism
$f\:E^\bu\rarrow B^\bu$ with $E^\bu\in\Hot^\st(\E)$ and such that
the cone of $f$ is in $\Acycl^\st(\B)$;
\item[($\beta$)] the equality $\Acycl^\st(\E) = 
\Hot^\st(\E)\cap\Acycl^\st(\B)$ holds.
\end{enumerate}
 Condition~($\alpha$) is satisfied by~\cite[Lemma A.3.3]{Pcosh}.
 Indeed, since the $\E$-resolution dimension of objects of $\B$ is
uniformly bounded by
Lemma~\ref{cotilting-class-resolving-finite-dim},
each $B^\bu \in \Hot^\st(\B)$ admits $f\:E^\bu\to B^\bu$
whose cone is in the smallest thick subcategory of $\Hot^\st(\B)$
generated by totalizations of suitably bounded short exact sequences
of complexes over $\B$.

 Regarding ($\beta$), for $\st=\b$, $+$, $-$, or $\varnothing$,
this is a direct consequence of the fact that the full subcategory $\E$
is resolving in~$\B$, and that the resolution dimension is bounded
by a finite constant again.
 The remaining cases $\abs+$, $\abs-$, $\co$, $\ctr$, and\/~$\abs$ are
more involved.
% The main point is to prove that a totalization of a short exact
%sequence of complexes over $\B$ belongs to $\Acycl^\st(\E)$.
 Denoting by $\E_d\subset\B$ the full subcategory of all objects of
resolution dimension~$\le d$ with respect to~$\E$ (so that $\E_0=\E$
and $\E_n=\B$), one proves that $\Hot^\st(\E_{d-1})\cap\Acycl^\st(\E_d)
= \Acycl^\st(\E_{d-1})$ for all $1\le d\le n$.
 For this purpose, one shows that every complex from $\Acycl^\st(\E_d)$
is the cokernel of an admissible monomorphism of complexes from
$\Acycl^\st(\E_{d-1})$.
 The key observation is that a short exact sequence (of complexes)
in $\B$ has a finite resolution by short exact sequences (of complexes)
in~$\E$; see the proof of~\cite[Proposition 2.3]{St}.
 The assertion involving the contraderived category $\D^\ctr(\E)$
also uses Lemma~\ref{tilting-cotilting-classes-co-product-closed}(b).
 We refer to~\cite[Theorem 1.4]{EfimovP}
and~\cite[Theorem 7.2.2]{Psemi} for further technical details.
\end{proof}

\begin{cor} \label{conv-abs-derived-equivalence}
 For every derived category symbol\/ $\st=\b$, $+$, $-$,
$\varnothing$, $\abs+$, $\abs-$, or\/~$\abs$, there is a natural
triangulated equivalence of (conventional or absolute) derived
categories\/ $\D^\st(\A)\simeq\D^\st(\B)$
as in~\eqref{derived-equivalence} provided by the mutually inverse
derived functors\/ $\boR\Psi\:\D^\st(\A)\rarrow\D^\st(\B)$ and\/
$\boL\Phi\:\D^\st(\B)\rarrow\D^\st(\A)$.
 The similar equivalence of the coderived or contraderived categories
($\st=\co$ or~$\ctr$) holds whenever the infinite coproducts or
infinite products, respectively, are exact in both\/ $\A$ and\/~$\B$.
\end{cor}

\begin{proof}
 Compare Theorem~\ref{exotic-derived-resolution-equivalences}(a) with
Theorem~\ref{exotic-derived-resolution-equivalences}(b).
\end{proof}

We conclude the section with a rather concrete example of a tilting equivalence in commutative algebra, which is closely related to more classical equivalences of Matlis and of Greenlees and May (see~\cite{Pmc,PMat}) and which we later extend in Example~\ref{Matlis-1-tilting-example} in the case of
commutative Noetherian rings of Krull dimension one.

Many examples of big $n$\+tilting and $n$\+cotilting modules are
also available in the literature; see~\cite[Chapters~13--17]{GTbook}.
In fact, one of them provides a more general (but less explicit from the computational point of view) way to assign a $1$\+tilting module 
to a multiplicative subset $S$ in a commutative domain;
see~\cite{Fac0,Fac1} and~\cite[Example~13.4]{GTbook}.
Some other examples of big $1$\+tilting objects in locally Noetherian
Grothendieck abelian categories are also discussed in the recent
paper~\cite{AK}.

\begin{ex} \label{torsion-modules-1-tilting-example}
Let $R$ be a commutative ring and $S\subset R$ be a multiplicative
subset such that all elements of $S$ are nonzero-divisors in $R$ and
the projective dimension of the $R$\+module $S^{-1}R$ does not
exceed~$1$.
The latter condition is satisfied for example if $S=\{1,s,s^2,\dotsc\}$
for $s\in R$, or, by~\cite[Corollaire II.3.2.7]{RG71}, for an arbitrary
multiplicative set $S$ if $R$ is Noetherian of Krull dimension
at most one.
% the second condition holds automatically,
%and the first condition is satisfied if and only if $s$~is not
%a zero-divisor in~$R$.
	
An $R$\+module $M$ is said to be \emph{$S$\+torsion} if for every
element $x\in M$ there exists $s\in S$ such that $sx=0$ in~$M$.
The maximal $S$\+torsion submodule of an $R$\+module $M$ is denoted
by $\Gamma_S(M)$; and the full subcategory of all $S$\+torsion
$R$\+modules is denoted by $R\modl_{S\tors}\subset R\modl$.
The category $\A=R\modl_{S\tors}$ is a Grothendieck abelian category
with an injective cogenerator $W=\Gamma_S(\Hom_\Z(R,\mathbb Q/\Z))$.
	
We claim that the $R$\+module $T=S^{-1}R/R$ is a $1$\+tilting object in 
the abelian category $\A=R\modl_{S\tors}$. One way to see that is using the fact that $S^{-1}R \oplus S^{-1}R/R$ is a $1$\+tilting $R$\+module
by~\cite[Theorem 14.59]{GTbook}, but we can also give a more
direct argument.
Indeed, $T$ has projective dimension at most~$1$ in $R\modl$ and hence also in $\A$ (since $\Ext^1_R(T,-)$ is right exact on~$\A$).
This implies condition~(i).
Condition~(ii), i.~e., the equality
$\Ext^1_R(S^{-1}R/R,\.(S^{-1}R/R)^{(I)})=0$ for every set $I$,
follows quickly from the fact that
\[ \Ext^1_R(S^{-1}R,\.S^{-1}R^{(I)})=
\Ext^1_{S^{-1}R}(S^{-1}R,\.S^{-1}R^{(I)})=0 \]
(see also~\cite[exact sequence~(III)]{PMat}).
In order to prove condition~(iii), we use the fact that,
by \cite[Theorem 6.6(a)]{PMat}, the canonical functor $\D(\A)\rarrow
\D(R\modl)$ induced by the embedding $\A \subset R\modl$ is fully 
faithful and its essential image is $\D_\A(R\modl)$, the full 
subcategory of complexes with $S$\+torsion cohomology.
Suppose now that $X^\bu\in \D(\A)$ is such that
$\Hom_{\D(\A)}(T,X^\bu[i])=0$ for all $i\in\Z$ or, equivalently,
such that $\boR\Hom_R(T,X^\bu)=0$.
An application of $\boR\Hom_R(-,X^\bu)$ to the short exact sequence
$0 \rarrow R \rarrow S^{-1}R \rarrow T \rarrow 0$ reveals that
\[ X^\bu \cong \Hom_R(R,X^\bu) \cong \boR\Hom_R(S^{-1}R,X^\bu). \]
In particular, the cohomology of $X^\bu$ consists of $S^{-1}R$-modules.
As the cohomology of $X^\bu$ was assumed to be $S$\+torsion,
$X^\bu$ is acyclic, as required.
This proves the claim.
	
The main object of interest for us is the heart
$\B={}^T\D^{\b,\le0}\cap{}^T\D^{\b,\ge0}\subset\D^\b(\A)$
of the tilting t\+structure for~$T$.
We know so far that $\B_\proj = \Add(S^{-1}R/R) \subset \A \subset
R\modl$.
It turns out that $\B$ is equivalent to the category of so-called
\emph{$S$\+contramodule} $R$\+modules, which is easiest defined as
\[ R\modl_{S\ctra} = \{ C \in R\modl \mid \Hom_R(S^{-1}R,C)=0=\Ext_R^1(S^{-1}R,C) \} \subset R\modl. \]
This is a full abelian subcategory of $R\modl$ (see, e.~g.,
\cite[Proposition~1.1]{GL91}).
A more intrinsic way to view this category, which will be discussed in detail
in \S\ref{contramodules-over-top-rings-subsecn} and
Section~\ref{big-tilting-module-secn}, is as follows.
The ring
$$
\fR=\Hom_R(S^{-1}R/R,\.S^{-1}R/R)=\varprojlim_{s\in S}R/sR
$$
is naturally a complete separated topological ring with a base
$\{s\fR\mid s\in S\}$ of neighborhoods of $0$
(see \cite[Proposition~3.2 and Theorem~2.5]{PMat}).
As such, infinite families $(r_i)_{i \in I}$ of elements of $\fR$ which converge to zero in this topology are naturally summable, i.~e., $\sum_{i \in I} r_i \in \fR$ is naturally defined.
The $S$\+contramodule $R$\+modules are, roughly speaking, precisely those $R$\+modules $C$ for
which there is a similar and naturally unique way to define all infinite summations
$\sum_{i \in I} r_i c_i \in C$, where $c_i \in C$ for all $i \in I$.
	
Using results from~\cite{PMat}, one quickly sees that $\Hom_R(T,-)$ induces an equivalence from $\B$ to $R\modl_{S\ctra}$. Indeed, as both $\B$ and $R\modl_{S\ctra}$ have enough projective objects (see the discussion at the very end of~\cite[Section 3]{PMat}), it suffices to show that $\Hom_R(T,-)$ induces an equivalence between the full subcategories of projective objects. Since $\Hom_R(T,S^{-1}R^{(I)}) = 0 = \Ext^1_R(T,S^{-1}R^{(I)})$ for each set $I$, we have
\[ \Hom_R(T,T^{(I)}) \cong \Ext^1_R(T,R^{(I)}), \]
and retracts of such $R$-modules are precisely the projective objects of $R\modl_{S\ctra}$ by~\cite[Section~3]{PMat}. Hence the restriction
\[ \Hom_R(T,-)\colon \B_\proj = \Add(T) \rarrow (R\modl_{S\ctra})_\proj \]
is essentially surjective. To see that it is also fully faithful, one applies $-\otimes_R T$ to the morphisms $\delta_{S,R^{(I)}}\: R^{(I)} \to \Ext^1_R(T,R^{(I)}) \cong \Hom_R(T,T^{(I)})$ from~\cite[Lemma~1.6]{PMat}. Since $X\otimes_R T=0=\Tor_1^R(X,T)$ for each $S^{-1}R$-module $X$, one obtains an isomorphism
\[ T^{(I)} \overset{\simeq}\rarrow \Hom_R(T,T^{(I)}) \otimes_R T, \]
which can be directly checked to be the inverse of the canonical evaluation morphism. It follows that $-\otimes_R T\: (R\modl_{S\ctra})_\proj \to \B_\proj$ is an inverse equivalence to $\Hom_R(T,-)$.
	
Hence, we can directly apply Corollary~\ref{conv-abs-derived-equivalence} to recover the derived equivalences from \cite[Theorem~4.6 and Corollary~6.7]{PMat},
\[ \boR\Hom_R(T,-)\: \D^\st(R\modl_{S\tors}) \rightleftarrows \D^\st(R\modl_{S\ctra}) {\;\:\!\!\!-\otimes_R^\boL T} \]
as a special case of our tilting equivalences.
The corresponding $1$\+cotilting object $W\in R\modl_{S\ctra}$ is
computed as
$$
W=\Hom_R(S^{-1}R/R,\.\Gamma_S(\Hom_\Z(R,\mathbb Q/\Z)))=
\Hom_\Z(S^{-1}R/R,\.\mathbb Q/\Z).
$$
%The cotilting heart $\A\subset\D^\b(R\modl_{S\ctra})$ is identified
%with the abelian category $R\modl_{S\tors}$ by the functor
%assigning to an object $A\in\A\subset\D^\b(R\modl_{S\ctra})$
%the $R$\+module
%$$
%S^{-1}R/R\ot^\boL_RA\in R\modl_{S\tors}\subset
%R\modl\subset\D(R\modl)
%$$
%(where $\D^\b(R\modl_{S\ctra})$ is a full subcategory in
%$\D^\b(R\modl)$ by~\cite[Theorem~6.6(b)]{PMat}).
%Our Corollary~\ref{conv-abs-derived-equivalence} provides
%a simple proof of~\cite[Theorem~7.6]{PMat} in the case of
%a multiplicative subset $S\subset R$ consisting of some
%nonzero-divisors in~$R$.
\end{ex}

\Section{Abelian Categories with a Projective Generator}
\label{projective-generator-secn}

In this section we discuss in detail cocomplete abelian categories
with a projective generator, as they form one end of our tilting-cotilting
correspondence from Section~\ref{tilting-cotilting-secn}
and their description as infinitary module-like structures,
which makes the discussion further much more concrete,
does not seem to be widely known.
%We build on results from~\cite{PR} and the references there.

\subsection{Modules and monads}
\label{modules-and-monads-subsecn}

Following the exposition in \cite[\S2.1]{Prev}, we 
first explain how ordinary modules over a (nontopological)
ring $R$ can be viewed as algebras over a monad
(see also \cite[Exercise VI.4.2]{MacLane}).
For a left $R$\+module $M$, we can define
for each $n\ge 0$ and elements $r_1,\dots,r_n\in R$
an $n$\+ary operation $M^n \rarrow M$ given by
\[ (m_1,\dots,m_n) \longmapsto \sum_{i=1}^n r_im_i \]
(using the convention that an empty sum evaluates to zero for $n=0$).
In fact, these are all the finitary operations on $M$ which can
be defined in general using the $R$\+module structure only. In our context,
it is useful to complement the picture by possibly infinitary operations
$M^X \rarrow M$ of the form $(m_x)_{x\in X} \mapsto \sum_{x\in X}r_xm_x$
with $\sum_{x\in X}r_xx\in R[X]$,
where $X$ is a set and $R[X]$ is the set of all formal $R$\+linear combinations
$\sum_{x\in X}r_xx$ such that only finitely many of the elements $r_x\in R$
are nonzero.
In fact, $R[X]$ can be viewed just as a different notation for the direct sum
$R^{(X)}$, and the assignment $X\mapsto R[X]$ is functorial: If $f\:X\to Y$
is a map of sets, then $R[f]\:R[X]\rarrow R[Y]$ sends $\sum_{x\in X}r_xx\in R[X]$
to $\sum_{x\in X}r_xf(x)\in R[Y]$ (with the obvious interpretation).

With this notation, an $R$\+module structure can be encoded by the map
$\alpha_M\:R[M]\rarrow M$, which sends each formal linear combination
$\sum_{m\in M}r_mm\in R[M]$ to its evaluation in $M$. In order to encode
the axioms for an $R$\+module in terms of $\alpha_M$, we will consider
two other maps $\mu_X\:R[R[X]]\rarrow R[X]$ and $\varepsilon_X\:X\rarrow R[X]$,
which are functorial in $X$. The map $\mu_X$ is
the ``opening of parentheses'' map, producing a formal linear
combination out of a formal linear combination of formal linear
combinations. This uses the multiplication in the ring $R$.
The map $\varepsilon_X$ interprets an element $x_0\in X$ as
the formal linear combination $\sum_xr_xx\in R[X]$ with $r_x=1$ for $x=x_0$
and $r_x=0$ for $x\ne x_0$.

Now we can be more precise: To specify a left $R$\+module structure on $M$ is the same
as to give a map $\alpha_M\:R[M]\rarrow M$ which satisfies the following axioms:
\begin{enumerate}
	\item The two maps $\alpha_M\circ\mu_M$ and $\alpha_M\circ R[\alpha_M]\:R[R[M]]\rarrow M$ coincide (this encodes the associativity and distributivity of the left $R$\+action).
	\item The maps $M \overset{\varepsilon_M}\rarrow R[M] \overset{\alpha_M}\rarrow M$ compose to the identity on $M$ (this encodes the unitality of the left $R$\+action).
\end{enumerate}

Needless to say, we have just interpreted the category of left $R$\+modules
as the Eilenberg--Moore category of the monad $(R[-]\:\Sets\rarrow\Sets, \mu, \varepsilon)$
on the category of sets (see \cite[\S\S VI.1, VI.2 and VI.8]{MacLane}).
Although this may seem as an overkill for defining ordinary modules, only a small
modification of the above will allow us to conveniently encode associative,
distributive and unital infinitary actions of topological rings.

\subsection{Contramodules over topological rings}
\label{contramodules-over-top-rings-subsecn}

As we have alluded in Example~\ref{torsion-modules-1-tilting-example}, one important and explicit source of abelian categories with coproducts and a projective generator are topological rings and contramodules over them~\cite{Prev,PR}. The category of contramodules can be viewed as a well-behaved (in particular, abelian) replacement of the (in general ill-behaved) category of complete separated topological modules. We again follow the exposition in \cite[\S2.1]{Prev} here.

To this end, let $\fR$ be a complete and separated topological associative ring
with a base of neighborhoods of zero formed by open right ideals
(a filter $\{\fU_j\}_{j\in J}$ of right ideals of $\fR$ makes $\fR$
a topological ring if and only if for each $j\in J$ and
$r\in\fR$, there exists $j'\in J$ such that $r\cdot \fU_{j'}\subset \fU_j$,
\cite[\S VI.4]{Sten}).
Then we denote by $\fR[[X]]$ the set of all the
infinite formal linear combinations $t=\sum_{x\in X}r_xx$ of elements of $X$ with
the coefficients $r_x\in\fR$ for which the family of elements $(r_x)_{x \in X}$
converges to zero in the topology on~$\fR$.
The latter condition means that for every open subset
$0\in \fU\subset\fR$ one must have $r_x\in\fU$ for all but a finite subset
of indices $x\in X$.

We in fact again obtain a functor $\fR[[-]]\:\Sets\rarrow\Sets$ where
given a map of sets $f\:X\rarrow Y$, the induced map
$\fR[[f]]\:\fR[[X]]\rarrow\fR[[Y]]$ sends $\sum_{x\in X}r_xx$ to
$\sum_{x\in X}r_xf(x)$. In more pedestrian terms, the coefficient
at each $y\in Y$ in the image satisfies $s_y=\sum_{x\in f^{-1}(y)} r_x$,
where we use infinite sums of converging families of elements
taken in the topology of~$\fR$.
%The monadic unit map $X\rarrow\fR[[X]]$ takes an element $x_0\in X$ to
%the formal linear combination $\sum_xr_xx$ with $r_x=1$ for $x=x_0$
%and $r_x=0$ for $x\ne x_0$.
%The monadic multiplication $\fR[[\fR[[X]]]]\rarrow\fR[[X]]$ is
%the ``opening of parentheses'' map, producing a formal linear
%combination out of a formal linear combination of formal linear
%combinations.
As in~\S\ref{modules-and-monads-subsecn}, we denote by $\mu_X\:\fR[[\fR[[X]]]]\rarrow\fR[[X]]$ and
$\varepsilon_X\:X\rarrow\fR[[X]]$ the ``opening of parentheses'' map
and the map which interprets each $x\in X$ as a formal
linear combination with
only one nonzero coefficient, respectively.
The former of these is constructed using the multiplication
in the ring $\fR$ and the infinite summation in the topology
of $\fR$; the condition that open right ideals form a base of the
topology of $\fR$ guarantees the convergence. One can check that
$(\fR[[-]], \mu, \varepsilon)$ is again a monad on
$\Sets$.

A \emph{left\/ $\fR$\+contramodule} is, by definition, a map
$\alpha_\fC\:\fR[[\fC]]\rarrow \fC$ of sets which satisfies the direct analogues
of (1) and~(2) in~\S\ref{modules-and-monads-subsecn}:
\begin{enumerate}
\item the two maps $\alpha_\fC\circ\mu_\fC$ and $\alpha_\fC\circ \fR[[\alpha_\fC]]\:\fR[[\fR[[\fC]]]]\rarrow \fC$ coincide and
\item the maps $\fC\overset{\varepsilon_\fC}\rarrow \fR[[\fC]] \overset{\alpha_\fC}\rarrow \fC$ compose to the identity.
\end{enumerate}
In other words, the map $\alpha_\fC$ tells us how each infinite linear combination $\sum_{c\in\fC}r_cc$ with the coefficients $r_c$ converging to zero in $\fR$ evaluates in $\fC$, and we also require suitable distributivity, associativity and unitality.
Each left $\fR$\+module $C$ has a natural topology induced by the topology
of $\fR$, and when this topology is complete and separated, $C$ becomes naturally an $\fR$\+contramodule; but the converse is not true in general
(see \cite[Remarks~1.2.2--4]{Pweak} and~\cite[\S1.5]{Prev}).

A morphism $f\colon \fC\rarrow\fD$ of contramodules is a map which satisfies $\alpha_\fD\circ\fR[[f]] = f\circ\alpha_\fC$ (or in other words, $f(\sum_{x\in X} r_xc_x) = \sum_{x\in X} r_xf(c_x)$ for each set $X$, formal linear combination $\sum_{x\in X} r_xx\in\fR[[X]]$ and elements $c_x\in\fC$). We will denote the set of homomorphisms from $\fC$ to $\fD$ by $\Hom^\fR(\fC,\fD)$ and the category of all $\fR$\+contramodules by $\fR\contra$. Moreover, as clearly $\fR[X]\subseteq \fR[[X]]$ for each $X$, we have a forgetful functor
\[ U_\fR\:\fR\contra\lrarrow\fR\modl \]
from $\fR$\+contramodules to the category of ordinary $\fR$\+modules.
In particular, the underlying set of each contramodule is naturally an abelian group, and in fact the category of contramodules is abelian, where the kernels, images and cokernels are computed in the usual way at the level of the underlying $\fR$-modules.

For each set $X$, there is a natural left contramodule structure
on $\fR[[X]]$ given by $\alpha_{\fR[[X]]} = \mu_X$.
Such contramodules are called \emph{free}, since any map
of sets $f\:X\rarrow\fC$ to a contramodule $\fC$ uniquely
extends to a map of contramodules
\[
\fR[[X]]\rarrow\fC, \qquad
\sum_{x\in X}r_xx \longmapsto \sum_{x\in X}r_xf(x) \]
(this map is in fact equal to $\alpha_\fC\circ \fR[[f]]$).
The category $\fR\contra$ is cocomplete, since each contramodule has a presentation of the form $\fR[[Y]]\rarrow\fR[[X]]\rarrow\fC\rarrow 0$ and $\fR[[\coprod_{i\in I}X_i]]$ is a coproduct of $(\fR[[X_i]])_{i\in I}$ in $\fR\contra$ (cf.\ the proof of
Proposition~\ref{tilting-heart}). It is also complete and limits
of $\fR$-contramodules are created by the limits of the
the underlying sets, as with usual modules (in particular, $U_\fR$ preserves limits).
Assuming that the ring $\fR$ has a base of neighborhoods of zero
of cardinality less than~$\kappa$, the category $\fR\contra$ is
a locally $\kappa$\+presentable abelian category.
The free left $\fR$\+contramodule with one generator
$\fR=\fR[[\{0\}]]$ is a $\kappa$\+presentable
%(or, equivalently, abstractly $\kappa$\+small)
projective generator of $\fR\contra$.
The projective objects in $\fR\contra$ are precisely the direct
summands of free $\fR$\+contramodules.

So far, we know that the forgetful functor $U_\fR$ is exact, product-preserving and faithful, but it is not completely understood when it is full.
This holds for example for the adic completions of
Noetherian rings by centrally generated
ideals~\cite[Theorem~B.1.1]{Pweak}, \cite[Theorem~C.5.1]{Pcosh}.
More general results of this kind can be found in~\cite[Theorem~1.1]{Psm},
\cite[Section~3]{Pperp} and~\cite[Section 6]{Pflat}.
Using Theorem~\ref{kappa-ary-dual-gabriel-popescu},
we will, however, prove in \S\ref{morita-subsecn}
that the forgetful functor is always full
if we pass from $\fR$ to a (in a suitable sense)
Morita equivalent topological ring.
%Further classes of examples coming from tilting theory are provided below in
%Propositions~\ref{fullyf-forgetful-functor-contramodules}
%and~\ref{good-tilting-module}, and Theorems~\ref{lwfg-addm-theorem}
%and~\ref{closed-functor-addm-theorem}.

%In addition to all the properties listed above (which are essentially
%common to all the categories of models of additive $\kappa$\+ary
%algebraic theories), the category $\fR\contra$ has a more special
%property that,
The category $\fR\contra$ has exact products, but in general neither exact coproducts nor even directed colimits.
However, it is still true that
for every family of projective objects $P_\alpha\in
\fR\contra$, the natural map $\coprod_\alpha P_\alpha\rarrow
\prod_\alpha P_\alpha$ is a monomorphism in this category.
This follows from the observation that the map $P^{(X)}\rarrow P^X$,
where $P=\fR$ denotes the standard projective generator of $\fR\contra$, is
a monomorphism for every set $X$, because the obvious map of sets
$\fR[[X]]\rarrow\fR^X$ is injective.

\subsection{Additive monads}
\label{additive-monads-subsecn}

We have seen in the previous subsection that the categories of contramodules over topological rings provide examples of 
abelian categories with coproducts and a projective generator, and that they can be viewed as the Eilenberg--Moore categories of monads.

Here we show that the latter generalizes to any abelian category
$\B$ with coproducts and a projective generator $P$.
We follow the introduction to~\cite{PR} (see also the references therein).
First of all, note that we have natural isomorphisms
\[ \Hom_\Sets(X,\Hom_\B(P,B)) = \Hom_\B(P,B)^X\cong \Hom_\B(P^{(X)},B) \]
for each $X\in\Sets$ and $B\in\B$, and thus the functor
$\Psi_P = \Hom_\B(P,-)\:\B\rarrow\Sets$ has a left adjoint $\Phi_P\:\Sets\rarrow\B$
given by $\Phi_P(X) = P^{(X)}$. The action of $\Phi_P$ on a morphism $f\:X\rarrow Y$
is specified by the equalities $\Phi_P(f)\iota_x = \iota'_{f(x)}$ for each $x\in X$,
where $\iota_x\:P\rarrow P^{(X)}$ and $\iota'_{f(x)}\:P\rarrow P^{(Y)}$ are
the coproduct inclusions.

Now we use the standard fact that each adjunction induces a monad
\cite[\S VI.1]{MacLane}. In our case, we obtain a monad $(\boT,\mu,\varepsilon)$
on $\Sets$, where
\[ \boT = \Psi_P\circ\Phi_P\: X \;\longmapsto\; \Hom_\B(P,P^{(X)}). \]
The monadic unit $\varepsilon$ is given by the unit of the adjunction
$X\rarrow\Hom_\B(P,P^{(X)})$ which sends $x\in X$ to the coproduct inclusion $\iota_x$.
The monadic composition $\mu\:\boT\circ\boT\rarrow\boT$ is obtained as
$\mu=\Psi_P\circ\zeta\circ\Phi_P$, where $\zeta_B\:P^{(\Hom_\B(P,B))}\rarrow B$
is the counit of the adjunction. More specifically, $\zeta_B$ is the obvious canonical
map and we have
\[ \mu_X=\Psi_P(\zeta_{P^{(X)}})\: \Hom_\B(P,P^{(\Hom_\B(P,P^{(X)}))})\rarrow \Hom_\B(P,P^{(X)}). \]
%by applying $\Psi_P$ to the map
%$\zeta_{P^{(X)}}\:P^{(\Hom_\B(P,P^{(X)}))}\rarrow P^{(X)}$.

\begin{ex} \label{ordinary-mod-example}
In the case where $\B=R\modl$ and $P=R$, we obtain the monad
$(R[-],\mu,\varepsilon)$ from~\S\ref{modules-and-monads-subsecn}.
If $\fR$ is a topological ring, $\B=\fR\contra$ and $P=\fR$,
we precisely recover the monad $(\fR[[-]],\mu,\varepsilon)$
from~\S\ref{contramodules-over-top-rings-subsecn}.
%Consider the trivial situation where $\B=R\modl$ is an ordinary module
%category and $P=R$. If $X$ is a finite set with $n=\lvert X\rvert$ and
%$t=(r_1,\ldots,r_n)\in\boT(X)=R^n$, then $t$ acts on $M$ as
%$M^n \to M$, $(m_1,\dots,m_n)\mapsto\sum_{i=1}^n r_im_i$.
\end{ex}

Our next aim it to prove
that the category $\B$ is equivalent to the
Eilenberg--Moore category $\boT\modl$ of this monad. 
We will use the term \emph{$\boT$\+module} for objects of the
Eilenberg--Moore category as we really view them as a generalization
of modules over a ring.

To this end, note that $\boT$-modules can be really interpreted as module-like
structures as follows (see also e.g.\ \cite[\S1]{PStAB5}).
For any set $X$, any fixed element $t\in\boT(X)$
induces an $X$\+ary operation on (the underlying sets of) all
$\boT$\+modules. More specifically, if $\alpha_M\:\boT(M)\rarrow M$ is
a $\boT$\+module and $t\in\boT(X)$, then $t$ acts on $M$ via
\begin{align*}
t_M\: M^X &\rarrow M, \\
m = (m_x)_{x\in X} \;&\longmapsto\; \alpha_M(\boT(m)(t)),
\end{align*}
where $m$~is viewed as a map of sets $X\rarrow M$ and $\boT(m)$ is the induced
map $\boT(X)\rarrow\boT(M)$. One can check that a map $f\: M\rarrow N$ between
$\boT$-modules is a homomorphism in $\boT\modl$ if and only if
\[
\begin{CD}
M^X @>{t_M}>> M   \\
@V{f^X}VV @VV{f}V \\
N^X @>{t_N}>> N
\end{CD}
\]
commutes for each set $X$ and $t\in\boT(X)$ (to prove the if-part,
put $X=M$ and $m=1_M \in M^M$).

The last paragraph is true for any monad, but for our particular monad $\boT\:X\longmapsto\Hom_\B(P,P^{(X)})$, the
elements $(\Delta\: P\rarrow P\oplus P)\in\boT(\{0,1\})$ (the diagonal map),
$(-1_P\:P\rarrow P)\in\boT(\{0\})$ and $(P\rarrow0)\in\boT(\varnothing)$ induce
on each $\boT$\+module operations $+\:M\times M\rarrow M$, $-\:M\rarrow M$ and
$0\in M$, respectively, which satisfy the axioms of an abelian group. Thus,
any $M\in\boT\modl$ carries a structure of an abelian group and it is not
difficult to check that $t_M\:M^X\rarrow M$ is a homomorphism of abelian groups
for each set $X$ and $t\in\boT(X)$. In particular, $\boT\modl$ is an additive category
and, in fact, it is abelian, where kernels, images and cokernels are computed
in the usual way at the level of underlying abelian groups
(see \cite[Lemma 1.1]{Pperp}).

Note also that $\boT(X)$ naturally carries the structure of a \emph{free} $\boT$\+module
$\alpha_{\boT(X)}=\mu_X\:\boT(\boT(X))\rarrow\boT(X)$ in that
any map of sets $f\:X\rarrow M$ to a $\boT$-module $M$ uniquely
extends to a map $\boT(X)\rarrow M$ in $\boT\modl$
(which is given by $\alpha_M\circ\boT(f)$). In particular, $\boT\modl$ has coproducts
(the argument is the same as in \S\ref{contramodules-over-top-rings-subsecn}),
$\boT(\{0\})$ is a projective generator, and $\B_\proj\subset\B$ consists precisely
of retracts of free $\boT$\+modules.

Now we prove the promised equivalence of $\B$ and $\boT\modl$
as a consequence of the following more general proposition.

\begin{prop} \label{copowers-kleisli}
For any category\/ $\A$ with set-indexed coproducts
and an object $M\in\A$, there is a monad $\boT\:\Sets\rarrow\Sets$ such that $\boT(X)=\Hom_\A(M,M^{(X)})$
and the full subcategory consisting of all
the objects $M^{(X)}$, $X\in\Sets$ in the category\/ $\A$ is 
equivalent to the full subcategory consisting of all the free
$\boT$\+modules\/ $\boT(X)$, $X\in\Sets$ in the category 
of\/ $\boT$\+modules\/ $\boT\modl$.
%\/~$\B$.
\end{prop}

\begin{proof}
As before, the functor $\Phi_M\:\Sets\rarrow\A$ assigning
the object $M^{(X)}$ to a set $X$ is left adjoint to
the functor $\Psi_M\:\A\rarrow\Sets$ assigning the set $\Hom_\A(M,N)$
to an object $N\in\A$. The monad $\boT_M\:\Sets\rarrow\Sets$ is
the composition of these two adjoint functors, $\boT=\Psi_M\circ\Phi_M$.
Hence, the functor $\Psi_M$ lifts naturally to a functor taking values
in the category of $\boT$\+modules---if $N\in\A$ and $\zeta_N\:M^{(\Hom_\A(M,N))}{}\rarrow N$ is the counit of adjunction, then the $\boT$\+module structure is given by $\Hom_\A(M,\zeta_N)\: \boT(\Hom_\A(M,N)) \rarrow \Hom_\A(M,N)$.
% The functor $\Psi_M\:\A\rarrow\B$, \ $\Psi_M(N)=\Hom_\A(M,N)$
Moreover, the lifted functor $\Psi_M\:\A\rarrow\boT\modl$ takes the object $M^{(X)}$ to the free $\boT$\+module $\boT(X)$.
It remains to compare the sets of morphisms in the two categories
in order to conclude that the restriction of the functor $\Psi_M$
%is an equivalence between the full subcategory of all the copowers
%of $M$ in $\A$ and the full subcategory of all the free
%$\boT$\+modules in~$\B$,
yields the desired equivalence:
\begin{equation*}
\Hom_\A(M^{(Y)},M^{(X)})=\Hom_\A(M,M^{(X)})^Y=\boT(X)^Y=
\Hom_\B(\boT(Y),\boT(X)).
\qedhere
\end{equation*}
%for any two sets $X$ and~$Y$.
\end{proof}

\begin{cor}\label{cor:B-as-Tmod}
Let $\B$ be an abelian category with set-indexed coproducts
and a projective generator $P$.
The functor $\Psi_P=\Hom_\B(P,-)\:\B\rarrow\boT\modl$ which
assigns to an object $B\in\B$ the $\boT$\+module
\[ \Hom_\B(P,\zeta_B)\: \boT(\Hom_\B(P,B)) = \Hom_\B(P,P^{(\Hom_\B(P,B))}) \rarrow \Hom_\B(P,B) \]
is an equivalence of categories.
\end{cor}

\begin{proof}
%Since $\Psi_P(P^{(X)}) = \boT(X)$ is a free $\boT$\+module, $\Psi_P$ induces an isomorphism
%\[ \Hom_\B(P^{(Y)},P^{(X)}) = \Hom_\B(P,P^{(X)})^Y = \boT(X)^Y = \Hom_{\boT\modl}(\boT(Y),\boT(X)) \]
%for each sets $X,Y$.
By adjoining direct summands on both sides of the equivalence from Proposition~\ref{copowers-kleisli}, we infer that $\Psi_P$ restricts to
an equivalence $\B_\proj\simeq\boT\modl_\proj$. Since an abelian category with enough projectives
is determined by its subcategory of projective objects (see e.g.~\cite[Proposition IV.1.2]{ARS96}) and
since $\Psi_P$ is exact, it follows that it is an equivalence $\B\rarrow\boT\modl$.
\end{proof}

\begin{rem}
The $\boT$\+module structure on $\Hom_\B(P,B)$ is easier to describe through
the operations induced by $t\in\boT(X)=\Hom_\B(P,P^{(X)})$, which are simply given by
\begin{align*}
\Hom_\B(P,B)^X=\Hom_\B(P^{(X)},B)&\rarrow\Hom_\B(P,B), \\
(m\:P^{(X)}\to B)&\longmapsto m\circ t.
\end{align*}
Morally the operations can be viewed as 
infinite $\Hom_\B(P,P)$-linear combinations of elements of $M$, but unlike in the case
of contramodules in~\S\ref{contramodules-over-top-rings-subsecn}, $t\in\Hom_\B(P,P^{(X)})$
might not be determined by its components $t_x\:P\rarrow P$ since the
canonical map $P^{(X)}\rarrow P^X$ might not be injective in $\B$
(e.g.\ in $\B=(k\modl)^\rop$ for a field $k$).
\end{rem}

%Generalizing the above discussion, let us consider an arbitrary
%category $\A$ with set-indexed coproducts and an object $M\in\A$.
%Then, once again, the functor $\boT=\boT_M\:X\mapsto
%\Hom_\A(M,M^{(X)})$ is a monad on the category of sets.
%The following lemma is standard.

 We have seen that any abelian category $\B$ with coproducts
and a projective generator $P$ is equivalent to $\boT\modl$,
where $\boT$ is a monad on sets. Moreover, $\boT$ is an \emph{additive monad}
in that there exist a binary operation $+\in\boT(\{0,1\})$,
a unary operation $-\in\boT(\{0\})$ and a constant
$0\in\boT(\varnothing)$ in the monad $\boT$ satisfying
the usual axioms of an abelian group and commuting with all
the other operations in~$\boT$, \cite[\S10]{Wr}
(such an additive structure in $\boT$ must be unique
by the well-known Eckmann-Hilton argument).
 The arguments above show the converse as well, however:
For any additive monad $\boT\:\Sets\rarrow\Sets$, the category
$\boT\modl$ is abelian, has set-indexed coproducts, and $P = \boT(\{0\})$
is a projective generator. We summarize the discussion in the following
proposition.

\begin{prop}
There exists a bijective correspondence between
\begin{enumerate}
\renewcommand{\theenumi}{\alph{enumi}}
\item equivalence classes of abelian categories\/ $\B$ with coproducts
and a chosen projective generator $P\in\B$, and
\item isomorphisms classes of additive monads $(\boT,\mu,\varepsilon)$
on $\Sets$.
\qed
\end{enumerate}
\end{prop}

\begin{rem} \label{rem:tilting-contra}
Proposition~\ref{copowers-kleisli} and the discussion above allow us
to conveniently restate the derived equivalences
from Section~\ref{derived-equivalences-secn} in a form
which is more concrete and looks more classical.
If $\A$ is an abelian category with products and an injective cogenerator
and $T\in\A$ is an $n$\+tilting object,
then functor $\Psi_T=\Hom_\A(T,-)$ identifies with the functor
$\Psi\colon\A\rarrow\B$ of Section~\ref{derived-equivalences-secn}.
Indeed, both the functors $\Psi\:\A\rarrow\B$ and
$\Hom_A(T,{-})\:\A\rarrow\B$ are left exact, so it suffices to show
that they coincide on the full subcategory of injective
objects $\A_\inj\subset\A$.
For this purpose, it suffices to construct an isomorphism between
the restrictions of the two functors to the exact subcategory
$\E\subset\A$.
Both the functors are exact on this subcategory, so the question
reduces to checking that they coincide on the full subcategory of
projective objects $\Add(T)\subset\E$, which we know from the construction.
The derived equivalences from
Corollary~\ref{conv-abs-derived-equivalence} then take the form
\[ \boR\Hom_\A(T,-)\:\D^\st(\A)\overset{\sim}\lrarrow\D^\st(\B). \]
\end{rem}

\subsection{Categories of models of algebraic theories}

Under a mild technical condition, we also obtain a connection
between categories of modules over additive monads on one hand
and ordinary module categories on the other hand,
which is in some sense dual to the usual Gabriel--Popescu theorem for
Grothendieck categories. This will be useful in the context of contramodules
over topological rings in the next section.

 Let $\kappa$ be a regular cardinal.
 A projective generator $P$ in an abelian category $\B$ with coproducts
is called \emph{abstractly\/ $\kappa$\+small}
if every morphism $P\rarrow P^{(X)}$ in $\B$ factorizes through
the natural embedding $P^{(Z)}\rarrow P^{(X)}$ of the coproduct of
copies of $P$ over some subset $Z\subset X$ of cardinality less
than~$\kappa$.
Equivalently, one may require that the functor $\boT\: P\longmapsto
\Hom_\B(P,P^{(X)})$ preserve $\kappa$\+filtered colimits, or that
the object $P\in\B$ be $\kappa$\+presentable. 
The existence of an abstractly small projective generator in the category $\B$
is also equivalent to $\B$ being locally presentable with enough projective objects.
 The Eilenberg--Moore categories of monads on $\Sets$ which preserve
$\kappa$\+filtered colimits are also called the \emph{categories
of models of\/ $\kappa$\+ary algebraic theories}~\cite{Wr}
(cf.~\cite{La}, where finitary algebraic theories are discussed).
 We call such a theory \emph{additive} if the monad $\boT$ is
additive.

\begin{ex} \label{monad-not-k-small}
The existence of an abstractly $\kappa$\+small generator is not for
free. Let $k$ be a field and consider $\B = (k\modl)^\rop$.
This is an abelian category with a projective generator $k$ and
the corresponding additive monad $\boT\:\Sets\rarrow\Sets$ is
given by $\boT(X) = \Hom_k(k^X,k)$. However, $k$ is certainly not
abstractly $\kappa$\+small in $\B$.
\end{ex}

 In order to stress the parallel between Grothendieck categories
and the categories of models of\/ $\kappa$\+ary algebraic theories,
we recall the classical Gabriel--Popescu theorem
\cite[\S X.4]{Sten}.

\begin{thm} \label{gabriel-popescu}
 Let\/ $\A$ be a Grothendieck abelian category and $G\in\A$ be
a generator in\/~$\A$.
 Denote by $S$ the ring\/ $\Hom_\A(G,G)^\rop$.
 Then the functor\/ $\A\rarrow S\modl$ assigning to an object
$A\in\A$ the left $S$\+module\/ $\Hom_\A(G,A)$ is fully faithful,
and has an exact left adjoint functor.
\end{thm}

\begin{cor}
 Any Grothendieck category $\A$ can be presented
as a reflective full subcategory\/ $\A\subset S\modl$ in
the category of modules over an associative ring $S$ such that
the reflection functor $S\modl\rarrow\A$ is exact.

 Conversely, any reflective full subcategory in $S\modl$ which
is abelian as a category and for which the reflection functor is
exact is a Grothendieck category.
\end{cor}

 The following theorem is a ``dual-analogous'' version of
Theorem~\ref{gabriel-popescu} for the categories of models of
additive $\kappa$\+ary algebraic theories.
 In the nonadditive context, its result goes back to
Isbell~\cite[\S2.2]{Isb} (see also~\cite[Remark~1.3]{Ros}).

\begin{thm} \label{kappa-ary-dual-gabriel-popescu}
 Let\/ $\B$ be a cocomplete abelian category with an abstractly
$\kappa$\+small projective generator~$P$.
 Let $Y$ be a set such that the successor cardinal of the cardinality
of $Y$ is greater or equal to\/~$\kappa$.
 Set $Q=P^{(Y)}$, and denote by $S$ the ring\/ $\Hom_\B(Q,Q)^\rop$. 
 Then the functor\/ $\B\rarrow S\modl$ assigning to an object
$B\in\B$ the left $S$\+module\/ $\Hom_\B(Q,B)$ is exact, fully faithful,
and has a left adjoint functor\/~$\Delta$.
\end{thm}

\begin{cor} \label{kappa-ary-dual-gabriel-popescu-cor}
 The category of models of any additive $\kappa$\+ary
algebraic theory can be presented as a reflective full abelian subcategory
in the category of modules over an associative ring $S$ such that
the embedding functor\/ $\B\rarrow S\modl$ is exact.

Conversely, any reflective full abelian subcategory in $S\modl$ with
an exact embedding functor is a cocomplete abelian category with
a projective generator (and hence the category of models of
an additive\/ $\kappa$\+ary algebraic theory if it has an
abstractly $\kappa$\+small such generator).
\end{cor}

\begin{rem}
 If we assume Vop\v{e}nka's principle, the latter corollary get sharper.
This is because then a full subcategory in a locally presentable category is
reflective if and only if it is closed under limits, and every such
full subcategory is locally presentable and accessibly
embedded~\cite[Corollary~6.24 and Theorem~6.9]{AR}.

Hence then an additive category $\B$ is the category of models of
an additive\/ $\kappa$\+ary algebraic theory for some $\kappa$
if and only if it is a full exact abelian subcategory of $S\modl$ for some ring $S$
and it is closed under products.
\end{rem}

\begin{proof}[Proof of Theorem~\ref{kappa-ary-dual-gabriel-popescu}]
 One may view the functor $\Hom_\B(Q,-)\:\B\rarrow S\modl$ as the restricted
Yoneda functor which sends $X\in\B$ to $\Hom_\B(-,X)$ restricted to the one-object
full subcategory $\{Q\}\subset\B$. Then the fact that $\Hom_\B(Q,-)$ has a left
adjoint $\Delta\:S\modl\rarrow\B$ is just an additive version of~\cite[Proposition 1.27]{AR}.
To obtain a more concrete description of $\Delta$, one computes that
$$
 \Hom_S(S^{(X)},\Hom_\B(Q,B)) = \Hom_\B(Q,B)^X
 = \Hom_\B(Q^{(X)},B)
$$
for any set $X$, and hence the functor $\Delta$ is defined on the full
subcategory of free $S$\+modules in $S\modl$ by the rule
$\Delta(S^{(X)})=Q^{(X)}$. To compute the functor $\Delta$ on an arbitrary
$S$\+module, one presents it as the cokernel of a morphism of free
$S$\+modules and uses the fact that left adjoint functors preserve
cokernels.

 Obviously, $B\longmapsto\Hom_\B(Q,B)$ is exact since $Q$ is projective.
 To prove that the functor is fully faithful,
let us identify the category $\B$ with the category of modules over
the monad $\boT\:\Sets\rarrow\Sets$, $X\longmapsto\Hom_\B(P,P^{(X)})$.
Hence $P$ identifies with the free $\boT$\+module $\boT(\{0\})$ with one generator.
 Consider two objects $C,D\in\B$, whose underlying sets admit canonical
identification $C=\Hom_\B(P,C)$ and $D=\Hom_\B(P,D)$.
% with the operations of the monad $\boT$ acting in them.
 Since then $Q = P^{(Y)} = \boT(Y)$ is a free $\boT$\+module, we also have
identifications $\Hom_\B(Q,C)=C^Y$ and $\Hom_\B(Q,D)=D^Y$.
 (Notice that the forgetful functor $\Hom_\B(P,{-})\:\B\rarrow\Sets$
preserves products, so our notation is unambiguous.)

 Let $f\:C^Y\rarrow D^Y$ be a morphism in the category $S\modl$
(recall that $S=\Hom_\B(Q,Q)^\rop$).
 For every element $y\in Y$, denote by $p_y\in S$ the idempotent
morphism $P^{(Y)}\rarrow P^{(Y)}$ acting as the identity on
the $y$\+indexed component $P$ in $P^{(Y)} = Q$ and by zero on
the $y'$\+indexed components for all $y\ne y'\in Y$.
 Let $p_y^*\:C^Y\rarrow C^Y$ and $p_y^*\:D^Y\rarrow D^Y$ denote
the induced maps on the underlying sets $C^Y=\Hom_\B(Q,C)$
and $D^Y=\Hom_\B(Q,D)$. These maps are also the idempotent
projectors for the $y$\+indexed components in $C^Y$ and~$D^Y$, and
they coincide with the $S$\+module action of $p_y$ on $C^Y$ and $D^Y$,
respectively. 
 Since $f$ is an $S$\+module map, we have equalities $fp_y^*=p_y^*f$
for all $y\in Y$, which in turn means that
there exist maps $g_y\:C\rarrow D$, one for each element
$y\in Y$, such that the map $f\:C^Y\rarrow D^Y$ is the product of
the family of maps~$g_y$, that is $f=\prod_{y\in Y} g_y$.

 For every pair of elements $y'\ne y''\in Y$, denote by $s_{y'y''}
\in S$ the involutive automorphism $P^{(Y)}\rarrow P^{(Y)}$ permuting
the $y'$\+indexed component $P$ in $P^{(Y)}$ with the $y''$\+indexed
component and acting by the identity on all the other components.
 Let $s_{y'y''}^*\:C^Y\rarrow C^Y$ and $s_{y'y''}^*\:D^Y\rarrow D^Y$
denote the induced maps describing the $S$-module action of $s_{y'y''}$.
 These are also simply the maps permuting the $y'$\+indexed
component with the $y''$\+indexed one in the Cartesian powers.
 We again have an equality $fs_{y'y''}^*=s_{y'y''}^*f$, which in turn
means that the maps $g_{y'}$ and $g_{y''}\:C\rarrow D$ are equal.
 Hence our morphism $f\:C^Y\rarrow D^Y$ is the direct power
of some map $g\:C\rarrow D$, i.~e., $f=g^Y$.

 It remains to show that the map of sets $g\:C\rarrow D$, which we obtained
from the $S$\+module morphism $f\:C^Y\rarrow D^Y$, is
a morphism in the category~$\B$.
 For this purpose, it suffices to check that $g$~commutes with
all the operations in the monad~$\boT$, that is for every set $X$
and every element $t\in\boT(X)=\Hom_\B(P,P^{(X)})$ the induced
maps $t^*\:C^X\rarrow C$ and $t^*\:D^X\rarrow D$ form a commutative
square with the maps $g^X\:C^X\rarrow D^X$ and $g\:C\rarrow D$.
 Since the object $P$ is abstractly $\kappa$\+small, one can assume
that the cardinality of the set $X$ is smaller than~$\kappa$.
 By assumption on the cardinality of the set $Y$, this means that
$X$ can be embedded into~$Y$.
 Choosing such an embedding, we can assume that $X=Y$.

 Choose an element $y\in Y$; and consider the composition
$q\:P^{(Y)}\rarrow P^{(Y)}$ of the projection onto the $y$\+indexed
component $P^{(Y)}\rarrow P$ with the morphism $t\:P\rarrow P^{(Y)}$.
 The induced map $q^*\:C^Y\rarrow C^Y$ is the composition of
the map $t^*\:C^Y\rarrow C$ with the embedding of the $y$\+indexed
component $C\rarrow C^Y$; and the induced map $q^*\:D^Y\rarrow D^Y$
is described similarly.
 Thus the equation $fq^*=q^*f$ for the $S$\+module map $f=g^Y$ implies
the desired equation $gt^*=t^*g^Y$.
\end{proof}

\begin{proof}[Proof of Corollary~\ref{kappa-ary-dual-gabriel-popescu-cor}]
 The first part is a direct consequence of
Theorem~\ref{kappa-ary-dual-gabriel-popescu}.

 Conversely, let $\B$ be a full exact abelian subcategory
in $S\modl$ with the reflection functor $\Delta\:S\modl\rarrow\B$.
 Then $\B$ is cocomplete, as one can compute colimits in $\B$
by applying the functor~$\Delta$ to the colimit of the same diagram
computed in $S\modl$.
 Besides, the category $\B$ has a generator $P=\Delta(S)$ which is
projective since $\Delta$ is left adjoint to an exact functor.
\end{proof}

 The following property of the exact fully faithful functor
$\B\rarrow S\modl$ is also worth noting.

\begin{lem}
 In the setting of
Theorem~\textup{\ref{kappa-ary-dual-gabriel-popescu}},
the functor\/ $\Hom_\B(Q,{-})\:\B\rarrow S\modl$ takes
the projective objects of\/ $\B$ to flat $S$\+modules.
\end{lem}

\begin{proof}
 Any projective object in $\B$ is a direct summand of a copower of
the projective generator $P$; so it suffices to show that
the $S$\+modules $\Hom_\B(Q,P^{(X)})$ are flat for all sets~$X$.
 Let $\lambda$~denote the cardinality of the set~$Y$; by assumption,
the successor cardinality~$\lambda^+$ of the cardinal~$\lambda$
is greater or equal to~$\kappa$.
 Without loss of generality, we can assume that the cardinality of~$X$
is not smaller than~$\lambda$.
 Then we have
$$
 P^{(X)}=\varinjlim\nolimits_{Z\subset X}^{|Z|=\lambda}P^{(Z)},
$$
where the $\lambda^+$\+filtered colimit is taken over all the subsets
$Z\subset X$ of the cardinality equal to~$\lambda$.
 The object $Q\in\B$ is $\lambda^+$\+presentable, so we have
$$
 \Hom_\B(Q,P^{(X)})=\varinjlim\nolimits_Z\Hom_\B(Q,P^{(Z)}).
$$
 Now all the objects $P^{(Z)}\in\B$ are isomorphic to $Q$, so
the $S$\+module $\Hom_\B(Q,P^{(Z)})$ is free (with one generator)
and the $S$\+module $\Hom_\B(Q,P^{(X)})$ is a filtered colimit
of free $S$\+modules.
\end{proof}

\Section{Morita Theory for Topological Rings}
\label{morita-theory-secn}

The aim of this section is to develop basic machinery which we need to manipulate contramodule categories in an analogous way to module categories. This allows us to express Morita equivalences and tilting adjunctions of categories rather explicitly in terms of Hom and contratensor (a contramodule analogue of the tensor) functors.

\subsection{Finite topology on endomorphism rings}
\label{finite-topology-subsecn}

We start with a key result which says that
the endomorphism ring of a module over an ordinary ring always carries a natural structure of a topological ring and hence we can consider contramodules over it. In fact, we will find it useful to state the result for left $\cR$\+modules, where $\cR$ is a small preadditive category (recall that in this case, an $\cR$\+module is a covariant additive functor $\cR\rarrow\Z\modl$, and the case of modules over a ring $R$ is recovered by viewing $R$ as a category with one object whose endomorphism ring is $R$).

The topology on the ring $\fS$ in the theorem appeared under the name
\emph{finite topology} in the literature; see~\cite[Section~IX.6]{Jac},
\cite[Section~107]{Fuc}, \cite{SM}, \cite{Niel} (it is implicit in
Jacobson's density theorem~\cite[Section~IX.13]{Jac}).

\begin{thm} \label{addm-theorem}
Let $\cR$ be a small preadditive category, $M$ be a left $\cR$\+module, and\/
$\Add(M)\subset \cR\modl$ be the full additive subcategory formed
by the direct summands of infinite direct sums of copies of
the object $M$ in the category of left $\cR$\+modules.
Then there exists a complete, separated topological associative
ring\/ $\fS$ with a base of neighborhoods of zero formed by open
right ideals such that the category\/ $\Add(M)$ is equivalent to
the category of projective left\/ $\fS$\+contramodules.
\end{thm}

\begin{proof}
Consider the ring $\Hom_\cR(M,M)$ and endow it with the topology in
which the base of neighborhoods of zero is formed by the annihilator
ideals $\Ann(E)\subset\Hom_\cR(M,M)$ of finitely generated $\cR$\+submodules $E\subset M$.
To be precise, the annihilator
$$
\Ann(E)=\Hom_\cR(M/E,M)=\{\,f\in\Hom_\cR(M,M)\mid f(E)=0\,\}
$$
is a left ideal in the ring $\Hom_\cR(M,M)$.
In view of the exact sequence
$$
0\lrarrow\Hom_\cR(M/E,M)\lrarrow\Hom_\cR(M,M)\lrarrow\Hom_\cR(E,M),
$$
the quotient module $\Hom_\cR(M,M)/\Ann(E)$ is identified with the set
of all $\cR$\+module morphisms $E\rarrow M$ that can be extended to
$\cR$\+module morphisms $M\rarrow M$.
Since the datum of an $\cR$\+module morphism $M\rarrow M$ is equivalent
to that of a compatible system of $\cR$\+module morphisms
$E\rarrow M$ defined for all the finitely generated submodules
$E\subset M$, we have an isomorphism
$$
\Hom_\cR(M,M)\.\cong\.\varprojlim\nolimits_{E\subset M}
\Hom_\cR(M,M)/\Ann(E),
$$
that is $\Hom_\cR(M,M)$ is a complete and separated topological ring.
	
Let $\fS=\Hom_\cR(M,M)^\rop$ be the topological ring opposite to
$\Hom_\cR(M,M)$; so $\fS$ is a complete, separated topological ring
with a base of the topology formed by open right ideals.
Consider the functor $\boT\:\Sets\rarrow\Sets$ assigning to every set
$X$ the set $\Hom_\cR(M,M^{(X)})$ of all $\cR$\+module morphisms from $M$ to
the direct sum of $X$ copies of~$M$.
Since $\boT=\Psi_M\circ\Phi_M$, where 
$\Psi_M=\Hom_\R(M,-)\:\cR\modl\rarrow\Sets$ and 
$\Phi_M\:\Sets\rarrow\cR\modl$ is the left adjoint of 
$\Psi_M$ which assigns the object $M^{(X)}$ to a set $X$,
the functor $\boT$ has a natural structure of a monad on the category of sets (see~\S\ref{additive-monads-subsecn}).
We claim that this monad is isomorphic to the monad~$\fS[[-]]$.
	
Indeed, $\Hom_\cR(M,M^{(X)})$ is a subset in $\Hom_\cR(M,M)^X$ and
$\fS[[X]]$ is a subset in $\fS^X$; first of all, it is claimed that
this is the same subset.
In other words, an $X$\+indexed family of $\cR$\+module morphisms
$f_x\:M\rarrow M$ corresponds to an $\cR$\+module morphism $M\rarrow M^{(X)}$ if and only if it converges to zero in the topology of
$\Hom_\cR(M,M)$, that is if and only if for every finitely-generated
submodule $E\subset M$ one has $f_x(E)=0$ for all but a finite
subset of indices $x\in X$.
This is obvious; and checking that the unit and composition
operations in our two monads $\boT\:X\longmapsto\Hom_\cR(M,M^{(X)})$ and
$X\longmapsto\fS[[X]]$ are the same is straightforward.
%(cf.\ the proof of Corollary~\ref{hom-m-n-contramodule} below).
	
In view of Proposition~\ref{copowers-kleisli}, the assignment of the $\cR$\+module $M^{(X)}$ to
a free $\fS$\+contra\-module $\fS[[X]]$ establishes an isomorphism
between the category of direct sums of copies of $M$ in
$\cR\modl$ and the category of free left $\fS$\+contramodules.
Adjoining direct summands to the categories on both sides, we
obtain an equivalence between $\Add(M)$ and the category of projective
left $\fS$\+contramodules.
\end{proof}

\begin{rem}
 The following classical fact is a particular case of
Theorem~\ref{addm-theorem}: for any associative ring $R$ and
a \emph{finitely generated} left $R$\+module $M$, the category
$\Add(M)$ is equivalent to the category of projective left modules
over the ring $S=\Hom_R(M,M)^\rop$.
 A standard reference is~\cite{Dr}, where a version of this result
is mentioned without proof.
 A proof can be found in~\cite[Theorem~32.1]{KMT}.
 Moreover, the same assertion holds for any \emph{self-small} left
$R$\+module, i.~e., a left $R$\+module $M$ such that every
homomorphism $M\rarrow M^{(X)}$ factorizes through
$M^{(Z)}\subset M^{(X)}$ for some finite subset $Z$ in a set~$X$
(cf.\ Remark~\ref{dually-slender} below).
\end{rem}

\subsection{Contratensor product}
\label{contratensor-subsecn}

%In order to describe the right exact functor $\Phi\:\B\rarrow\A$,
%we need to recall one general construction related to contramodules.

If $M$ is an $R$\+$S$\+bimodule over associative rings, the functor
$\Hom_R(M,-)$ has a left adjoint $M\otimes_S-$. In the world of contramodules,
the same role is played by the contratensor product. However, the correct object
to tensor a left contramodule with is a right discrete module.

Let $\fS$ be a complete, separated topological ring with a base of
neighborhoods of zero formed by open right ideals.
A right $\fS$\+module $L$ is called \emph{discrete} if 
the action map $L\times\fS\rarrow L$ is continuous in the given topology
of $\fS$ and the discrete topology of $L$, or in other words, if
the annihilator ideal of every element in $L$ is open in~$\fS$.
Discrete right $\fS$\+modules form a full subcategory $\discr\fS$ 
of the category of right $\fS$\+modules
which is closed under submodules, quotients
and arbitrary direct sums, \cite[Proposition VI.4.2]{Sten}.
In particular, $\discr\fS$ is a Grothendieck abelian category.

For any discrete right $\fS$\+module $L$ and every abelian group $V$
the group $\fC=\Hom_\Z(L,V)$ has a natural left $\fS$\+contramodule
structure with the monadic action map $\fS[[\fC]]\rarrow\fC$ defined
by the rule
\begin{equation} \label{hom-contramodule}
\Biggl(\sum_{x\in X}s_xf_x\Biggr)(l)=\sum_{x\in X}f_x(ls_x),
\end{equation}
where
% $X$ is a set mapping into $\fC$ by an arbitrary map denoted by
%$x\longmapsto f_x$,
the family of elements $s_x\in\fS$ converges to zero
in the topology of $\fS$ and $l\in L$.
The infinite sum in the left-hand side (representing an element
of~$\fC$) is a symbolic notation for the monadic action map that we are
defining, while the sum on the right-hand side is finite because
the right $\fS$\+module $L$ is discrete.

Let $L$ be a discrete right $\fS$\+module and $\fC$ be a left
$\fS$\+contramodule.
The \emph{contratensor product} $L\ocn_\fS\fC$ is the abelian group
constructed as the quotient group of the group $L\ot_\Z\fC$ by
the subgroup generated by all elements of the form
\begin{equation} \label{contratensor-construction}
\sum_{x\in X}ls_x\ot c_x - l\ot\sum_{x\in X}s_xc_x, 
\end{equation}
where $l\in L$ is an element, $X$ is a set, $s_x\in\fS$ is a family
of elements converging to zero in the topology of $\fS$, and
$c_x\in\fC$ is an arbitrary family of elements.
The sum in the left-hand summand is finite because the right
$\fS$\+module $L$ is discrete, while the infinite sum in the
right-hand summand is a symbolic notation for the monad action
map $\fS[[\fC]]\rarrow\fC$.

The contratensor product is a right exact functor of two arguments
$$
\ocn_\fS\:\discr\fS\times\fS\contra\rarrow\Z\modl.
$$
For any discrete right $\fS$\+module $L$ and any set $X$, there is
a natural isomorphism of abelian groups
\begin{equation} \label{contratensor-with-free}
L\ocn_\fS\fS[[X]]\cong L^{(X)}.
\end{equation}
For any discrete right $\fS$\+module $L$, any left
$\fS$\+contramodule $\fC$, and any abelian group~$V$ there is
a natural isomorphism
\begin{equation} \label{contratensor-group-adjunction}
\Hom_\Z(L\ocn_\fS\fC,\>V)\cong\Hom^\fS(\fC,\.\Hom_\Z(L,V)),
\end{equation}
(we recall that $\Hom^\fS$ denotes the group of morphisms in
the category of left $\fS$\+contramodules $\fS\contra$).

More generally, let $\cR$ be a small preadditive category and
let $L\:\cR\rarrow\Z\modl$ be a left $\cR$\+module endowed with
an right action $\fS\rarrow\Hom_\cR(L,L)^\rop$, and assume that
the action of $\fS$ makes each $L(r)$, $r\in\cR$, a discrete $\fS$-module.
If $\cR$ has only one object with endomorphism ring $R$, then $L$
is none other than an $R$\+$\fS$\+bimodule which is discrete on the right.
Let further $V$ be a left $\cR$\+module and $\fC$ be a left $\fS$\+contramodule.
Then $\Hom_\cR(L,V)$ has a left contramodule structure given by
the formula~\eqref{hom-contramodule} for each $r\in\cR$ and $l\in L(r)$.
%is a subcontramodule of the left $\fR$\+contramodule
%$\Hom_\Z(L,V)$, so the group $\Hom_A(L,V)$ has a natural left
%$\fR$\+contramodule structure.
Furthermore, the contratensor product defines a left $\cR$\+module
$\cR\rarrow\Z\modl$, $r\mapsto L(r)\ocn_\fS\fC$, which we denote by
$L\ocn_\fS\fC\in\cR\modl$.
Then there is a natural isomorphism of abelian groups
\begin{align} \label{contratensor-module-adjunction}
\Hom_\cR(L\ocn_\fS\fC,\>V)\cong\Hom^\fS(\fC,\.\Hom_\cR(L,V))
\end{align}
which sends a homomorphism of left $\cR$\+modules
$(f_r\: L(r)\ocn_\fS\fC\rarrow V(r))_{r\in\cR}$ to the morphism
of left $\fS$\+contramodules which assigns to each $c\in\fC$
the homomorphism of left $\cR$\+modules $(f_r(-\ocn c))_{r\in R}$.

In other words, the contratensor product functor
$$
L\ocn_\fS{-}\:\fS\contra\lrarrow \cR\modl
$$
is left adjoint to the Hom functor
$$
\Hom_\cR(L,{-})\:\cR\modl\lrarrow\fS\contra.
$$
We can now refine Theorem~\ref{addm-theorem} as follows.

\begin{prop} \label{prop:hom-contratensor-m}
Let\/ $\cR$ be a small preadditive category, $M\in\cR\modl$ and
let\/ $\fS$ be the topological ring\/ $\Hom_A(M,M)^\rop$ as in
Theorem~\ref{addm-theorem}. Then the right action of $\fS$ makes
all $M(r)$, $r\in\cR$ discrete\/ $\fS$\+modules and the pair
of adjoint functors
$$
M\ocn_\fS{-}\:\fS\contra\lrarrow\cR\modl
\quad\textrm{and}\quad
\Hom_\cR(M,{-})\:\cR\modl\lrarrow\fS\contra
$$
restricts to a pair of inverse equivalences between the subcategories\/
$\fS\contra_\proj\subset\fS\contra$ and\/ $\Add(M)\subset\cR\modl$.
\end{prop}

\begin{proof}
The right $\fS$\+modules $M(r)$, $r\in\cR$ are discrete by the definition.
The fact that $\Hom_\cR(M,{-})$ restricts to an equivalence
$\Add(M)\simeq\fS\contra$ follows directly by identifying it with
the equivalence constructed in Theorem~\ref{addm-theorem}.
On the other hand, the isomorphism~\eqref{contratensor-with-free}
shows that the functor $M\ocn_\fS{-}$ takes $\fS[[X]]$ to $M^{(X)}$
and, hence, that $M\ocn_\fS{-}$ restricts to an equivalence
$\fS\contra\simeq\Add(M)$.
\end{proof}

\subsection{Morita equivalence for contramodules}
\label{morita-subsecn}

Now we can answer the question when precisely two topological rings
have equivalent contramodule categories. This will also allow us
to explain nice consequences of Theorem~\ref{kappa-ary-dual-gabriel-popescu}
for topological rings.

We first recall necessary results from~\cite[\S6]{PR} and \cite[\S\S4 and~5]{Pflat}.
Let $\fR$ be a complete, separated topological ring with a base of
neighborhoods of zero formed by open right ideals. Given
a left $\fR$\+contramodule $\fC$ and an open right ideal $\fU\subset\fR$,
we denote by $\fU\tim\fC\subset\fC$ the additive subgroup formed by the elements
of the form $\{\sum_{x\in X}u_xc_x\mid u_x\in\fU, c_x\in\fC \}$, where the sequences
$u_x\in\fU$ converge to zero in the topology of $\fR$. These subgroups make $\fC$
a topological group and we can consider the completion morphism
$\lambda_{\fR,\fC}\:\fC\rarrow\varprojlim_{\fU}\fC/(\fU\tim\fC)$. The contramodule
$\fC$ is called \emph{complete} if $\lambda_{\fR,\fC}$ is surjective and \emph{separated}
if $\lambda_{\fR,\fC}$ is injective.

\begin{ex} \label{ex:projective complete separated}
All projective $\fR$-contramodules are complete and separated. This follows more or less by definition,
since for each set $X$, the subgroup $\fU\tim\fR[[X]]$ consists precisely of those formal linear combinations
$\sum_{x\in X}s_xx\in\fR[[X]]$ with $s_x\in\fU$ for each $x\in X$. In particular,
$\fR[[X]]/(\fU\tim\fR[[X]]) = (\fR/\fU)[X]$, where the symbol stands for the set of formal sums
$\sum_{x\in X}\overline{s}_xx$ with $\overline{s}_x\in(\fR/\fU)$ and nonzero only for finitely many
indices $x\in X$. Since $\fR$ itself is complete and separated, it follows that
\[
\lambda_{\fR,\fR[[X]]}\:\fR[[X]]\rarrow\varprojlim\nolimits_{\fU}\fR[[X]]/(\fU\tim\fR[[X]]) = \varprojlim\nolimits_{\fU}(\fR/\fU)[X]
\]
is an isomorphism.
\end{ex}

Now we denote by $\cR$ the full subcategory
$\{\fR/\fU \mid \fU\subset\fR \textrm{ open right ideal} \} \subset \discr\fR$
and by $\Pi\subseteq\cR$ the nonfull subcategory with the same objects, but
only the projection morphisms $\fR/\fU\twoheadrightarrow \fR/\fU'$ whenever
$\fU\subset\fU'$.
Then we can consider the functor $CT\:\fR\contra\rarrow\cR\modl$ which sends
$\fC\in\fR\contra$ to the left $\cR$\+module
\[ CT(\fC)\: \fR/\fU\longmapsto \fR/\fU\ocn_\fR\fC = \fC/(\fU\tim\fC). \]
We will need the following results.

\begin{prop} \label{prop:F-systems}
The functor $CT\:\fR\contra\rarrow\cR\modl$ has a right adjoint functor
$PL\:\cR\modl\rarrow\fR\contra$, which is computed at the level of the underlying
groups as $PL(M) = \varprojlim_\Pi M(\fR/\fU)$, and
$\lambda_{\fR,\fC}\:\fC\rarrow PL(CT(\fC))$ is the
unit of this adjunction for each $\fC\in\fR\contra$ (in particular, it is
a homomorphism of left $\fR$\+contramodules). 
\end{prop}

\begin{proof}
We refer to~\cite[Propositions 4.2(b) and 5.2]{Pflat} or
to~\cite[Lemma 6.2(a,c)]{PR}.
\end{proof}

\begin{cor} \label{cor:F-systems}
The restriction of $CT\:\fR\contra\rarrow\cR\modl$ to the full subcategory $\fR\contra_\cs$ of
complete and separated contramodules is fully faithful.
So is the restriction of $PL\:\cR\modl\rarrow\fR\contra$ to the essential image
of $CT\:\fR\contra_\cs\rarrow\cR\modl$.
\end{cor}

\begin{proof}
If $\fC\in\fR\contra$ is complete and separated, then $\lambda_{\fR,\fC}$ is
an isomorphism, and so is, by
the triangular identities~\cite[Theorem IV.1.1]{MacLane},
the adjunction counit $CT(PL(M))\rarrow M$ for $M=CT(\fC)$.
The rest follows from~\cite[Theorem IV.3.1]{MacLane}.
\end{proof}

As an immediate consequence, we also obtain the following result.

\begin{cor} \label{cor:bicontramodule}
Let\/ $\fC\in\fR\contra$ be such that the coproduct\/ $\fC^{(X)}$ is
a complete and separated left\/ $\fR$\+contramodule for each set~$X$.
Then\/ $\fS=\Hom^\fR(\fC,\fC)^\rop$ with the base of neighborhoods
of zero formed by the open right ideals
\[ \fV_{F,\.\fU} = \{ f\:\fC\rarrow\fC \mid f(F)\subset\fU\tim\fC \} \subset \fS, \]
where $F$ runs over the finite subsets of\/ $\fC$ and\/ $\fU$ runs over
the open right ideals of\/ $\fR$, is a complete, separated
topological ring. Moreover, the category\/ $\Add(\fC)\subset\fR\contra$
is equivalent to the category of left projective\/ $\fS$\+contramodules.
\end{cor}

\begin{proof}
By Corollary~\ref{cor:F-systems}, we have an isomorphism of rings
$\fS=\Hom^\fR(\fC,\fC)^\rop \cong \Hom_\cR(M,M)^\rop$, where $M=CT(\fC)$.
Under this isomorphism, the right ideal $\fV_{F,\fU}$ is carried over to the
right ideal $\{f\:M\rarrow M\mid f(E)=0\}\subset \Hom_\cR(M,M)^\rop$,
where $E\subset M$ is the $\cR$\+submodule generated by the cosets
of elements of $F$ in $\fR/\fU$. The latter ideal is clearly open
in $\Hom_\cR(M,M)^\rop$ in the topology of Theorem~\ref{addm-theorem}.

Conversely, let $E\subset M$ be an $\cR$\+submodule generated by
finitely many elements $y_i\in\fR/\fU_i$, $i=1,\dots,n$. If we lift
each $y_i$ to an element $x_i\in\fR$ and put $F=\{x_1,\dots,x_n\}$
and $\fU=\fU_1\cap\cdots\cap\fU_n$, one checks that the open right ideal in
$\Hom_\cR(M,M)^\rop$ consisting of the morphisms $f\:M\rarrow M$ vanishing
on $E$ contains the image of $\fV_{F,\.\fU}$. Hence,
$\Hom^\fR(\fC,\fC)^\rop \cong \Hom_\cR(M,M)^\rop$ is an isomorphism of
topological rings.

Finally, the equivalence $\Add(\fC)\simeq\fS\contra_\proj$
is obtained by composing the equivalence
$CT\:\Add(\fC)\rarrow\Add(M)$ given by Corollary~\ref{cor:F-systems}
with the equivalence $\Hom_\cR(M,-)\:\Add(M)\rarrow\fS\contra_\proj$
established in Theorem~\ref{addm-theorem} and
Proposition~\ref{prop:hom-contratensor-m}.
\end{proof}

\begin{rem}
More explicitly, the equivalence of
Corollary~\ref{cor:bicontramodule} assigns to $\fD\in\Add(\fC)$
the set $\Hom^\fR(\fC,\fD)$ with the $\fS$\+contramodule structure given
by the formula $(\sum_{x\in X}s_xf_x)(c)=\sum_{x\in X}f_x(cs_x)$
for each family of elements $s_x\in\fS$ converging to zero
in the topology of $\fS$ and $c\in \fC$.
The infinite sum in the left-hand side is a symbolic notation for
the monadic action map that we are
defining, while the sum on the right-hand side is convergent
in the topology on $\fD$ with the base of open neighborhoods of zero
given by $\fU\tim\fD$, where $\fU$
runs over open right ideals of $\fR$.
The inverse equivalence assigns to $\fP\in\fS\contra_\proj$
the left $\fR$\+contramodule
$\varprojlim_\fU \fC/(\fU\tim\fC)\ocn_\fS\fP$.
\end{rem}

We summarize the results in the following theorem.

\begin{thm} \label{morita-thm}
Let\/ $\fR$ be a complete, separated topological ring with a base of
neighborhoods of zero formed by open right ideals, $\fP\in\fR\contra$
be a projective generator, and\/ $\fS=\Hom^\fR(\fP,\fP)^\rop$.
Then\/ $\fS$ admits the structure of a complete, separated topological
ring with a base of neighborhoods of zero formed by open right ideals
such that there is an equivalence of abelian categories\/
$\fR\contra\simeq\fS\contra$ taking the projective contramodule\/ $\fP$
to the free contramodule with one generator\/ $\fS\in\fS\contra$.
%Moreover, any equivalence of contramodule
%categories for topological rings $\fR$ and $\fS$ with bases of
%neighborhoods of zero formed by open right ideals arises in this way.
\end{thm}

\begin{proof}
In view of Example~\ref{ex:projective complete separated}, 
it is an immediate consequence of Corollary~\ref{cor:bicontramodule}
that $\fR\contra_\proj\simeq\fS\contra_\proj$ with the corresponding
topology on $\fS$. Then also $\fR\contra\simeq\fS\contra$
(see the proof of Corollary~\ref{cor:B-as-Tmod}).
\end{proof}

\begin{ex} \label{morita-expl}
Suppose that $\fP=\fR[[Y]]$ for a nonempty set $Y$. Then $\fS$ is the ring
of matrices over $\fR$ with rows and columns indexed by $Y$ and with
every row converging to zero in the topology of $\fR$. The equivalence
$\fR\contra\rarrow\fS\contra$ sends $\fC\in\fR\contra$ to
$\Hom^\fR(\fP,\fC)=\fC^Y$ with the left $\fS$-contramodule structure given by
$\sum_{x\in X}s_xd_x = (\sum_{x\in X,y\in Y}r^x_{zy}c^x_y)_{z\in Y}\in\fC^Y$,
where $s_x = (r^x_{zy})_{z,y\in Y}\in\fS$ is a collection of matrices over
$\fR$ which converges to zero in $\fS$ and $d_x=(c^x_y)_{y\in Y}\in\fC^Y$.
\end{ex}

We conclude by discussing an interpretation of
Theorem~\ref{kappa-ary-dual-gabriel-popescu} for topological rings.
The point is that up to Morita equivalence, the category of contramodules
is a full subcategory of the category of ordinary modules, and
the contratensor product reduces to an ordinary tensor product.

To this end, let $\fR$ be a topological ring as above 
and $Y$ be a set such that
the cardinality of $Y$ is greater than or equal to the cardinality of
a base of neighborhoods of zero in $\fR$.
By setting $\fS=\Hom^\fR(\fR[[Y]],\fR[[Y]])^\rop$
with the topology from Corollary~\ref{cor:bicontramodule},
we obtain functors
\[ \fR\contra \simeq \fS\contra \overset{\subset}\lrarrow \fS\modl, \]
where the forgetful functor is fully faithful by
Theorem~\ref{kappa-ary-dual-gabriel-popescu}. Regarding the tensor product,
we have following observation.

\begin{lem} \label{contratensor-vs-tensor}
Let\/ $\fS$ be a complete, separated topological ring with a base
of neighborhoods of zero formed by open right ideals.
Suppose that the forgetful functor\/ $\fS\contra\rarrow\fS\modl$ is
fully faithful.
Then the natural surjection of abelian groups
$$
N\ot_\fS\fC\lrarrow N\ocn_\fS\fC
$$
is an isomorphism for every discrete right\/ $\fS$\+module $N$ and
left\/ $\fS$\+contramodule\/~$\fC$.
\end{lem}

\begin{proof}
Compare the two natural isomorphisms
\begin{align*}
\Hom_\Z(N\ocn_\fS\fC,\>V)&\cong\Hom^\fS(\fC,\.\Hom_\Z(N,V)) \\
\intertext{and}
\Hom_\Z(N\ot_\fS\fC,\>V)&\cong\Hom_\fS(\fC,\.\Hom_\Z(N,V)),
\end{align*}
which hold for every abelian group~$V$
(see~\eqref{contratensor-group-adjunction}).
\end{proof}

\Section{Big Tilting Modules}
\label{big-tilting-module-secn}

Here we specialize the tilting theory which we have developed
to module categories.
We obtain rather concrete derived equivalences which relate directly
to the more classical tilting results in the literature.

To start with, let $\fA$ be a topological ring with a base of neighborhoods
of zero formed by open \emph{left} ideals, and let $M\in\fA\discrl$
be a discrete left $\fA$\+module.
It is interesting here to consider even the case where $\fA$
carries the discrete topology (so it is just an ordinary ring).

If we equip the endomorphism ring $\fR=\Hom_\fA(M,M)^\rop$ with the finite
topology (Theorem~\ref{addm-theorem}), then
\begin{enumerate}
\item $\fR$ is a complete separated
topological ring with a base of neighborhoods of zero formed by open right
ideals,
\item $M$ is an $\fA$\+$\fR$\+bimodule, and
\item $M$ is a discrete module both over $\fA$ and
(by Proposition~\ref{prop:hom-contratensor-m}) over $\fR$.
\end{enumerate}

We also know that there is an adjunction
\begin{equation}\label{eq:discr-bimodule}
M\ocn_\fR{-}\:\fR\contra \;\rightleftarrows\; \fA\discrl \;\:\!\!\Hom_\fA(M,{-}).
\end{equation}
Up to the Morita equivalence (i.e., passing from $M$ to $M^{(Y)}$ and from $\fR$
to $\fS=\Hom_\fA(M^{(Y)},M^{(Y)})^\rop\cong\Hom^\fR(\fR[[Y]],\fR[[Y]])^\rop$,
see Theorem~\ref{morita-thm} and Example~\ref{morita-expl}),
the category $\fR\contra$ is a full subcategory of $\fR\modl$.
Indeed, this follows from Theorem~\ref{kappa-ary-dual-gabriel-popescu}
if $Y$ is an infinite set of cardinality greater than or equal to
the minimal cardinality of a set of generators of $M$.
The adjunction~\eqref{eq:discr-bimodule}
is then, by Lemma~\ref{contratensor-vs-tensor},
simply a restriction of the usual adjunction
\[
M\otimes_\fR{-}\:\fR\modl \;\rightleftarrows\; \fA\modl \;\:\!\!\Hom_\fA(M,{-}).
\]

Suppose now that $T\in\fA\discrl$ is an $n$\+tilting object in the sense
of Section~\ref{tilting-t-structure-secn} and again that
$\fR=\Hom_\fA(T,T)^\rop$. 
If we pass to the derived functors of the adjunction~\eqref{eq:discr-bimodule},
we obtain, in view of Remark~\ref{rem:tilting-contra} and
Theorem~\ref{addm-theorem}, the equivalences from
Corollary~\ref{conv-abs-derived-equivalence} in the form
\begin{equation}\label{eq:discr-tilting}
T\ocn^\boL_\fR{-}\:\D^\st(\fR\contra) \;\rightleftarrows\; \D^\st(\fA\discrl) \;\:\!\boR\!\Hom_\fA(T,{-}).
\end{equation}

To recover a corresponding $n$\+cotilting object in $\fR$\+contra,
note that the maximal discrete left $\fA$\+submodule $W$ of
$\Hom_\Z(\fA,\boQ/\Z)$ is an injective cogenerator of $\fA\discrl$.
The $n$\+cotilting object which we look for is then, 
by Theorem~\ref{tilting-cotilting-thm},
\[ \Hom_\fA(T,W) = \Hom_\fA(T,\Hom_\Z(\fA,\boQ/\Z)) = \Hom_\Z(T,\boQ/\Z). \]
The $\fR$\+contramodule structure on $\Hom_\Z(T,\boQ/\Z)$ is given by
the rule~\eqref{hom-contramodule} in~\S\ref{contratensor-subsecn}.

\medskip

In certain cases, we can do even better in that some of the
canonical functors $\D^\st(\fR\contra)\rarrow\D^\st(\fR\modl)$
and $\D^\st(\fA\discrl)\rarrow\D^\st(\fA\modl)$ are fully faithful
and that the corresponding equivalence~\eqref{eq:discr-tilting}
is simply a restriction of the standard adjunction
\begin{equation}\label{eq:derived-adj}
T\otimes^\boL_\fR{-}\:\D^\st(\fR\modl) \;\rightleftarrows\; \D^\st(\fA\modl) \;\:\!\boR\!\Hom_\fA(T,{-}).
\end{equation}

One such situation occurs when $A$ is an ordinary ring with discrete
topology (so $A\discrl=A\modl$). This observation
goes back to~\cite{Ba-eq,BMT}.
A tilting object $T\in A\modl$ is then,
by Corollary~\ref{tilting-objects-and-tilting-modules},
precisely an infinitely generated tilting module, whose
definition goes back to the papers~\cite{AC,Ba0}. That is,
$T$ is required to satisfy the conditions~\textup{(i)}
(finite projective dimension) and~\textup{(ii)}
(no self\+extensions) of Section~\ref{tilting-t-structure-secn}
and the free left $A$\+module $A$ must have a finite coresolution
$$
 0\lrarrow A\lrarrow T^0\lrarrow\dotsb\lrarrow T^r\lrarrow0
$$
by $A$\+modules $T^i$ belonging to the subcategory
$\Add(T)\subset A\modl$.

The tilting $A$\+module $T$ is called \emph{good}
(\cite[Section~13.1]{GTbook}, \cite{AC}, \cite{Ba0})
if the $A$\+modules $T^i$ in the above coresolution
%in the condition~(iii$_\mathrm{m}$)
can be chosen to be direct summands
of \emph{finite} direct sums of copies of~$T$.
 Obviously, replacing an arbitrary $n$\+tilting $A$\+module $T$ with
the $A$\+module $T^{(Y)}$ for a large enough set $Y$ produces a good
$n$\+tilting module with the same associated tilting t\+structure.
The key facts about good tilting modules are summarized in the following
lemma.

\begin{lem} \label{lem:good-tilting-module-hom}
 Let $T\in A\modl$ be a good $n$\+tilting module and\/ $\fR=\Hom_A(T,T)^\rop$.
 Then the right\/ $\fR$\+module $T$ is perfect, i.e., it admits a finite
resolution (of length~$n$) by finitely generated projective right\/ $\fR$\+modules.
 If\/ $\E=\bigcap_{i>0}\ker\Ext^i_A(T,-)\subset A\modl$ is the tilting
class corresponding to~$T$ (recall Theorem~\ref{tilting-cotorsion-pair-thm}
and Remark~\ref{rem:tilting-class}) and\/
$\F=\bigcap_{i>0}\ker\Tor_i^\fR(T,-)\subset\fR\modl$,
then the adjunction
\[ 
T\otimes_\fR{-}\:\fR\modl \;\rightleftarrows\; 
A\modl \;\:\!\!\Hom_\A(T,{-})
\]
restricts to an adjunction 
\[ 
T\otimes_\fR{-}\:\F \;\rightleftarrows\; 
\E \;\:\!\!\Hom_\A(T,{-}).
\]
Moreover, both the restricted functors are exact and 
the adjunction counit $T\ot_\fR\Hom_A(T,E)\rarrow E$
is an isomorphism for each $E\in\E$.
\end{lem}

\begin{proof}
The right $\fR$\+module $T$ is perfect by~\cite[Proposition~1.4]{BMT}.
Furthermore, we have $T\ot_\fR\Hom_A(T,E)\cong E$ and
$\Tor_i^\fR(T,\Hom_A(T,E))=0$ for each $E\in\E$ and $i>0$
by~\cite[Lemma~1.5\,(1--2)]{BMT}.
In particular, $\Hom_A(T,{-})$ restricts to an exact functor $\E\rarrow\F$.
On the other hand, each $F\in\F$ admits a free resolution
\[ \fR^{(I_n)}\rarrow\cdots\rarrow\fR^{(I_2)}
\rarrow\fR^{(I_1)}\rarrow F\rarrow0 \]
in $\fR\modl$ and all the syzygies of $F$ are in $\F$ again
by simple dimension shifting.
If we apply $T\ot_\fR{-}$, we obtain an exact sequence
\[ T^{(I_n)}\rarrow\cdots\rarrow T^{(I_2)}\rarrow T^{(I_1)}
\rarrow T\ot_\fR F\rarrow0 \]
in $A\modl$, so $T\ot_\fR F\in\E$ by
Lemma~\ref{tilting-cotilting-classes-lemma}~(a).
\end{proof}

The following proposition then extends results from \cite{BMT}.

\begin{prop} \label{prop:good-tilting-module}
 Let $T\in A\modl$ be a good $n$\+tilting $A$\+module and\/ $\B\subset
\D^\b(A\modl)$ be the tilting heart; so\/ $\B\simeq\fR\contra$, where\/
$\fR=\Hom_A(T,T)^\rop$.
For every derived category symbol\/ $\st=\b$, $+$, $-$,
$\varnothing$, $\abs+$, $\abs-$, $\ctr$, or\/~$\abs$, 
the forgetful functor\/ $\fR\contra\rarrow\fR\modl$ induces
a fully faithful functor\/ $\D^\st(\fR\contra)\rarrow\D^\st(\fR\modl)$
of (conventional or absolute) derived categories.
In particular, the functor
\[ \boR\!\Hom_A(T,-)\:\D^\st(A\modl)\lrarrow\D^\st(\fR\modl) \]
is fully faithful and restricts to an equivalence 
\[ \boR\!\Hom_A(T,-)\:\D^\st(A\modl)\lrarrow\D^\st(\fR\contra). \]
\end{prop}

\begin{proof}
Let $\E=\bigcap_{i>0}\ker\Ext^i_A(T,-)\subset A\modl$ and
$\F=\bigcap_{i>0}\ker\Tor_i^\fR(T,-)\subset\fR\modl$ as above.
Then $\D^\st(\E)\simeq\D^\st(A\modl)$
by Theorem~\ref{exotic-derived-resolution-equivalences}.
Moreover, the same arguments imply that
$\D^\st(\F)\simeq\D^\st(\fR\modl)$. We only need that
$\F$ has similar closure properties in $\fR\modl$ to those
mentioned in Lemma~\ref{cotilting-class-resolving-finite-dim}:
it is closed under summands, infinite products, extensions, kernels
of epimorphisms, and contains all projective $\fR$\+modules.
Finally, since any $n$\+th syzygy of an $\fR$\+module is in $\F$,
each $Y\in\fR\modl$ admits an exact sequence
\[
 0 \lrarrow F_n \lrarrow \dotsb \lrarrow F_1\lrarrow F_0
 \lrarrow Y \lrarrow 0,
\]
where $F_0$, $F_1$,~\dots, $F_n \in \F$.

Now we observe that the adjunction between $\E$ and $\F$
from Lemma~\ref{lem:good-tilting-module-hom} lifts to an adjunction
\begin{equation} \label{eq:adj-modules}
 T\otimes^\boL_\fR{-}\:\D^\st(\fR\modl)\simeq\D^\st(\F)
 \;\rightleftarrows\;
 \D^\st(\E)\simeq\D^\st(A\modl) \;\:\!\boR\!\Hom_A(T,{-}).
\end{equation}
(we may use the characterization of adjunctions
in~\cite[Theorem~IV.2~(v)]{MacLane} here).
Since the counit $T\ot^\boL_\fR\boR\!\Hom_A(T,E^\bu)\rarrow E^\bu$
is also an isomorphism for each $E^\bu\in\D^\st(\E)$
by Lemma~\ref{lem:good-tilting-module-hom}, the functor
$\boR\!\Hom_A(T,-)\:\D^\st(A\modl)\rarrow\D^\st(\fR\modl)$
is fully faithful.
In view of the equivalence $\boR\!\Hom_A(T,-)\:\D^\st(A\modl)
\rarrow\D^\st(\fR\contra)$
given by Corollary~\ref{conv-abs-derived-equivalence}, also the canonical
functor $\D^\st(\fR\contra)\rarrow\D^\st(\fR\modl)$ is fully faithful.
\end{proof}

\begin{rem}
There are various other situations where~\eqref{eq:discr-tilting}
arises as the restriction of~\eqref{eq:derived-adj} for some choices
of the symbol $\st$.

On one hand, Theorem~\ref{kappa-ary-dual-gabriel-popescu} has
been improved in~\cite[Theorem 4.10]{Pperp}. It implies that for any
topological ring $\fR$ with a base of neighborhoods
of zero formed by open right ideals and a set $Y$ of the cardinality
not smaller than that of a base of neighborhoods of zero in $\fR$,
the Morita equivalent topological ring
$\fS=\Hom^\fR(\fR[[Y]],\fR[[Y]])$ has the property that the forgetful functor
\[ \D^-(\fS\contra)\lrarrow\D^-(\fS\modl) \]
is fully faithful. In this case, the contratensor and usual tensor
product over $\fS$ are identified in the sense
of Lemma~\ref{contratensor-vs-tensor}.

The forgetful functor $\D^\st(\fA\discrl)\lrarrow\D^\st(\fA\modl)$
is also known to be fully faithful in some cases.
This happens for $\st=\b$, $+$, $-$, or $\varnothing$ if:
\begin{enumerate}
\item $\fA$ is a commutative ring, $\fI\subset\fA$ is a weakly proregular
finitely generated ideal and the powers of $\fI$ form a base of neighborhoods
of zero (see \cite[Corollary~1.4]{Pmc}), or
\item $\fA$ is a commutative ring, $S\subset\fA$ is a multiplicative
subset satisfying certain conditions and $\{\fA s\mid s\in S\}$ is a base
of neighborhoods of zero (see \cite[Theorem 6.6~(a)]{PMat}).
\end{enumerate}

In both these situations, there are examples
of tilting equivalences in our sense. These are constructed
using \cite[Theorem 19.1]{BP-properfect}
from certain naturally occurring flat ring epimorphisms originating in $\fA$
(namely, the localization $S^{-1}\fA$ in case~(2), and
the epimorphism corresponding to the specialization closed
subset $V(\fI)\subset\operatorname{Spec}\fA$ in the sense of~\cite[Theorem 1.2]{AHMSVT}, provided that it exists,
in case~(1)).
In fact, Example~\ref{torsion-modules-1-tilting-example} is a special case of situation~(2).
\end{rem}

We conclude the section by presenting an explicit example
which illustrates Proposition~\ref{prop:good-tilting-module}.
It extends Example~\ref{torsion-modules-1-tilting-example} in a different direction
in the case of commutative Noetherian rings of Krull dimension one.

\begin{ex}\label{Matlis-1-tilting-example}
Let $R$ be a commutative Noetherian ring of Krull dimension~$1$ and let $S \subset R$ be a multiplicative subset consisting of nonzero-divisors. Then $S^{-1}R$ is a flat $R$\+module, hence its projective dimension is at most $1$ by~\cite[Corollaire II.3.2.7]{RG71}. In particular $T = S^{-1}R\oplus S^{-1}R/R$ is a good $1$-tilting module by~\cite[Theorem 14.59]{GTbook}.
 
Let now $P=\bigcup_{s\in S}V(s) = \{\fp\mid\fp\cap S\ne\varnothing\}$ be the set of all primes that intersect~$S$. Since $S$ contains nonzero-divisors only, $P$ must consist of maximal ideals only (if $\fp\in\Spec R$ is not maximal, it is an associated prime of the $R$\+module $R$ and hence all its elements are zero-divisors). Moreover, it follows that an $R$\+module is $S$-torsion (i.~e., $S^{-1}M=0$) if and only if the support of $M$ is contained in $P$, and in such a case
\[ M \cong \bigoplus_{\fp\in P} M_\fp \]
(see for example~\cite[Lemma 13.1]{Pcadj}). In particular, $S^{-1}R/R \cong \bigoplus_{\fp\in P}S^{-1}R_\fp/R_\fp$ and, by~\cite[Lemma 13.5]{Pcadj}, for each $\fp\in P$ there exists $s_\fp\in S$ such that $S^{-1}R_\fp = R_\fp[s_\fp^{-1}]$.

We know by the previous discussion that there are derived equivalences
\begin{equation} \label{derived-equivalences-for-Matlis-1-tilting}
\boR\Hom_R(T,-)\: \D^\st(R\modl) \lrarrow \D^\st(\fR\contra)
\end{equation}
for the topological ring $\fR=\Hom_R(T,T)^\rop$ and various choices of the symbol $\st$. It turns out that $\fR$ has a nice description as a certain $2\times 2$ triangular matrix ring. To this end, we have $\Hom_R(S^{-1}R,S^{-1}R)=S^{-1}R$, $\Hom_R(S^{-1}R/R,S^{-1}R)=0$, and
\[ \Hom_R(S^{-1}R/R,S^{-1}R/R)=\prod_{\fp\in P}\Hom_R(S^{-1}R_\fp/R_\fp,S^{-1}R_\fp/R_\fp)=\prod_{\fp\in P}\widehat{R_\fp}, \]
since for each $\fp\in P$, we have $\Hom_R(R_\fp[s_\fp^{-1}]/R_\fp,R_\fp[s_\fp^{-1}]/R_\fp)=\varprojlim R/s_\fp^n R$, which is the usual adic completion of $R_\fp$ (see Example~\ref{torsion-modules-1-tilting-example} and~\cite[Proposition 3.2]{PMat}). If we denote $\fA_S=\Hom_R(S^{-1}R,S^{-1}R/R)$, then as rings
\begin{equation}\label{topological-ring-for-Matlis-1-tilting}
\fR = \left(\begin{matrix} S^{-1}R & \fA_S \\ 0 & \prod_{\fp\in P}\widehat{R_\fp} \end{matrix}\right).
\end{equation}

A closer look at $\fA_S$ reveals that it is a version of the ring of finite adeles. To see this, let us apply $\Hom_R(-,S^{-1}R/R)$ to the short exact sequence $0\to R\to S^{-1}R\to S^{-1}R/R\to 0$. We obtain a short exact sequence
\[ 0 \rarrow \prod_{\fp\in P}\widehat{R_\fp} \rarrow \fA_S \rarrow S^{-1}R/R \rarrow 0. \]
We claim that $\fA_S\cong(\prod_{\fp\in P}\widehat{R_\fp})\otimes_R S^{-1}R$, which in fact gives $\fA_S$ a natural ring structure. To see this, express $S^{-1}R=\varinjlim_{s\in S}R_s$, where $R_s=R$ for each~$s$ and the maps in the direct system are given by $t\cdot-\:R_s\to R_{st}$ for $s,t\in S$. Then
\[ \fA_S=\Hom_R(S^{-1}R,S^{-1}R/R)=\varprojlim_{s\in S}\Hom_R(R_s,S^{-1}R/R). \]
The latter inverse system can be identified with the inverse system
$\bigl(\fA'\big/s\prod_{\fp\in P}\widehat{R_\fp}\bigr)_{s\in S}$, where
$\fA'=S^{-1}\prod_{\fp\in P}\widehat{R_\fp}$ and the maps are simply projections.
It follows that $\fA_S\cong\fA'$, which proves the claim.

 To summarize, given a commutative Noetherian ring $R$ of Krull
dimension~$1$ and a multiplicative set of nonzero-divisors in it, we have
derived equivalences~\eqref{derived-equivalences-for-Matlis-1-tilting},
where $\st$~can be one of $\b$, $+$, $-$, $\varnothing$,
$\abs+$, $\abs-$, $\ctr$, or\/~$\abs$.
 The ring structure of the topological ring $\fR$ has an explicit
description via~\eqref{topological-ring-for-Matlis-1-tilting},
and the topology is also the obvious one: $S^{-1}R$ is discrete,
$\prod_{\fp\in P}\widehat{R_\fp}$ carries the product topology
of the adic topologies, and the adelic ring
$\fA_S=S^{-1}\prod_{\fp\in P}\widehat{R_\fp}$ has a canonical
locally pseudo-compact topology (in the sense of~\cite[\S4.3]{Gab}).
\end{ex}

\Section{Tilting Hearts Which are Categories of Contramodules}
\label{weakly-finitely-generated-secn}

It turns out that there are several cases where the tilting equivalences from Corollary~\ref{conv-abs-derived-equivalence} are in fact of the form~\eqref{eq:discr-tilting}, but this is not readily visible, since $\A$ and $\B$ are not \emph{a priori} given as categories of discrete modules or contramodules. These will be discussed in the final section.

Here we focus on the following subproblem: Suppose that
we are in the situation of the tilting-cotilting correspondence
and $T\in\A$ is a tilting object. We know that $\B=\boT\modl$ for some additive monad $\boT$ by Corollary~\ref{cor:B-as-Tmod}. Under what circumstances does there exist a topological ring $\fR$ such that $\B=\fR\contra$?

Although a full characterization of additive monads originating from topological
rings does not seem to be available, we aim at obtaining useful sufficient criteria.

% The two dual classes of arbitrary complete, cocomplete abelian
%categories $\A$ with an injective cogenerator and complete, cocomplete
%abelian categories $\B$ with a projective generator appearing in
%the theory developed in Section~\ref{tilting-cotilting-secn} are
%often too broad and abstract.
%It would be interesting to know which specific subclasses of abelian
%categories $\A$ or $\B$ on one of the sides does the tilting-cotilting
%correspondence assign to various natural subclasses of abelian
%categories on the other side.
%Are there any natural ways to strengthen the conditions on both
%the categories $\A$ and $\B$ in
%Corollary~\ref{tilting-cotilting-correspondence-cor} such that
%the correspondence remains one-to-one?
%
%Assume that $\A$ is a Grothendieck abelian category; what, precisely,
%can one then say about the abelian category~$\B$?
%Assume that $\A$ is a locally presentable abelian category; what is,
%precisely, the class of abelian categories $\B$ corresponding to
%such categories $\A$ under the tilting-cotilting correspondence?
%
%We do not know the answers to these questions (one of the reasons for
%this may be that we do not yet understand what are the natural
%subclasses of the class of all the categories of models of additive
%$\kappa$\+ary algebraic theories).
%Some partial results in this direction are presented in
%this section.

\subsection{Hearts which are models of algebraic theories}

First we give an easy sufficient condition for the category $\B$ to be at least the category of models of an additive\/ $\kappa$\+ary algebraic theory.
It is based on the following refinement of Proposition~\ref{copowers-kleisli}.

\begin{prop} \label{addm-locally-presentable}
 Let\/ $\A$ be a locally presentable abelian category, $M\in\A$ be
an object, and\/ $\Add(M)\subset\A$ be the full additive subcategory
formed by the direct summands of all coproducts of copies of
the object $M$ in the category\/~$\A$.
 Then there exists a (naturally defined) category of models\/ $\B$
of an additive\/ $\kappa$\+ary algebraic theory such that
the category\/ $\Add(M)$ is equivalent to the full subcategory\/
$\B_\proj\subset\B$ of projective objects in\/~$\B$.
\end{prop}

\begin{proof}
 Notice first of all that any locally presentable category is
complete and cocomplete~\cite[Remark~1.56]{AR}.
 Consider the additive monad $\boT\:X\longmapsto\Hom_\A(M,M^{(X)})$
on the category of sets as in Proposition~\ref{copowers-kleisli}.
 Let $\B$ be the category of $\boT$\+modules; then $\B$
is an abelian category with a natural projective generator $P$
corresponding to the free $\boT$\+module with one generator
and there is an equivalence of additive
categories $\Add(M)\simeq\B_\proj$ taking the object $M\in\Add(M)$
to the object $P\in\B_\proj$ and the object $M^{(X)}\in\Add(M)$ to
the object $P^{(X)}\in\B_\proj$.

 It remains to show that the object $P\in\B$ is abstractly
$\kappa$\+small for some cardinal~$\kappa$.
 Indeed, let $\kappa$ be the presentability rank of the object
$M\in\A$.
 Then, in particular, the object $M\in\A$ is abstractly $\kappa$\+small,
i.~e., every morphism $M\rarrow M^{(X)}$ in $\A$ factorizes through
the natural embedding $M^{(Z)}\rarrow M^{(X)}$ for some subset of
indices $Z\subset X$ of the cardinality smaller than~$\kappa$.
 In view of the equivalence of categories $\Add(M)\simeq\B_\proj$,
the desired assertion follows.
\end{proof}

\begin{cor} \label{tilting-a-locally-presentable}
 Suppose that $\A$ is a locally presentable abelian category
with an injective cogenerator.
If $T\in\A$ is an $n$\+tilting object and\/ $\B$ is the heart of the
tilting t\+structure on\/ $\D^\b(\A)$, then $\B$ is the category
of models of an additive\/ $\kappa$\+ary algebraic theory for some $\kappa$.
% The left exact functor\/ $\Psi\:\A\rarrow\B$ from Section~\ref{derived-equivalences-secn}
%can be computed as the functor\/ $\Hom_\A(T,{-})$
%at the level of underlying sets.
\end{cor}

\begin{proof}
This immediately follows by combining Remark~\ref{rem:tilting-contra} with Proposition~\ref{addm-locally-presentable}.
% There are enough projective objects in the tilting heart $\B$,
%and the full subcategory $B_\proj$ of projective objects in $\B$ is
%equivalent to the category $\Add(T)\subset\A$.
% So the abelian category $\B$ is nothing but the category of modules
%over the monad $\boT\:X\longmapsto\Hom_\A(T,T^{(X)})$ described in
%the proof of Proposition~\ref{addm-locally-presentable}.
% The second assertion follows from \S\ref{monads-and-equiv-subsecn}.
\end{proof}

 In other words, we have shown that $\B$ is a locally presentable
category whenever the category $\A$ is locally presentable.
 We do not know whether the converse assertion is true.

If we wish to get closer to contramodules over topological rings, note that $\fR\contra$ has the additional property that 
for every family of projective objects $P_\alpha$ the canonical morphism\/
$\coprod_\alpha P_\alpha\rarrow\prod_\alpha P_\alpha$ is a monomorphism (see the final paragraph of \S\ref{contramodules-over-top-rings-subsecn}). This necessary condition is also transferred well by the tilting equivalences.

\begin{prop} \label{tilting-a-grothendieck}
Suppose that $\A$ is an abelian category with set-indexed products, an injective cogenerator and an $n$\+tilting object $T$. If the canonical morphism $\coprod_\alpha T_\alpha\rarrow\prod_\alpha T_\alpha$ is a monomorphism in $\A$ for every family
of objects $T_\alpha\in\Add(T)$ (e.g.\ if $\A$ is a Grothendieck category), then the canonical morphism\/ $\coprod_\alpha P_\alpha\rarrow\prod_\alpha P_\alpha$ is a monomorphism for every family
of projective objects $P_\alpha\in\B$
in the heart\/ $\B$ of the tilting t\+structure given by $T$.
\end{prop}

\begin{proof}
Consider the category $\E=\A\cap\B$ studied in Section~\ref{derived-equivalences-secn} and let $i\:\coprod_\alpha P_\alpha\rarrow\prod_\alpha P_\alpha$
be a morphism as in the statement.
Since $T_\alpha:=P_\alpha\in\B_\proj=\Add(T)$ for each $\alpha$, both ends of $i$ actually
belong to $\E$ by Lemma~\ref{tilting-cotilting-classes-co-product-closed}(a) and hence we
have a short exact sequence
\[ 0\rarrow\coprod_\alpha P_\alpha\rarrow\prod_\alpha P_\alpha
\rarrow\prod_\alpha P_\alpha\Big/\coprod_\alpha P_\alpha\rarrow0 \]
in $\E\subset\A$ by Lemma~\ref{tilting-class-coresolving-finite-dim}(a). 
Thus, the same exact sequence must exist in $\B$. The conclusion follows from Remark~\ref{E-is-AB4-and-AB4*} as the coproduct and the product displayed is actually also the coproduct and product in $\B$, respectively.

If $\A$ is a Grothendieck category, then $i$ is a filtered
colimit of split inclusions (see~\cite[Corollary~III.1.3]{Mit}), hence always a monomorphism in $\A$.
\end{proof}

%\begin{cor} \label{tilting-a-grothendieck}
% Under the assumptions of Proposition~\textup{\ref{addm-grothendieck}},
%$\B$ is the category of models of an additive\/ $\kappa$\+ary
%algebraic theory for some $\kappa$ having the additional property that
%for every family of projective objects $P_\alpha\in\B$ the natural morphism\/
%$\coprod_\alpha P_\alpha\rarrow\prod_\alpha P_\alpha$ is a monomorphism
%in\/~$\B$.
%\end{cor}
%
%\begin{proof}
% This immediately follows from Proposition~\ref{addm-grothendieck} and Corollary~\ref{tilting-a-locally-presentable}.
%\end{proof}
%
%The situation where the tilting object in $\A$, or equivalently the projective generator of $\B$, is compact in $\D(\A)\simeq\D(\B)$, is covered by~\cite{St}.
%
%\begin{thm}
% Let $\A$ be an abelian category with set-indexed products,
%an injective cogenerator, and an $n$\+tilting object.
% Let $\B$ be the heart of the corresponding tilting t\+structure.
%
% If\/ $\B=B\modl$ is the category of modules over an associative ring $B$,
%then\/ $\A$ is a Grothendieck abelian category.
%
% Conversely, if\/ $\A$ is a Grothendieck abelian category and
%the tilting object $T\in\A$ is a compact object of the derived
%category\/ $\D(\A)$, then\/ $\B\simeq B\modl$ is the category of
%modules over the ring $B=\Hom_A(T,T)^\rop$.
%\end{thm}
%
%\begin{proof}
% These are the main results of the paper~\cite{St}.
% The first assertion is~\cite[Theorem~6.2]{St}; the second one
%follows from the discussion in~\cite[Section~1]{St}.
%\end{proof}

\subsection{Locally weakly finitely generated abelian categories}
\label{lwfg-subsecn}

Now we attack the problem more directly and seek an answer to the following question: If $\A$ is an abelian category and $M\in\A$, can we find a natural topology on $\Hom_\A(M,M)^\rop$ with a base of neighborhoods of zero formed by right ideals? For module categories, we presented such a topology in Theorem~\ref{addm-theorem}. Here we generalize this result to a broad class of abelian categories including for instance all locally finitely presentable Grothendieck abelian categories.

 Let $\C$ be a cocomplete abelian category and $\lambda$ be a regular
cardinal.
 Let us call an object $C\in\C$ \emph{weakly\/ $\lambda$\+generated}
if every morphism $C\rarrow\coprod_{x\in X} D_x$ from $C$ into
the coproduct of a family of objects $D_x$, \,$x\in X$ in $\C$
factorizes through the natural embedding $\coprod_{z\in Z} D_z\rarrow
\coprod_{x\in X}D_x$ of the coproduct of a subfamily indexed by
a subset $Z\subset X$ of the cardinality smaller than~$\lambda$.
 Any quotient object of a weakly $\lambda$\+generated object in $\C$
is weakly $\lambda$\+generated.
 The class of all weakly $\lambda$\+generated objects in $\C$ is also
closed under extensions and $\lambda$\+small coproducts.
 Any $\lambda$\+generated object in the sense
of~\cite[Definition~1.67]{AR} is weakly $\lambda$\+generated.
 A weakly $\omega$\+generated object (where $\omega$~denotes
the countable cardinal) is called \emph{weakly finitely generated}.

\begin{rem} \label{dually-slender}
 Weakly finitely generated objects of the category of left $R$\+modules
$\C=R\modl$ are known as \emph{dually slender} (or ``small'')
modules~\cite{CT94,EGT,Zem}.
 An associative ring $R$ is said to be \emph{left steady} if all
dually slender left $R$\+modules are finitely generated.
 All left Noetherian rings, left perfect rings, and countable
commutative rings are left steady.
 On the other hand, the free associative algebra with two generators
$R=k\{x,y\}$ over a field~$k$ is \emph{not} steady; in fact, there
is a proper class of dually slender $R$\+modules, as any injective
$R$\+module is dually slender~\cite[Lemma~3.2]{Zem}.
\end{rem}

 A cocomplete abelian category $\C$ is called \emph{locally weakly
$\lambda$\+generated} if it is generated by its weakly $\lambda$\+generated
objects, that is, for any object $M\in\C$, the minimal subobject of $M$
containing the images of all morphisms into $M$ from weakly
$\lambda$\+generated objects of $\C$ coincides with~$M$.
 Since the class of all weakly $\lambda$\+generated objects of $\C$
is closed under quotients (and also $\lambda$\+small coproducts), this
condition simply means that every object in $\C$ is
the ($\lambda$\+directed) union of its weakly $\lambda$\+generated
subobjects.
 (One could add the condition that $\C$ has a set of generators to
this definition; and then it would follow that $\C$ has a set of weakly
$\lambda$\+generated generators;
cf.\ Remark~\ref{lwfg-single-generator-remark} below;
but we will not need this.)
 A locally weakly $\omega$\+generated category is called \emph{locally
weakly finitely generated}.

 It turns out that locally weakly finitely generated
abelian categories automatically satisfy Grothendieck's axiom~Ab5,
i.~e., have exact functors of filtered colimits.

\begin{prop} \label{lwfg-exact-direct-lim}
Let $(E_i)_{i\in I}$ be a direct system of subobjects of $M\in\C$,
where $\C$ is a locally weakly finitely generated category.
 Then the colimit\/ $\colim_{i\in I}E_i$ is again a subobject of $M$,
i.~e., the colimit map\/ $\colim_{i\in I}E_i\rarrow M$ is monic.
\end{prop}

\begin{proof}
 Put $N=\colim_{i\in I}E_i$ and consider the cocone
$(f_{E_i}\:E_i\to N)$.
 We denote by $f\:\coprod_{i\in I}E_i\rarrow N$ and
$p\:\coprod_{i\in I}E_i\rarrow M$ the canonical morphisms.
 We have to show that the epimorphism~$f$ annihilates the kernel
of the morphism~$p$.
 Since $p$ factors through $f$ by the universal property of the colimit, 
this will mean that the kernels of $p$ and $f$ are equal and
the induced map $N\rarrow M$ is injective, as desired.

 Let $b\:B\rarrow\coprod_{i\in I}E_i$ be a morphism from a weakly
finitely generated object $B$ with the image lying in the kernel of~$p$.
 It suffices to show that $fb=0$ for every such~$b$.
 The morphism~$b$ factorizes through the coproduct of a finite subset
of objects $\coprod_{j=1}^m E_j\subset\coprod_{i\in I} E_i$.
 Denote by~$b'$ the related morphism $B\rarrow\coprod_{j=1}^m E_j$.
 Choose $k\in I$ such that $E_j\subset E_k$ for all $1\le j\le m$,
% Let $F\subset M$ denote the sum of the subobjects $E_j\subset M$;
%then $F$ is also a weakly finitely generated subobject in~$M$.
and denote by $q\:\coprod_{j=1}^m E_j\rarrow E_k$ the natural morphism.
 Then $qb'=0$, because $pb=0$ and $E_k$ is a subobject in~$M$.
 Let $g\:\coprod_{j=1}^m E_j\rarrow N$ denote the morphism with
the components $f_{E_j}\:E_j\rarrow N$.
 Then $g=f_{E_k}q$, since the system of morphisms
$(f_{E_i}\:E_i\to N)_{i\in I}$ is compatible.
 Thus $fb=gb'=f_{E_k}qb'=0$.
\end{proof}

\begin{cor} \label{lwfg-ab5}
 In any locally weakly finitely generated abelian category,
the functors of filtered colimit are exact.
\end{cor}

\begin{proof}
 According to~\cite[Proposition~III.1.2 and Theorem~III.1.9]{Mit},
this is an equivalent reformulation of
Proposition~\ref{lwfg-exact-direct-lim}.
\end{proof}

%\begin{cor} \label{lwfg-ab5-sum-product-mono}
% Let\/ $\C$ be a locally weakly finitely generated abelian category.
% Then \par
%\textup{(a)} the functors of filtered colimit are exact in\/~$\C$; \par
%\textup{(b)} if\/ $\C$ is complete, then for any family of objects
%$C_x\in\C$ the natural morphism\/ $\coprod_x C_x\rarrow
%\prod_x C_x$ is a monomorphism.
%\end{cor}
%
%\begin{proof}
 %According to~\cite[Proposition~III.1.2 and %Theorem~III.1.9]{Mit},
%part~(a) is an equivalent reformulation of
%Proposition~\ref{lwfg-exact-direct-lim}.
% Part~(b) is provided by~\cite[Corollary~III.1.3]{Mit}.
%\end{proof}
%
%\begin{cor} \label{lwfg-sum-product-mono}
% In a complete locally weakly finitely generated abelian category\/
%$\C$, for any family of objects $C_x\in\C$, the natural morphism\/
%$\coprod_x C_x\rarrow\prod_x C_x$ is injective.
%\end{cor}

%\begin{proof}
 %It suffices to show that for any weakly finitely generated object
%$B\in\C$ and any nonzero morphism $b\:B\rarrow\coprod_x C_x$
%the composition $B\rarrow \coprod_{x\in X} C_x\rarrow\prod_{x\in X}C_x$
%is also nonzero.
 %Let $X$ denote the set of indices~$x$; then there exists a finite
%subset $Z\subset X$ such that the morphism~$b$ factorizes through
%the embedding of a direct summand $\coprod_{z\in Z}C_z\rarrow
%\coprod_{x\in X}C_x$.
 %Now we have $\coprod_{z\in Z}C_z=\prod_{z\in Z}C_z$ and the morphism
%$\prod_{z\in Z}C_z\rarrow\prod_{x\in X}C_x$ is also the embedding of
%a direct summand.
 %Hence if the morphism $B\rarrow \coprod_{z\in Z}C_z$ is nonzero, then
%so is the morphism $B\rarrow\prod_{x\in X}C_x$.
%\end{proof}

\begin{rem} \label{lwfg-single-generator-remark}
Suppose that $\C$ is a locally weakly finitely generated category with
a generator $G$ and let $\cR$ be the full subcategory of $\C$ formed 
by all the quotient objects of finite coproducts of weakly finitely 
generated subobjects of $G$.
Then $\cR$ is essentially small by~\cite[Proposition IV.6.6]{Sten}
and the restricted Yoneda functor
\begin{align*}
h_\cR\:\C&\rarrow\cR^\rop\modl,\\
X&\longmapsto\Hom_\C(-,X)|_\cR,
\end{align*}
has a left adjoint $\Delta\:\cR^\rop\modl\rarrow\C$ by~\cite[Proposition 1.27]{AR}. Here, $\cR^\rop\modl$
stands for the category of \emph{$\cR^\rop$-modules}, i.~e., of additive functors $\cR^\rop\to\Z\modl$.

It turns out that the counit of adjunction $\varepsilon\:\Delta\circ h_\cR\rarrow 1_\C$
is a natural equivalence and hence $h_\cR$ is fully faithful. Indeed, for any object $M\in\C$, the morphism~$\varepsilon_M$
admits an explicit description.
If $D_M\:\cR/M\to\C$ is the canonical diagram
of $M$ (see~\cite[Definition 0.4]{AR}), then $\varepsilon_M$ is simply the colimit morphism
$\colim_{(g\:E\to M)\in\cR/M}E\to M$. Since any morphism $E\to M$ with $E\in\cR$ factors through its image $F\in\cR$ as
$E\to F\subseteq M$, the direct  system of all subobjects of $M$ 
belonging to $\cR$ is cofinal in~$D_M$.
As $M$ is the union of its subobjects belonging to $\cR$,
the morphism~$\varepsilon_M$ is an isomorphism by
Proposition~\ref{lwfg-exact-direct-lim}, as required.

In particular, $\C$ identifies with a coproduct-closed full reflective subcategory of $\cR^\rop\modl$ in this case, and $\C$ is automatically complete.
\end{rem}

\begin{exs} \label{exs:lwfg}
 Any locally finitely presentable (Grothendieck) abelian category is
locally weakly finitely generated.
 In particular, any locally Noetherian or locally coherent
Grothendieck category is locally weakly finitely generated.

 On the other hand, a Grothendieck abelian category in general does
\emph{not} need to be locally weakly finitely generated.
 Here is a counterexample: let $\boQ$ be the set of all rational
numbers, viewed as a topological space with the topology induced
from its embedding into the real line.
 Let $k$ be a (discrete) field.
 Then the category of sheaves of $k$\+vector spaces over $\boQ$ is
not locally weakly finitely generated.

 Indeed, let us show that the constant sheaf $\underline{k}_\boQ$
with the stalk~$k$ over $\boQ$ has no nonzero weakly finitely
generated subobjects.
 For any nonzero subsheaf $F\subset\underline{k}_\boQ$, there
exists a nonempty open subset $U\subset\boQ$ such that $F$ contains
the constant section $\underline{1}_U$ of $\underline{k}_\boQ$
over~$U$.
 Let $D\subset U$ be an infinite discrete subset that is closed in
$\boQ$ (e.~g., a sequence of elements of $U$ converging to
an irrational number).
 Then the composition $F\rarrow\underline{k}_\boQ\rarrow
\underline{k}_D$ is a sheaf epimorphism (where the constant sheaf
$\underline{k}_D$ on $D$ is viewed as a sheaf over $\boQ$ using
the extension by zero).
 It remains to observe that the sheaf $\underline{k}_D$ is
the coproduct of an infinite collection of nonzero objects
(namely, the skyscraper sheaves at the points of~$D$).
% Here is a counterexample: let $K$ be the Cantor set $K=\{0,1\}^\Z$
%endowed with its natural (product) topology, and let $k$ be a vector
%space.
% Then the category of sheaves of $k$\+vector spaces over $K$ is
%not locally weakly finitely generated.
% Indeed, the constant sheaf $\underline{k}_K$ with the stalk $k$ over
%$K$ has no nonzero weakly finitely generated subobjects.
\end{exs}

Now we get to the main observation in the subsection, which is a generalization of Theorem~\ref{addm-theorem} to any locally weakly finitely generated abelian
category~$\C$.
Let $M\in\C$ be a fixed object.
We will endow the ring $\fR=\Hom_\C(M,M)^\rop$ with the topology in
which a base of neighborhoods of zero is formed by the annihilator
ideals $\Ann(E)=\{\,f\in\Hom_\C(M,M)\mid f|_E=0\,\}\subset\fR$ of weakly finitely generated subobjects
$E\subset M$.
These are left ideals in $\Hom_\C(M,M)$ and right ideals in~$\fR$.

\begin{thm} \label{lwfg-addm-theorem}
Let\/ $\C$ be a locally weakly finitely generated abelian category
and $M\in\C$ be an object which is the direct union of a set of its weakly finitely generated subobjects (this is automatic if\/ $\C$ has a generator).
Then\/ $\fR=\Hom_\C(M,M)^\rop$ is a complete and separated
topological ring and the category\/ $\Add(M)$ formed by the direct summands of infinite coproducts of
copies of~$M$ is equivalent to the category\/
$\fR\contra_\proj$ of projective left\/ $\fR$\+contramodules.
%
%Moreover, let $Y$ be a set of the cardinality greater or equal to
%the cardinality of some base of neighborhoods of zero in\/~$\fR$.
%Put $L=M^{(Y)}$ and consider the topological ring\/
%$\fS=\Hom_\C(L,L)^\rop$.
%Then the category\/ $\Add(M)$ is equivalent to the additive category\/
%$\fS\contra_\proj$ and the forgetful functor between the abelian
%categories\/ $\fS\contra\rarrow\fS\modl$ is fully faithful.
\end{thm}

\begin{proof}
If $\C$ has a generator, the theorem follows immediately from Theorem~\ref{addm-theorem} and Remark~\ref{lwfg-single-generator-remark}.
We will now present a general and direct argument analogous to that for Theorem~\ref{addm-theorem}.

% In order to show that the multiplication in a topological ring $\fR$
%is continuous with respect to a topology with a base of neighborhoods
%of zero formed by some right ideals, it suffices to check that for every
%open right ideal $\fI\subset\fR$ and every element $r\in R$ there
%exists an open right ideal $\fJ\subset\fR$ such that $r\fJ\subset\fI$.
%
First of all, let $\fI=\Ann(E)$ be the annihilator of
a weakly finitely generated subobject $E\subset M$ and $r\:M\rarrow M$
be a morphism in the category~$\C$.
Then $r\fJ\subset\fI$ in~$\fR$, where $\fJ=\Ann(F)$ and
$F=rE\subset M$ is the weakly finitely generated quotient object of~$E$ which is the image of the composition $E\rarrow M
\overset r\rarrow M$.
Hence, the multiplication in $\fR$ is continuous and $\fR$ is a topological ring.

%The topology is separated, that is $\bigcap_E\Ann(E)=0$ in $\fR$,
%because $M$ is the union of its weakly finitely generated
%subobjects~$E$.
%Let us show that the topology is complete.
Next we claim that $\fR$ is complete and separated, that is the map
$\fR\rarrow\varprojlim_E\fR/\Ann(E)$ is bijective.
Here $\Ann(E)$ is nothing but the group of morphisms $\Hom_\C(M/E,M)$.
We will prove the following more general fact: For any object $N\in\C$,
the natural map
\[
\Hom_\C(M,N)\lrarrow\varprojlim\nolimits_E\Hom_\C(M,N)/\Hom_\C(M/E,N),
\]
where the projective limit is taken over all the weakly finitely
generated subobjects $E\subset M$, is an isomorphism.
Indeed, the exact sequence
$$
0\lrarrow\Hom_\C(M/E,N)\lrarrow\Hom_\C(M,N)\lrarrow\Hom_\C(E,N)
$$
shows that the quotient group $\Hom_\C(M,N)/\Hom_\C(M/E,N)$ is
isomorphic to the subgroup in $\Hom_\C(E,N)$ consisting of all
the morphisms $E\rarrow N$ that can be extended to a morphism
$M\rarrow N$. Hence, elements of the inverse limit
$\varprojlim\nolimits_E\Hom_\C(M,N)/\Hom_\C(M/E,N)$
identify with compatible systems of morphisms $E\rarrow N$ in $\C$
defined for all the weakly finitely generated subobjects $E\subset M$.

Now we encounter a difference from Theorem~\ref{addm-theorem} in that we are not guaranteed that $M$ does not have a proper class of such subobjects~$E$. However, we will show that $M$ is still a colimit of the class of all such subobjects. By showing this, we at once prove the claim as well.

To this end, fix $N\in\C$ and a compatible system of morphisms
$f_E\:E\rarrow N$ in defined for all the weakly finitely generated subobjects $E\subset M$.
Recall that we assume the existence of a set $\cF'$ of weakly finitely generated subobjects in $M$ such
that no proper subobject in $M$ contains all $F\in\cF'$.
Moreover, since there is only a set of subsets in the set $\Hom_\C(M,N)$
(as all our abelian categories have Hom sets),
we can form a set $\cF''$ of weakly finitely generated subobjects
in $M$ such that for each weakly finitely generated subobject
$E\subset M$ there exists $F''\in\cF''$ for which the two subgroups
$\Hom_\C(M/E,N)$ and $\Hom_\C(M/F'',N)$ in $\Hom_\C(M,N)$
coincide.
 Let $\cF$ denote the closure of $\cF'\cup\cF''$ with respect to
the operation of the passage to a finite sum of subobjects in~$M$.
 Then $\cF$ is still a set of weakly finitely generated subobjects
in~$M$, and $M$ is the direct union of all $F\in\cF$.
 Therefore, $M=\colim_{F\in\cF}F$ by Proposition~\ref{lwfg-exact-direct-lim}.
 Thus, there is a unique $h\:M\rarrow N$ such that $h|_F=f_F$ 
for all $F\in\cF$.
 It remains to verify that $h|_E=f_E$ for all weakly finitely generated subobjects $E\subset M$.
 Given $E$, consider $F''\in\cF$ related to $E$ as above.
 Then $E+F''$ is also a weakly finitely generated subobject in $M$, and
the three subgroups $\Hom_\C(M/E,N)$, \ $\Hom_\C(M/(E+F''),N)$, and
$\Hom_\C(M/F'',N)$ in $\Hom_\C(M,N)$ coincide.
 Hence it follows from the compatibility of the morphisms $f_E$, $f_{F''}$,
and $f_{E+F''}$ with respect to the restriction of morphisms to subobjects
that the equality $h|_{F''}=f_{F''}$ implies $h|_{E+F''}=f_{E+F''}$
and $h|_E=f_E$, as required.

 %Indeed, consider the natural epimorphism $p\:\coprod_E E\rarrow M$.
 %The system of morphisms~$f_E$ defines a morphism $f\:\coprod_EE
%\rarrow N$.
 %We have to show that the morphism~$f$ annihilates the kernel of
%the epimorphsm~$p$.
 %Let $b\:B\rarrow\coprod_EE$ be a morphism from a weakly finitely
%generated object $B$ with the image lying in the kernel of~$p$.
 %It suffices to show that $fb=0$ for every such~$b$.
%
 %The morphism~$b$ factorizes through the coproduct of a finite subset
%of objects $\coprod_{i=1}^m E_i\subset\coprod_E E$.
 %Denote by~$b'$ the related morphism $B\rarrow\coprod_{i=1}^m E_i$.
 %Let $F\subset M$ denote the sum of the subobjects $E_i\subset M$;
%then $F$ is also a weakly finitely generated subobject in~$M$.
 %Denote by $q\:\coprod_{i=1}^m E_i\rarrow F$ the natural epimorphism.
 %Then $qb'=0$, because $pb=0$.
 %Let $g\:\coprod_{i=1}^m E_i\rarrow N$ denote the morphism with
%the components $f_{E_i}\:E_i\rarrow N$.
 %Then $g=f_Fq$, since the system of morphisms $(f_E)_E$ is compatible.
 %Thus $fb=gb'=0$.

%\begin{rem}
% From this point on, we will tacitly assume our locally weakly finitely
%generated categories to satisfy the assumption of
%Lemma~\ref{lwfg-endomorphism-ring-complete} for all of their objects,
%that is, each $M\in\C$ is the direct union of a set of its weakly
%finitely generated subobjects.
% This will make it possible to describe morphisms $M\rarrow N$ in $\C$
%in terms of their restrictions to weakly finitely generated subobjects
%$E\subset M$, as in the argument above.
%\end{rem}

 Finally, we show that the monad $\boT\:X\longmapsto\Hom_\C(M,M^{(X)})$
on the category of sets from Proposition~\ref{copowers-kleisli}
is isomorphic to the monad~$\fR[[-]]$.
 The natural morphisms $M^{(X)}\rarrow M^X$ are monomorphisms
in $\C$ by Corollary~\ref{lwfg-ab5}, so the induced 
maps of sets $\boT(X)\rarrow\prod_{x\in X}\boT(\{x\})$ are injective.
 Let us describe the image of this map.

 If a morphism $M\rarrow M^X$ factorizes through $M^{(X)}$, then
for every weakly finitely generated subobject $E\subset M$
the composition $E\rarrow M^X$ factorizes through the natural split
monomorphism $M^Z\rarrow M^X$ for some finite subset $Z\subset X$.
 Conversely, let $M\rarrow M^X$ be a morphism having this factorization
property with respect to all the weakly finitely generated subobjects
$E\subset M$.
 Then the composition of morphisms $E\rarrow M\rarrow M^X$ factorizes
through the monomorphism $M^{(X)}\rarrow M^X$.
 Let $\cE$ be a set of weakly finitely generated subobjects of $M$ such that
$M$ is the direct union of $E\in\cE$.
 Then $\coprod_{E\in\cE}E\rarrow M$ is an epimorphism, and it follows
that the morphism $M\rarrow M^X$ also factorizes through~$M^{(X)}$.

 We have shown that $\boT(X)$ as a subset in $\prod_{x\in X}
\boT(\{x\})=\fR^X$ consists precisely of all the $X$\+indexed families
of elements in $\fR$ that converge to zero in the topology of~$\fR$.
 According to the discussion in~\cite[Section~1.2]{PR}, it remains
to check that the ``summation map'' $\Sigma_X\:\boT(X)\rarrow\fR$
induced by the natural morphism $M^{(X)}\rarrow M$ is nothing but
the map of summation of converging to zero $X$\+indexed families
of elements in the topology of~$\fR$ (in the sense of the topological
limit of finite partial sums).
 This is easily demonstrated by restricting a morphism
$M\rarrow M^{(X)}$ in question to weakly finitely generated
subobjects $E\subset M$.
%
% The last assertion follows from
%Theorem~\ref{kappa-ary-dual-gabriel-popescu} and is provable
%in the same way as
%Proposition~\ref{fullyf-forgetful-functor-contramodules}.
\end{proof}

\begin{rem}
 Let $\fR$ be the topological ring from Theorem~\ref{lwfg-addm-theorem}.
 Then there are two ways to bound the size of a set $Y$ for which
the topological ring $\fS=\Hom^\fR(\fR[[Y]],\fR[[Y]])^\rop$ has
the property that the forgetful functor $\fR\contra\simeq\fS\contra
\rarrow\fS\modl$ is fully faithful.
 On the one hand, this holds whenever $\fR$ has a base of neighborhoods
of zero of the cardinality not exceeding~$Y$.
 On the other hand, let $\lambda$ be the cardinality of a set $\cF'$ of
weakly finitely generated subobjects in $M$ such that $M$ is the sum of
its subobjects belonging to~$\cF'$.
 Then for every set $X$, a morphism $f\:M\rarrow M^{(X)}$, and
$E\in\cF'$, there exists a finite subset $Z_E\subset X$ such that
the image of the morphism $f|_E\:E\rarrow M^{(X)}$ is contained in
$M^{(Z_E)}\subset M^{(X)}$.
 Set $Z=\bigcup_{E\in\cF'}Z_E$; then the cardinality of $Z$ does not
exceed~$\lambda$ and the image of the morphism~$f$ is
contained in $M^{(Z)}\subset M^{(X)}$.
 So Theorem~\ref{kappa-ary-dual-gabriel-popescu} is applicable for any
set $Y$ of the cardinality greater or equal to~$\lambda$.
\end{rem}

\subsection{Closed functors}
\label{closed-functors-subsecn}

In the next section, we will also encounter a situation where the categories in question may not be locally weakly finitely generated, but endomorphism rings of objects still carry a natural topology. To this end, we will employ the notion of closed functors.

Let $\C$ be a locally weakly finitely generated abelian category and assume that each object is a direct union of a set of its weakly finitely generated subobjects
(this is true if $\C$ has a generator and in particular in the cases mentioned in Examples~\ref{exs:lwfg}). Theorem~\ref{lwfg-addm-theorem} implies that each $\Hom_\C(M,N)$ is a complete separated topological abelian group and the composition in $\C$ is continuous---if $M,N,L\in\C$, we can view the composition $\Hom_\C(M,N)\times\Hom_\C(N,L)\rarrow\Hom_\C(M,L)$ as the restriction of multiplication in the ring $\fR=\Hom_\C(M\oplus N\oplus L,M\oplus N\oplus L)^\rop$.

Let now $\A$ be an additive category and $F\:\A\rarrow\C$ be an additive functor. We say that $F$ is a \emph{closed} functor if the following conditions are satisfied:

\begin{enumerate}
\renewcommand{\theenumi}{\Roman{enumi}}
\item $\A$ is idempotent-complete and has all set-indexed coproducts;
\item the functor $F$ is faithful and preserves coproducts;
\item for any two objects $K$ and $L\in\A$, the image of the embedding
\[\Hom_\A(K,L)\rarrow \Hom_\C(F(K),F(L))\]
is a closed subset of $\Hom_\C(F(K),F(L))$.
% and any morphism $g\:
%F(K)\rarrow F(L)$ in $\C$ such that for every weakly finitely
%generated subobject $E\subset F(K)$ there exists a morphism
%$h\:K\rarrow L$ in $\A$ for which the morphisms $g$ and $F(h)$
%coincide in the restriction to $E$, there exists a morphism
%$f\:K\rarrow L$ in $\A$ such that $g=F(f)$.
\end{enumerate}

 %Notice that when the functor $F$ is \emph{fully} faithful,
%the complicated condition~(IV) automatically holds.
%
%\begin{lem}
 %Let $\A$ be an additive category endowed with a closed additive functor
%$F\:\A\rarrow\C$.
 %Then for any object $N\in\A$ the ring\/ $\fQ=\Hom_\A(N,N)^\rop$ is
%a closed subring of the topological ring\/
%$\fR=\Hom_\C(F(N),F(N))^\rop$.
%\end{lem}
%
%\begin{proof}
 %Let $g\:F(N)\rarrow F(N)$ be a morphism in $\C$ not belonging to
%the image of the map $F\:\Hom_\A(N,N)\rarrow\Hom_\C(F(N),F(N))$,
%that is $g\in\fR\setminus\fQ$.
 %Then, according to the condition~(IV), there exists a weakly finitely
%generated subobject $E\subset F(N)$ such that in the category $\A$
%there are no morphisms $N\rarrow N$ whose image under $F$ coincides
%with~$g$ in the restriction to~$E$.
 %Now the set $\fV\subset\fR$ of all morphisms $F(N)\rarrow F(N)$
%in $\C$ coinciding with~$g$ in the restriction to $E$ is an open
%neighborhood of the element $g\in\fR$ that does not intersect
%the subring $\fQ\subset\fR$.
%\end{proof}

Thus, if $F\:\A\rarrow\C$ is a closed functor and $N\in\A$,
the ring $\fQ=\Hom_\A(N,N)^\rop$ is a closed subring of the topological ring\/
$\fR=\Hom_\C(F(N),F(N))^\rop$.
 We will endow the ring $\fQ=\Hom_\A(N,N)^\rop$ with the subspace topology.
% of a subring in $\fR=\Hom_\C(F(N),F(N))^\rop$.
 This makes $\fQ$ a complete, separated topological ring with a base of
neighborhoods of zero formed by open right ideals.

\begin{thm} \label{closed-functor-addm-theorem}
 Let\/ $\A$ be an idempotent-complete additive category endowed with
a closed additive functor $F\:\A\rarrow\C$ into a locally weakly finitely
generated abelian category\/~$\C$.
 Let $N\in\A$ be an object and $\Add(N)\subset\A$ be the full additive
subcategory formed by the direct summands of infinite coproducts of
copies of $N$ in~$\A$.
 Then the category\/ $\Add(N)$ is equivalent to the category\/
$\fQ\contra_\proj$ of projective left contramodules over the topological
ring\/ $\fQ=\Hom_\A(N,N)^\rop$.
%
 %Moreover, let $Y$ be a set of the cardinality greater or equal to
%the cardinality of some base of neighborhoods of zero in\/~$\fQ$.
 %Put $L=N^{(Y)}$ and consider the topological ring\/ $\fS=
%\Hom_\A(L,L)^\rop$.
 %Then the category\/ $\Add(N)$ is equivalent to the additive category\/
%$\fS\contra_\proj$ and the forgetful functor between the abelian
%categories\/ $\fS\contra\rarrow\fS\modl$ is fully faithful.
\end{thm}

\begin{proof}
 We have to check that the monad $\boT\:X\longmapsto\Hom_\A(N,N^{(X)})$
on the category of sets from Proposition~\ref{copowers-kleisli}
is isomorphic to the monad~$\fQ[[-]]$.
 For any set $X$, the map
\begin{equation} \label{closed-functor-monad}
 \Hom_\A(N,N^{(X)})\lrarrow\Hom_\A(N,N)^X
\end{equation}
is injective, because $F(N^{(X)})=F(N)^{(X)}$, the functor $F$ is
faithful, and the map $\Hom_\C(F(N),F(N)^{(X)})\rarrow
\Hom_\C(F(N),F(N))^X$ is injective.
 In view of the arguments in the proof of
Theorem~\ref{lwfg-addm-theorem}, the question reduces to showing
that the image of the map~\eqref{closed-functor-monad} consists
precisely of all the $X$\+indexed families of elements of
the ring $\fQ$ converging to zero in the topology of~$\fQ$.

 Indeed, for any morphism $N\rarrow N^{(X)}$, the related
$X$\+indexed family of elements in $\Hom_\A(N,N)$ converges to zero
in the topology of $\fQ$ since, viewed as a family of elements of
the ring $\fR=\Hom_\C(F(N),F(N))^\rop$, it comes from a certain morphism
$F(N)\rarrow F(N)^{(X)}$, and therefore converges to zero in
the topology of~$\fR$.

 To prove the converse implication, we have to check that every
morphism $g\:F(N)\rarrow F(N)^{(X)}$ in $\C$ whose compositions with
the coordinate projection morphisms $F(N)^{(X)}\rarrow F(N)$
are images of some morphisms $h_x\:N\rarrow N$ under the functor $F$,
is itself the image of a certain morphism $f\:N\rarrow N^{(X)}$ under
the functor~$F$.
 Here we need to use the condition~(III) again.

 Let $E\subset F(N)$ be a weakly finitely generated subobject.
 Then the restriction of the morphism~$g$ to $E$ factorizes through
the natural embedding $F(N)^{(Z)}\rarrow F(N)^{(X)}$ corresponding
to some finite subset $Z\subset X$.
 The related morphism $E\rarrow F(N)^{(Z)}$ is the restriction of
the morphism $F(h_Z)$ to $E$, where $h_Z\:N\rarrow N^{(Z)}=N^Z$ is
the morphism with the components $h_z$, \,$z\in Z$.
 Composing the morphism $h_Z$ with the embedding $N^{(Z)}\rarrow
N^{(X)}$, we obtain a morphism $h_E\:N\rarrow N^{(X)}$ in
the category $\A$ whose image under the functor $F$ coincides with
the morphism~$g$ in the restriction to the subobject $E\subset F(N)$.

Now $\Hom_\C(F(N),F(N)^{(X)})$ is a complete separated abelian group
with the base of neighborhoods of zero formed by the subgroups
$\Ann(E)=\{\,g'\:F(N)\rarrow F(N)^{(X)}\mid g'|_E=0\,\}$, and we have
proved that for any $E$, there exists $h_E\:N\rarrow N^{(X)}$ in $\A$
such that $F(h_E)$ lies in the open neighborhood $g+\Ann(E)$ of $g$.
Since the image of $\Hom_\A(N,N^{(X)})$ under $F$
is closed in $\Hom_\C(F(N),F(N)^{(X)})$
and any open neighborhood of $g$ intersects this image, it follows
that $g$ itself is in the image. In other words, $g=F(f)$ for some
$f\:N\rarrow N^{(X)}$, as required.
%
 %This proves the first assertion of the theorem; and
%the remaining assertions are provable as above.
\end{proof}

%\begin{cor}
 %Let\/ $\A$ be an abelian category such that either \par
%\textup{(a)} $\A$ is locally weakly finitely generated, or \par
%\textup{(b)} $\A$ is endowed with a closed additive functor\/
%$F\:\A\rarrow\C$ to a locally weakly finitely generated abelian
%category~$\C$. \par
 %Suppose further that\/ $\A$ has infinite products and an injective
%cogenerator.
 %Then, for any $n$\+tilting object $T\in\A$, the abelian category\/
%$\B$ in the heart of the tilting t\+structure on\/ $\D^\b(\A)$
%associated with $T$ is equivalent to the abelian category of left
%contramodules\/ $\fR\contra$ over the topological ring\/
%$\fR=\Hom_\A(T,T)^\rop$.
 %The left exact functor\/ $\Psi\:\A\rarrow\B$ can be computed as
%the functor\/ $\Hom_\A(T,{-})\:\A\rarrow\fR\contra$.
%
 %Moreover, let $Y$ be a set of the cardinality greater or equal to
%the cardinality of some base of neighborhoods of zero in\/~$\fR$.
 %Put $L=T^{(Y)}$ and consider the topological ring\/
%$\fS=\Hom_\A(L,L)^\rop$.
 %Then the abelian category\/ $\B$ is equivalent to\/ $\fS\contra$ and
%the forgetful functor\/ $\fS\contra\rarrow\fS\modl$ is fully faithful.
%\qed
%\end{cor}
%
%\begin{proof}
 %Part~(a) follows from Theorem~\ref{lwfg-addm-theorem} and part~(b)
%(which is more general) from Theorem~\ref{closed-functor-addm-theorem}.
 %The assertion about the functor $\Psi$ was explained
%in~\S\ref{monads-and-equiv-subsecn}.
%\end{proof}

\Section{Examples}
\label{examples-secn}

The final section of the paper is devoted to nontrivial classes
of examples, which in fact motivated the tilting theory here.
The main observation in this context is that the tilting-cotilting
correspondence coincides under suitable homological
assumptions with the comodule-contramodule correspondence introduced
by the first-named author in \cite[\S0.2--3 and Chapters~5--6]{Psemi}.
Moreover, it turns out that, up to category equivalence,
most of these situations fit into the concrete framework discussed
in Section~\ref{big-tilting-module-secn}.

\subsection{Coalgebras over a field}
\label{coalgebras-subsecn}

Let $\cC$ be a coassociative, counital coalgebra over a field~$k$,
with comultiplication $\Delta\:\cC\rarrow\cC\otimes_k\cC$ and 
counit $\varepsilon\:\cC\rarrow k$.
We will consider (coassociative and counital) left
$\cC$\+comodules $\cM$ with the coaction $\cM\rarrow\cC\otimes_k \cM$.
It is known that any comodule is a direct union of its finite
dimensional subcomodules. In particular,
the category $\A=\cC\comodl$ of left
$\cC$\+comodules is a locally Noetherian (even locally finite)
Grothendieck abelian category and
the forgetful functor $\cC\comodl\rarrow k\modl$ is exact and
preserves coproducts. If $\cM$, $\cN\in\cC\comodl$, we will denote
the group of homomorphisms from $\cM$ to $\cN$ by $\Hom_\cC(\cM,\cN)$.

The vector space dual $\cC^*$ naturally carries the structure
of a $k$-algebra, with the multiplication of $f,g\in\cC^*$ given by
\[
(f\cdot g)(c) = \sum_{i=1}^m f(c_{2,i}) g(c_{1,i})
\quad \textrm{for each } c\in\cC \textrm{ and }
\Delta(c)=\sum_{i=1}^m c_{1,i}\otimes c_{2,i}.
\]
There is a natural functor $\cC\comodl\rarrow\cC^*\modl$ which sends
a comodule $\varpi\:\cM\rarrow\cC\otimes_k \cM$ to the module
$M=\cM$ with the action
given by
\[
\cC^*\otimes_k \cM \overset{\cC^*\otimes\varpi}\lrarrow
\cC^*\otimes_k \cC \otimes_k \cM \overset{\mathrm{ev}\otimes \cM}\lrarrow \cM.
\]
This functor is fully faithful and the essential image is closed
under submodules, quotients and direct sums~\cite[Theorem 2.1.3]{Sw}.
Thus, by \cite[Proposition VI.4.2]{Sten},
we can equip $\fA=\cC^*$ with a topology with a base of neighborhoods
of zero formed by left ideals such that $\cC\comodl\simeq\fA\discrl$.
 In fact, this is simply the canonical profinite-dimensional topology on
the dual vector space $\cC^*$ to a vector space~$\cC$; and
the topological ring $\fA$ even has a base of neighborhoods of zero
consisting of open two-sided ideals (namely, the annihilators of
finite-dimensional subcoalgebras in~$\cC$ \cite[Theorem~2.2.1]{Sw}).

On the other hand, we can consider over any coalgebra over a field the
category $\cC\contra$ of left \emph{$\cC$\+contramodules}
(see \cite[Sections~1.1--1.2]{Prev}, \cite[Section~0.2]{Psemi}
for details). As far as we are concerned here, 
the key fact from \cite[Section~2.3]{Prev}
is that there is a category equivalence
$\cC\contra\simeq\fR\contra$, where
$\fR=\Hom_\cC(\cC,\cC)^\rop=\cC^*$ equipped with the topology
given by Theorem~\ref{lwfg-addm-theorem} applied to
$\C=\cC\comodl$ and $M=\cC$.
Actually, $\fR=\cC^*=\fA$ is one and the same ring,
and the topologies on $\fR$ and $\fA$ coincide; but we will not
need to use this fact.

We will be interested in $T=\cC$ with the obvious
$\fA$-$\fR$-bimodule structure, which induces an adjunction
of the form~\eqref{eq:discr-bimodule}
in Section~\ref{big-tilting-module-secn}.
The coproducts of copies of the left $\cC$\+comodule $\cC$ are
called the \emph{cofree} left $\cC$\+comodules.
The cofree $\cC$\+comodules are injective objects in $\cC\comodl$,
and every injective $\cC$\+comodule is a direct summand of a cofree
one.
Hence $T\in\fA\discrl$ satisfies the condition~(ii) 
(no self-extensions) of the definition of a tilting object
from Section~\ref{tilting-t-structure-secn}.

Of course, the projective dimension of the left $\cC$\+comodule $\cC$
does not have to be finite.
The best one can say in general is that $\cC$ is
an $\infty$\+tilting object in the sense of~\cite{PStinfty};
see~\cite[Example~6.9]{PStinfty}.
The aspect we focus on is the so-called
\emph{comodule-contramodule correspondence} \cite[Sections~0.2.6--0.2.7]{Psemi},
which is a triangle equivalence
\[ \cC\ocn_\cC^\boL{-}\:\D^\ctr(\cC\contra) \;\rightleftarrows\; \D^\co(\cC\comodl) \;\:\!\boR\Hom_\cC(\cC,-). \]
That is, the \emph{coderived category} $\D^\co(\cC\comodl)$ is equivalent
to the \emph{contraderived category} $\D^\ctr(\cC\contra)$.
In view of the identifications $\cC\comodl\simeq\fA\discrl$ and
$\cC\contra\simeq\fR\contra$, this equivalence takes the form
\begin{equation} \label{comodule-contramodule-correspondence}
T\ocn^\boL_\fR{-}\:\D^\ctr(\fR\contra) \;\rightleftarrows\; \D^\co(\fA\discrl) \;\:\!\boR\!\Hom_\fA(T,{-}).
\end{equation}

%A version of the tilting t\+structure exists on the \emph{coderived
%category} $\D^\co(\cC\comodl)$ of left $\cC$\+comodules, in place
%of the conventional (bounded or unbounded) derived category.
%The heart of this t\+structure is the abelian category of
%left \emph{$\cC$\+contramodules} $\cC\contra$
%\cite[Sections~1.1--1.2]{Prev}, \cite[Section~0.2]{Psemi}.
In order to obtain from~\eqref{comodule-contramodule-correspondence}
derived equivalences
as in Corollary~\ref{conv-abs-derived-equivalence},
we need to enforce some homological finiteness conditions on $\cC$.
We say that the coalgebra $\cC$ is \emph{left Gorenstein} if
\begin{enumerate}
	\renewcommand{\theenumi}{\alph{enumi}}
	\item the left $\cC$\+comodule $\cC$ has finite projective dimension in $\cC\comodl$;
	\item the left $\cC$\+contramodule $\cC^*$ has finite injective dimension in $\cC\contra$.
\end{enumerate}
Equivalently, the first condition says that the injective cogenerator
$T\in\fA\discrl$ has finite projective dimension, while the second one
says that the projective generator $\fR\in\fR\contra$
has finite injective dimension.

%The latter condition can be equivalently reformulated in terms of functors $\Ctrtor^\cC_n$, the left derived functors of the contratensor product. Since one has a natural isomorphism $\Ctrtor^\cC_n(\cC,\fP)^* \cong \Ext^n_{\cC\contra}(\fP,\cC^*)$ for any left $\cC$\+contramodule $\fP$, condition~(b) simply says that the right $\cC$\+comodule $\cC$ has \emph{finite contraflat dimension} in the sense that $\Ctrtor^\cC_n(\cC,{-}) \equiv 0$ for $n\gg0$. We refer to~\cite[Section~3]{Pmc} for more details.
The first condition also implies that the functor $\boR\!\Hom_\fA(T,{-})$
in~\eqref{comodule-contramodule-correspondence} has finite homological dimension.
Since adjunction~\eqref{contratensor-module-adjunction}
in~\S\ref{contratensor-subsecn} induces a natural isomorphism
$(T\ocn^\boL_\fR-)^*=\boR\!\Hom^\fR(-,T^*)$, the second condition
implies that the functor $T\ocn^\boL_\fR-$ has finite homological
dimension as well
(in the terminology of \cite[Section~3]{Pmc}, the right comodule $\cC$
has finite contraflat dimension).

%When the coalgebra $\cC$ is left Gorenstein, the derived functors
%$\cM\longmapsto\boR\Psi_\cC(\cM)=\boR\Hom_\cC(\cC,\cM)$ and
%$\fP\longmapsto\boL\Phi_\cC(\fP)=\cC\ocn_\cC^\boL\fP$ establishing
%the triangulated equivalence $\D^\co(\cC\comodl)\simeq
%\D^\ctr(\cC\contra)$ have finite homological dimension and therefore
%take acyclic complexes to acyclic complexes.
%So they descend to the conventional derived categories,
%providing a triangulated equivalence
%$\D(\cC\comodl)\simeq\D(\cC\contra)$.
%They also take bounded complexes to bounded complexes, hence
%an equivalence of the bounded derived categories
%$\D^\b(\cC\comodl)\simeq\D^\b(\cC\contra)$.

To summarize, for any left Gorenstein coalgebra,
the triangle equivalences~\eqref{comodule-contramodule-correspondence} restrict
to equivalences
\[ T\ocn^\boL_\fR{-}\:\D^\b(\fR\contra) \;\rightleftarrows\; \D^\b(\fA\discrl) \;\:\!\boR\!\Hom_\fA(T,{-}) \]
between the bounded derived categories.
In view of Proposition~\ref{bounded-tilting-t-structure-implies-generation},
$\D^\b(\fA\modl)$ admits a tilting t\+structure and $T\in\fA\modl$ becomes
a tilting object. The tilting heart is the category $\fR\contra\simeq\cC\contra$
and the related cotilting object is $\fR=\cC^*$.
It also follows that for any left Gorenstein coalgebra $\cC$,
the projective dimension of the left $\cC$\+comodule
$\cC$ is equal to the contraflat dimension of the right
$\cC$\+comodule~$\cC$.
	
%This implies the existence of a tilting t\+structure on
%$\D^\b(\cC\comodl)$.
%By Proposition~\ref{bounded-tilting-t-structure-implies-generation},
%we conclude that $T=\cC\in\A=\cC\comodl$ is an $n$\+tilting object,
%and the category $\B=\cC\contra$ is the tilting heart.
%The related $n$\+cotilting object in $\cC\contra$ is~$\cC^*$.

\subsection{Gorenstein locally Noetherian Grothendieck categories}
\label{Gorenstein-Grothendieck-subsecn}

The description of the situation where an injective cogenerator
becomes a tilting object generalizes from categories
of comodules over coalgebras to any locally Noetherian
Grothendieck categories. We will discuss this situation,
which also includes the case of module categories~\cite{AHT}, here.

To this end, we call a locally Noetherian Grothendieck category $\A$
\emph{Gorenstein} if

\begin{enumerate}
	\renewcommand{\theenumi}{g\arabic{enumi}}
	\item all injective objects in $\A$ have finite projective dimension;
	\item $\A$ has a generator of finite injective dimension.
\end{enumerate}

Note that for any locally Noetherian Grothendieck category $\A$
there exists an injective object $J\in\A$ such that the full additive
subcategory $\Add(J)\subset\A$ coincides with the full subcategory of 
injective objects $\A_\inj\subset\A$. Equivalently, this means that
$J$ contains every indecomposable injective in $\A$ as a direct summand.
Clearly, condition~(g1) is equivalent to requiring that $J$ has finite
projective dimension.

\begin{ex} \label{locally-noetherian-gorenstein-example}
If $A$ is a two-sided Noetherian ring, then the category $\A=A\modl$
is Gorenstein if and only if $A$ is an Iwanaga--Gorenstein ring
\cite[Example~2.3]{DASS17}.
\end{ex}

\begin{rem}
One can also prove (by a variation of the argument for~\cite[Lemma~2.6]{DASS17})
that a locally Noetherian Grothendieck category is Gorenstein
if and only if it is Gorenstein in the sense of \cite[Definition~2.18]{EEGR08},
introduced by Enochs, Estrada and Garc\'{\i}a-Rozas.
\end{rem}

Given any locally Noetherian Grothendieck category $\A$,
the additive category $\A_\inj$ is according to Theorem~\ref{lwfg-addm-theorem} or
\cite[Theorem~3.6]{Prev} equivalent
to the full subcategory of projective objects $\B_\proj\subset\B$
in the abelian category $\B=\fR\contra$ of left contramodules over
the topological ring $\fR=\Hom_\A(J,J)^\rop$.
This equivalence assigns the free left $\fR$\+contramodule with
one generator $\fR\in\fR\contra_\proj$ to our chosen injective object
$J\in\A_\inj$.
There are enough injective objects in $\A$ and projective objects
in~$\B$.
Moreover, both the full subcategories $\A_\inj\subset\A$ and
$\B_\proj\subset\B$ are closed under both the infinite products
and coproducts in the abelian categories $\A$ and~$\B$
\cite[Theorem~3.6]{Prev}.

As in \S\ref{coalgebras-subsecn}, without extra homological conditions
we do not obtain an equivalence of usual derived categories of $\A$ and~$\B$,
but we rather have triangulated equivalences
\begin{equation} \label{eq:loc-Noetherian-correspondence}
\D^\co(\A)\simeq\Hot(\A_\inj)\simeq
\Hot(\B_\proj)\simeq\D^\ctr(\B),
\end{equation}
according to~\cite[Proposition~A.3.1(b)]{Pcosh} and the assertion
dual to it.
In this sense, the object $T=W=J$ is again
an $\infty$\+tilting object in $\A$ in the terminology of~\cite{PStinfty};
see~\cite[Examples~6.3 and~6.4]{PStinfty}.

The following theorem characterizes the situation where $T\in\A$ is
actually a tilting object in the sense of Section~\ref{tilting-t-structure-secn}.
In view of Example~\ref{locally-noetherian-gorenstein-example}, it can be viewed
as a generalization of results in \cite[Section~3]{AHT}, and part (3)
generalizes the case of Gorenstein coalgebras from
\S\ref{coalgebras-subsecn}.

\begin{thm} \label{locally-noetherian-gorenstein-char-theorem}
Let $\A$ be a locally Noetherian Grothendieck category and $J\in\A$ be
such that $\Add(J)=\A_\inj$. Then the following
are equivalent:
	
\begin{enumerate}
\item $\A$ is Gorenstein;
\item $J$ is a tilting object of $\A$;
\item $J$ has finite projective	dimension in $\A$ and the topological
ring $\fR=\Hom_\A(J,J)^\rop$
has finite injective dimension in~$\B=\fR\contra$.
\end{enumerate}
\end{thm}

\begin{proof}
(1)~$\Longleftrightarrow$~(2) follows directly from the equivalence
between~(1) and~(3) in Theorem~\ref{tilting-cotorsion-pair-thm} applied
to $T=J\in\A$.
	
(2)~$\Longrightarrow$~(3): If $T$ is tilting, $\fR$ is cotilting
according to Corollary~\ref{tilting-cotilting-correspondence-cor}.
By the very definition, $J$ has finite projective dimension and
$\fR$ finite injective dimension.

(3)~$\Longrightarrow$~(2):
The additive embedding functor $\A_\inj\simeq\B_\proj\rarrow\B$
can be uniquely extended to a left exact functor
$\Psi\:\A\rarrow\B$, which can be computed as taking an object
$N\in\A$ to the left $\fR$\+contramodule $\Hom_\A(J,N)$.
When the projective dimension of the object $J\in\A$ is finite,
(the right derived functor $\boR^*\Psi$ of) the functor $\Psi$ has
finite homological dimension.
	
Similarly, the additive embedding functor $\B_\proj\simeq\A_\inj
\rarrow\A$ can be uniquely extended to a right exact functor
$\Phi\:\B\rarrow\A$.
The functor taking every left $\fR$\+contramodule $\fC$ to
the abelian group $\Hom_\A(\Phi(\fC),J)$ takes, in particular,
the free $\fR$\+contramodule $\fR[[X]]$ to the abelian group
$\Hom_\A(J^{(X)},J)=\fR^X$, so this is nothing but the functor
of homomorphisms $\Hom^\fR({-},\fR)$ in the category of left
$\fR$\+contramodules.
Both the functors $\Hom_\A(\Phi({-}),J)$ and $\Hom^\fR({-},\fR)$
are left exact, so they are isomorphic as functors on the whole
abelian category $\B=\fR\contra$.
It follows that when the injective dimension of the object
$\fR\in\B$ is finite, (the left derived functor $\boL_*\Phi$ of)
the functor $\Phi$ has finite homological dimension.

Hence, if~(3) holds, the
equivalences~\eqref{eq:loc-Noetherian-correspondence}
restrict to equvalences
\[ \boL\Phi\:\D^\b(\A) \;\rightleftarrows\; \D^\b(\B) \;\:\!\boR\Psi \]
%
%Arguing as in Examples~\ref{coalgebra-infinity-tilting-example}
%and~\ref{semialgebra-infinity-tilting-example}
%(see also the argument in
%Example~\ref{coring-infinity-tilting-example}), we conclude that
%the equivalence between the coderived and the contraderived
%category $\D^\co(\A)\simeq\D^\ctr(\B)$ descends to an equivalence
%of the conventional derived categories $\D(\A)\simeq\D(\B)$,
%which restricts to an equivalence of the bounded derived
%categories $\D^\b(\A)\simeq\D^\b(\B)$.
%
By Proposition~\ref{bounded-tilting-t-structure-implies-generation}
and Corollary~\ref{bounded-cotilting-t-structure}(b), it follows
that $T=J$ is an $n$\+tilting object in $\A$ and $W=\fR$ is
an $n$\+cotilting object in~$\B$.
(Cf.\ the discussion of the injective tilting module over
a Noetherian Gorenstein ring in~\cite{AHT}
and~\cite[Example~13.8 and Theorem~17.12]{GTbook}.)
\end{proof}

\subsection{Corings and semialgebras}
\label{corings-and-semialgebras-subsecn}

The tilting equivalences from~\S\ref{coalgebras-subsecn}
also generalize in a different way, where we obtain
relatively concrete
(in the sense of Section~\ref{big-tilting-module-secn})
equivalences for corings over general rings and
semialgebras over general coalgebras.

Let $A$ be an associative ring and $\cC$ be an $A$\+$A$\+bimodule
endowed with a coassociative, counital coring structure with
the comultiplication map $\Delta\:\cC\rarrow\cC\ot_A\cC$ and the counit
map $\cC\rarrow A$\, \cite{BW}, \cite[Section~1.1]{Psemi},
\cite[Section~2.5]{Prev}.
Assume that $\cC$ is a flat right $A$\+module; then the category
of left $\cC$\+comodules $\A=\cC\comodl$ is a Grothendieck abelian
category with an injective cogenerator $W=\cC\ot_AJ$, where $J$ is
an injective cogenerator of the category of left $A$\+modules.
We also have the following proposition which makes it possible to
apply Theorem~\ref{closed-functor-addm-theorem} to $\A$.

\begin{prop} \label{prop:closed-functor-corings}
Let $A$ be an associative ring and $\cC$ a coring over~$A$.
Then the additive category\/ $\A=\cC\comodl$ admits a closed functor to
a locally weakly finitely generated abelian category.
\end{prop}

\begin{proof}
%Let $A$ be an associative ring and $\cC$ be a right $A$\+module.
%Denote by $\cC\ncomodl$ the category of \emph{noncoassociative left\/
%$\cC$\+comodules}, that is left $A$\+modules $\cM$ endowed with
%an abelian group homomorphism $\cM\rarrow\cC\ot_A\cM$.
%A morphism $\cM\rarrow\cN$ in the category $\cC\ncomodl$ is
%a morphism of left $A$\+modules such that the square diagram
%$\cM\rarrow\cC\ot_A\cM\rarrow\cC\ot_A\cN$, \ $\cM\rarrow\cN\rarrow
%\cC\ot_A\cN$ is commutative.
%Then $\cC\ncomodl$ is a cocomplete, idempotent-complete additive category
%and the forgetful functor $\cC\ncomodl\rarrow A\modl$ is closed
%(as is the forgetful functor $\cC\ncomodl\rarrow\Z\modl$).
%
The locally weakly finitely generated category in question will be simply
the category $\C=A\modl$ of left $A$\+modules. This is clearly
a locally weakly finitely generated abelian category, and the
forgetful functor $F\:\cC\comodl\rarrow A\modl$ is faithful
and preserves coproducts. 

%To check the condition~(IV), it suffices to notice that for every
%object $\cK\in\cC\ncomodl$ and every element $k\in\cK$ there exists
%a finite set of elements $k_1'$,~\dots, $k_m'\in\cK$ such that
%the image of~$k$ under the coaction map $\cK\rarrow\cC\ot_A\cK$
%can be presented in the form of a tensor $\sum_{i=1}^mc_i\ot k'_i$
%with some elements $c_i\in\cC$.
%Let $\cL\in\cC\ncomodl$ be another object and $g\:\cK\rarrow\cL$ be
%a left $A$\+module homomorphism.
%Suppose that for every element $k\in\cK$ and the related elements
%$k_1'$,~\dots, $k_m'\in\cK$ there exists a morphism
%$h\:\cK\rarrow\cL$ in the category $\cC\ncomodl$ such that
%$g(k)=h(k)$ and $g(k_i')=h(k_i')$ for all $1\le i\le m$.
%Then $g\:\cK\rarrow\cL$ is a morphism in the category $\cC\ncomodl$.

To check condition~(III) from \S\ref{closed-functors-subsecn}, we will
prove that the complement of the image of
$F\:\Hom_\cC(\cM,\cN)\rarrow\Hom_A(\cM,\cN)$,
where $\cM$, $\cN$ are any fixed $\cC$\+comodules with coactions
$\varpi_\cM\:\cM\rarrow\cC\ot_A\cM$ and $\varpi_\cN\:\cN\rarrow\cC\ot_A\cN$, respectively,
is open. To this end, let $f\:\cM\rarrow \cN$ be a morphism of left $A$\+modules
which is not a $\cC$\+comodule homomorphism. Thus, there is $m\in \cM$ such that
$(1_\cC\ot_A f)(\varpi_\cM(m))\ne\varpi_\cN(f(m))$. Let $E\subset \cM$
be the $A$\+submodule generated by $m$ and a finite set of elements
$m_1$,~\dots, $m_n\in \cM$ such that $\varpi_\cM(m)=\sum_{i=1}^nc_i\ot m_i$
with some elements $c_i\in\cC$. Then $(1_\cC\ot_Ag)(\varpi_\cM(m))\ne\varpi_\cN(g(m))$
for any $g\:\cM\rarrow \cN$ from the open neighborhood 
$\fU = \{ g\in\Hom_A(\cM,\cN)\mid f|_E = g|_E \}$ of $f$ in $\Hom_A(\cM,\cN)$.
\end{proof}

Assume that $\cC$ is a projective left $A$\+module.
One can define the category $\cC\contra$ of \emph{left\/
$\cC$\+contramodules} \cite[Section~III.5]{EM},
\cite[Section~4]{Bar}, \cite[Section~3.1]{Psemi},
\cite[Section~2.5]{Prev} and, under the assumption just made,
$\cC\contra$ is abelian.
%In fact, it is isomorphic to the category $\fR\contra$ of left
%contramodules over the topological ring $\fR=\Hom_\cC(\cC,\cC)^\rop$
%(see Example~\ref{coring-contramodules-example} below).

It turns out that there is an equivalence
$\cC\contra\simeq\fR\contra$, where the ring/group of
left $\cC$\+comodule/left $A$\+comodule homomorphisms
$\fR=\Hom_\cC(\cC,\cC)^\rop=\Hom_A(\cC,A)$ is endowed with
the topology given by Theorem~\ref{closed-functor-addm-theorem}.
To see this, we observe that the three full additive subcategories formed by
\begin{itemize}
\item the cofree left $\cC$\+comodules $\cC^{(X)}$ in
the category $\cC\comodl$,
\item the free left $\cC$\+contramodules
$\Hom_\cC(\cC,\cC^{(X)})=\Hom_A(\cC,A^{(X)})$ in the category of left
$\cC$\+contramodules $\cC\contra$ \cite[Section~3.1.2]{Psemi}, and
\item the free left $\fR$\+contramodules $\fR[[X]]$ in the category
$\fR\contra$
\end{itemize}
are naturally equivalent.
The equivalence between the first and the second full subcategories
is a simple form of the co-contra
correspondence~\cite[Section~5.1.3]{Psemi}, \cite[Section~3.4]{Prev},
and an equivalence between the first and the third ones follows from
Theorem~\ref{closed-functor-addm-theorem}.

From this point on, we assume that $\cC$ is a flat right and a projective
left $A$\+module.
These assumptions imply that the cofree left $\cC$\+comodule $T=\cC$
satisfies the condition~(ii) from the definition of a tilting object.
Indeed, one has
$$
\Ext^i_\cC(\cC,\cC^{(I)})\cong\Ext_A^i(\cC,A^{(I)})=0
\quad\text{for $i>0$.}
$$
In order to make it a tilting object in $\A=\cC\comodl$, we will
again introduce a homological finiteness condition.

We say that the coring $\cC$ is \emph{left Gorenstein} if
\begin{enumerate}
	\renewcommand{\theenumi}{\alph{enumi}}
	\item the comodule $T = \cC$ has finite projective dimension in $\cC\comodl$;
	\item the contramodule $W=\Hom_\cC(\cC,\>\cC\ot_AJ)=\Hom_A(\cC,J)$, where $J$ is an injective cogenerator in $A\modl$, has finite injective dimension in $\cC\contra$.
\end{enumerate}

To see that $T$ is indeed tilting in $\A=\cC\comodl$ in this case,
we start with what is going to be the adjunction between $\A$
and the tilting heart $\B=\fR\contra$.
The adjoint functors are given by the assignments $\cM\longmapsto
\Psi_\cC(\cM)=\Hom_\cC(T,\cM)$ and $\fP\longmapsto\Phi_\cC(\fP)=
T\ocn_\fR\fP$ between $\cC\comodl\ni \cM$ and $\fR\contra\ni\fP$
\cite[Sections~5.1.1--2]{Psemi}.
%Under additional homological assumptions on $A$, this adjunction again
%provides an equivalence between the coderived and the contraderived
%category as in~\eqref{comodule-contramodule-correspondence},
%see \cite[Section~5.4]{Psemi}.
When we denote by $\boR^i\Psi_\cC$ the right derived functors of
the left exact functor $\Psi_\cC = \Hom_\cC(T,{-})$ (constructed by
applying $\Psi_\cC$ to a right injective resolution of
a $\cC$\+comodule $\cM$) and by $\boL_i\Phi_\cC$ the left derived functors
of the right exact functor $\Phi_\cC = T\ocn_\fR{-}$ (constructed by applying
$\Phi_\cC$ to a left projective resolution of an $\fR$\+contramodule~$\fP$),
we can again reformulate (a) and (b) equivalently by requiring that both 
the derived functors $\boR^*\Psi_\cC$ and $\boL_*\Phi_\cC$ have finite
homological dimensions. 
The argument immediately below will show that these two homological
dimensions again coincide; for the time being, let us denote the larger
of them by~$n$.

%Let $\cC$ be a left Gorenstein coring.
If we now denote by $\E_\A\subset\cC\comodl$ the full subcategory formed
by all the left $\cC$\+comodules $\cM$ such that
$\boR^i\Psi_\cC(\cM)=0$ for all $i>0$, and by $\E_\B\subset\fR\contra$
the full subcategory formed by all the left $\fR$\+contramodules $\fP$
such that $\boL_i\Phi_\cC(\fP)=0$ for all $i>0$,
then $\E_\A$ is a coresolving subcategory in $\A=\cC\comodl$ with
every object of $\A$ having coresolution dimension~$\le n$, while
$\E_\B$ is a resolving subcategory in $\B=\fR\contra$ with
every object of $\B$ having resolution dimension~$\le n$.
The composition of the two adjoint functors $\Phi_\cC\Psi_\cC$
restricted to the full subcategory of injective left $\cC$\+comodules
$\A_\inj\subset\A$ is the identity functor, while the composition
$\Psi_\cC\Phi_\cC$ restricted to the full subcategory of projective
left $\fR$\+contramodules is the identity
functor~\cite[Section~5.1.3]{Psemi}, \cite[Section~3.4]{Prev}.
Arguing as in~\cite[proof of Theorem~5.3]{Psemi}
(cf.~\cite[Example~6.1]{PStinfty}), one easily shows
that the functors $\Psi_\cC$ and $\Phi_\cC$ take $\E_\A$ into $\E_\B$
and $\E_\B$ into $\E_\A$, and establish an equivalence between
these two exact categories.
	
Hence we obtain a derived equivalence $\D^\b(\cC\comodl)\simeq
\D^\b(\fR\contra)$ and,
% for every symbol $\st=\b$, $+$, $-$,
%$\varnothing$, $\abs+$, $\abs-$, or~$\abs$.
by applying
Proposition~\ref{bounded-tilting-t-structure-implies-generation},
one concludes that $T=\cC$ is an $n$\+tilting object in
the Grothendieck abelian category $\A=\cC\comodl$.
The tilting heart is the abelian category $\B=\fR\contra$, and
the related $n$\+cotilting object is the left $\fR$\+contramodule
$W=\Hom_A(T,J)$.

In order to present this tilting equivalence in the form
of Section~\ref{big-tilting-module-secn}, we need a stronger assumption
on the coring $\cC$; namely that $\cC$ is a projective right $A$\+module.
In this case, the group of right $A$\+module homomorphisms
$\fA=\Hom_{A^\rop}(\cC,A)$ has a ring structure given by the homomorphism
of $A$\+$A$\+bimodules
\[
\fA\otimes_A \fA \rarrow \Hom_{A^\rop}(\cC,\Hom_{A^\rop}(\cC,A)) \cong
\Hom_{A^\rop}(\cC\otimes_A\cC,A) \rarrow \Hom_{A^\rop}(\cC,A)=\fA,
\]
where the first map sends $f\ot g$ to the homomorphism
of right $A$\+modules $c\mapsto f\cdot g(c)$. More explicitly,
the multiplication of $f$, $g\in\fA$ is given by
\[
(f\cdot g)(c) = \sum_{i=1}^m f(g(c_{1,i})\cdot c_{2,i})
\quad \textrm{for each } c\in\cC \textrm{ and }
\Delta(c)=\sum_{i=1}^m c_{1,i}\otimes c_{2,i}.
\]
Furthermore, each left $\cC$\+comodule $\varpi\:\cM\rarrow\cC\ot_A \cM$
acquires a left $\fA$\+module structure via
$fm = \sum_{i=1}^n f(c_{-1,i})m_{0,i}$ for each $f\in\fA$ and
$m\in \cM$, where $\varpi(m)=\sum_{i=1}^nc_{-1,i}\ot m_{0,i}$. 
As in~\S\ref{coalgebras-subsecn}, this construction is functorial
and the resulting functor $\cC\comodl\rarrow\fA\modl$ is exact and fully
faithful, and its essential image is closed under coproducts, submodules
and quotient modules. Hence, using \cite[Proposition VI.4.2]{Sten},
we can again equip $\fA$ with a topology with a base of neighborhoods
of zero formed by left ideals such that $\cC\comodl\simeq\fA\discrl$.
We summarize our findings in the following proposition.

\begin{prop} \label{coring-prop}
Let $A$ be an associative ring and $\cC$ be a coring over $A$ such
that $\cC$ is projective both as a left and a right $A$\+module.
The the abelian groups\/ $\fA=\Hom_{A^\rop}(\cC,A)$ and\/
$\fR=\Hom_A(\cC,A)$ have structures of topological rings with
bases of neighborhoods of zero formed by open left ideals for\/ $\fA$
and open right ideals for\/~$\fR$. There are equivalences
$\cC\comodl\simeq\fA\discrl$ and $\cC\contra\simeq\fR\contra$.

If, moreover, $\cC$ is left Gorenstein, then $T=\cC$ is naturally
an\/ $\fA$\+$\fR$\+bimodule, discrete from either side, which is
a tilting object in\/ $\fA\discrl$ and induces
tilting equivalences in the form~\eqref{eq:discr-tilting}
from Section~\ref{big-tilting-module-secn}. \qed
\end{prop}

\begin{rem}
 The projectivity assumption on the right $A$\+module $\cC$ in
the proposition can be weakened a bit.
 A right $A$\+module $P$ is called \emph{locally projective}~\cite{ZH} if,
for any surjective right $A$\+module morphism $f\:M\rarrow N$ and a finitely
generated submodule $F\subset P$, for any morphism $g\:P\rarrow N$
there exists a morphism $h\:P\rarrow M$ such that $g|_F=fh|_F$.
 A right $A$\+module $P$ is locally projective if and only if, for any
left $A$\+module $M$, the natural map of abelian groups $P\ot_AM\rarrow
\Hom_A(\Hom_{A^\rop}(P,A),M)$ is injective~\cite[42.10]{BW}.

 For any coring $\cC$ over $A$ one can define the ring structure on
$\fA=\Hom_{A^\rop}(\cC,A)$ as above and a natural ring homomorphism
$A\rarrow\fA$.
 The above rule also defines a left $\fA$\+module structure on any left
$\cC$\+comodule~$\cM$.
 The resulting functor $\cC\comodl\rarrow\fA\modl$ is fully faithful if and
only if the right $A$\+module $\cC$ is locally projective~\cite[19.3]{BW}.
 In this case, all the above considerations apply and the category
$\cC\comodl$ is identified with the category of discrete left $\fA$\+modules
for a certain topology of left ideals on~$\fA$.
 Algebraically, those left $\fA$\+modules $M$ which come from left
$\cC$\+comodules $\cM$ are distinguished by the condition that the image
of the action map $M\rarrow\Hom_A(\fA,M)$ is contained in
the left $A$\+submodule $\cC\ot_AM\subset\Hom_A(\fA,M)$.

 Locally projective modules are also known as ``flat strict Mittag-Leffler
modules''~\cite[\S II.2.3]{RG71}, \cite{Az}, \cite[Section~3]{HT},
``trace modules'' or ``universally torsionless modules''~\cite{Gar}.
 All locally projective modules are flat.
 All pure submodules of locally projective modules are locally projective;
hence, in particular, any syzygy module of a flat module is locally projective.
\end{rem}

\begin{rem}
Interesting special cases of equivalences as in Proposition~\ref{coring-prop}
have been studied for instance in representation theory
of finite-dimensional algebras.
The corresponding corings are called \emph{bocses} in this context.
The so-called Ringel's duality for finite-dimensional
quasi-hereditary algebras was presented as special case
of the proposition in \cite[\S4.3 and 4.4]{Kuelshammer}.
\end{rem}

\medskip

Another related notion leading to tilting equvivalences
is the one where the roles of the algebra and coalgebra are swapped.
This concept among others naturally arises in the study
of locally profinite groups (see Example~\ref{group-ring-example})
and locally linearly compact Lie algebras (see~\cite[Appendix~D]{Psemi}).

To start with, let $\cC$ be a coassociative, counital coalgebra
over a field~$k$ as in \S\ref{coalgebras-subsecn}.
The category of $\cC$\+$\cC$\+bicomodules with the functor ${-}\oc_\cC{-}$
of cotensor product over~$\cC$ is a monoidal category, and we can consider
a monoid $(\bcS$, $\bcS\oc_\cC\bcS \to \bcS$, $\cC \to \bcS)$ in this category.
Such $\bcS$ is by definition called a semiassociative, semiunital
\emph{semialgebra} over the coalgebra~$\cC$ \cite[Section~0.3]{Psemi}, \cite[Section~2.6]{Prev}.
	
A \emph{left semimodule} $\bcM$ over $\bcS$ is a left module object
over the monoid $\bcS$ in the left module category of left
$\cC$\+comodules over the monoidal category of $\cC$\+$\cC$\+bicomodules
(i.~e., $\bcM$ is simply a left $\cC$\+comodule endowed with an associative,
unital left action $\bcS \oc_\cC \bcM \to \bcM$).
Assume that the semialgebra $\bcS$ is an injective right
comodule over the coalgebra~$\cC$.
Then the category of left $\bcS$\+semimodules $\bcS\simodl$
is a Grothendieck abelian category. Again, we show that
Theorem~\ref{closed-functor-addm-theorem} applies.

\begin{prop} \label{prop:closed-functor-semialgebras}
Let\/ $\cC$ be a coassociative counital coalgebra over a field
and let\/ $\bcS$ a semialgebra over\/~$\cC$.
Then the additive category\/ $\A=\bcS\simodl$ admits a closed
functor to a locally weakly finitely generated abelian category.
\end{prop}

\begin{proof}
We choose $\C=\cC\comodl$ to be the locally weakly finitely
generated category. The
forgetful functor $F\:\bcS\simodl\rarrow\cC\comodl$ is clearly
faithful and preserves coproducts. 

We only need to check condition~(III) from \S\ref{closed-functors-subsecn}
for $F$. Let $\bcK$ and $\bcL$ be left $\bcS$\+semimodules with
the semiactions $\sigma_\bcK\:\bcS\oc_\cC\bcK\rarrow\bcK$ and 
$\sigma_\bcL\:\bcS\oc_\cC\bcL\rarrow\bcL$, respectively, and let
$h\:\bcK\rarrow\bcL$ a morphism of left $\cC$\+comodules which is not
a morphism of semimodules. We will show that $h$ has an open neighborhood
in $\Hom_\cC(\bcK,\bcL)$ which does not intersect the image of $F$.

To this end, we know that there exists $t\in\bcS\oc_\cC\bcK$ such
that $h(\sigma_\bcK(t))\ne\sigma_\bcL((1_\bcS\oc_\cC h)(t))$, and
that there exists a finite set of
elements $v_1$,~\dots, $v_m\in\bcK$ such that the image of~$t$ under
the natural embedding $\bcS\oc_\cC\bcK\rarrow\bcS\ot_k\bcK$ can be
presented in the form of a tensor $\sum_{i=1}^m s_i\ot v_i$ with some
elements $s_i\in\bcS$. Let $\cE\subset\bcK$ be the smallest left
$\cC$\+subcomodule containing $v_1$,~\dots, $v_m$ and $\sigma_\bcK(t)$.
Then $\cE$ is finite dimensional over~$k$ by \cite[Corollary 2.1.4]{Sw}.
In particular, $\cE$ is certainly weakly finitely generated in $\C$
and the desired open neighborhood of $h$ is
$\{ g\in\Hom_\cC(\bcK,\bcL) \mid g|_\cE=h|_\cE \}$.
\end{proof}

%Applying Theorem~\ref{closed-functor-addm-theorem}, we conclude that
%for every object $\bcN\in\bcS\nsimodl$ the full additive subcategory
%$\Add(\bcN)\subset\bcS\nsimodl$ is equivalent to the category
%$\fQ\contra_\proj$ of projective contramodules over the topological
%ring $\fQ=\Hom_\bcS(\bcN,\bcN)^\rop$ of endomorphisms of the object
%$\bcN$ in the category $\bcS\nsimodl$.
%The equivalence is provided by the functor $\Hom_\bcS(\bcN,{-})\:
%\bcS\nsimodl\rarrow\fQ\contra$.
%
%In particular, let $\bcS$ be a semiunital, semiassociative semialgebra
%over a counital, coassociative coalgebra $\cC$ over the field~$k$
%(cf.\ Example~\ref{semialgebra-infinity-tilting-example}).
%Then the category $\bcS\simodl$ of (conventional semiassociative and
%semiunital) left $\bcS$\+semimodules is a full additive subcategory
%closed under coproducts and the images of idempotent endomorphisms
%in $\bcS\nsimodl$.
%Hence for any semimodule $\bcN\in\bcS\simodl$ the additive category
%$\Add(\bcN)\subset\bcS\simodl$ can be described as above.

If we now apply Theorem~\ref{closed-functor-addm-theorem} and use
the fact that the semialgebra $\bcS$ is naturally a left semimodule
over itself, we obtain an equivalence between
the category $\Add(\bcS)\subset\bcS\simodl$
and the category of projective contramodules
$\fR\contra_\proj$ over the topological ring
$\fR=\Hom_\bcS(\bcS,\bcS)^\rop$. The equivalence is given by
the functor $\Hom_\bcS(\bcS,{-})\:\bcS\simodl\rarrow\fR\contra$.

Assume that $\bcS$ is an injective left $\cC$\+comodule.
Then one can define the category $\bcS\sicntr$ of
\emph{left $\bcS$\+semicontramodules}, it is an abelian category,
and the forgetful functor from it to the category of
$k$\+vector spaces is exact~\cite[Section~0.3.5]{Psemi}.
%The same applies to the category $\simodr\bcS$ of right
%$\bcS$\+semimodules~\cite[Sections~0.3.2]{Psemi}.
%The objects of the category $\Add(\bcS)\subset\bcS\simodl$ are
%called \emph{semiprojective} left
%$\bcS$\+semimodules~\cite[Sections~3.4.3 and~6.2, and
%Proposition~6.2.3(a)]{Psemi}.
%
% Under the above assumption, besides the structure of a topological ring,
Moreover, the $k$\+vector space $\Hom_\bcS(\bcS,\bcS)=\Hom_\cC(\cC,\bcS)$ has
a natural structure of a left semicontramodule over
the semialgebra~$\bcS$ \cite[Section~6.1.3]{Psemi}, and
the three full additive subcategories formed by
\begin{itemize}
\item the semifree left $\bcS$\+semimodules $\bcS^{(X)}$ in
the category $\bcS\simodl$,
\item the free left $\bcS$\+semicontramodules
$\Hom_\bcS(\bcS,\bcS^{(X)})=\Hom_\cC(\cC,\bcS^{(X)})$
in the abelian category $\bcS\sicntr$, and
\item the free left $\fR$\+contramodules $\fR[[X]]$ in the category
$\fR\contra$
\end{itemize}
are naturally equivalent.
%
%In fact, the three full additive subcategories of semiprojective
%left $\bcS$\+semimodules in the additive category $\bcS\simodl$,
%projective objects in the abelian category $\bcS\sicntr$, and
%projective objects in the abelian category $\fR\contra$ are naturally
%equivalent.
%
The equivalence between the first and the second full subcategories
is a form of the \emph{semimodule-semicontramodule
correspondence}~\cite[Sections~0.3.7 and~6.2]{Psemi},
\cite[Proposition~3.5(b)]{Prev}, and an equivalence between the first
and the third ones was constructed above.
It follows that the abelian categories $\bcS\sicntr$ and
$\fR\contra$ are equivalent.
	
%Moreover, both the equivalences $\Add(\bcS)\rarrow\bcS\sicntr_\proj$
%and $\Add(\bcS)\rarrow\fR\contra_\proj$ are provided by the functor
%$\Hom_\bcS(\bcS,{-})$ (with the respective additional structure on
%the Hom group).
%Hence the equivalence $\bcS\sicntr_\proj\simeq\fR\contra_\proj$
%forms a commutative diagram with the forgetful functors
%$\bcS\sicntr_\proj\rarrow k\modl$ and $\fR\contra_\proj\rarrow k\modl$.
%Besides, the abelian category of left $\bcS$\+semicontramodules has
%enough projective objects.
%Thus the equivalence of the additive categories of projective objects
%can be uniquely extended to an equivalence of the abelian categories
%$\bcS\sicntr\simeq\fR\contra$ forming a commutative diagram with
%the forgetful functors $\bcS\sicntr\rarrow k\modl$ and
%$\fR\contra\rarrow k\modl$. {\hbadness=1050\par}

For any semialgebra $\bcS$ over a coalgebra $\cC$ over a field~$k$
satisfying the left and right injectivity assumption,
the semifree left $\bcS$\+semimodule $T=\bcS$ satisfies
the condition~(ii) from the definition of a tilting object.
Indeed, one has
$$
\Ext^i_\bcS(\bcS,\bcS^{(I)})\cong
\Ext^i_\cC(\cC,\bcS^{(I)})=0 \quad\text{for $i>0$},
$$
because the class of injective left $\cC$\+comodules is closed under
infinite direct sums, so $\bcS^{(I)}$ is an injective left
$\cC$\+comodule.
	
Again, the projective dimension of the left $\bcS$\+semimodule
$\bcS$ does not have to be finite, so $T=\bcS$
is in general an $\infty$\+tilting object in the abelian
category $\bcS\simodl$ in the sense of~\cite[Example~6.10]{PStinfty},
exactly as $\cC$ was \S\ref{coalgebras-subsecn}.
A version of the tilting t\+structure exists on
the so-called \emph{semiderived category} $\D^\si(\bcS\simodl)$ of
left $\bcS$\+semimodules, in place of the conventional (bounded or
unbounded) derived category; the heart of this t\+structure is
the abelian category $\B=\bcS\sicntr\simeq\fR\contra$ \cite[Sections~0.3.4]{Psemi}.
Moreover, in general there is an equivalence of the semiderived categories
$\D^\si(\bcS\simodl)\simeq\D^\si(\bcS\sicntr)$, which is called
the \emph{derived semimodule-semicontramodule
correspondence} in~\cite[Sections~0.3.7 and~6.3]{Psemi}.
	
If the coalgebra $\cC$ is left Gorenstein in the
sense of~\S\ref{coalgebras-subsecn}, however, then
$\bcS$ has finite projective dimension in $\bcS\simodl$,
i.~e., condition~(i) of the definition of tilting object is satisfied
by $T = \bcS$.
Moreover, the equivalence of the semiderived categories in
this case takes acyclic complexes to acyclic complexes, and
therefore descends to an equivalence of the conventional unbounded
derived categories $\D(\bcS\simodl)\simeq\D(\bcS\sicntr)$, which
further restricts to an equivalence of the bounded derived
categories $\D^\b(\bcS\simodl)\simeq\D^\b(\bcS\sicntr)$.
	
Thus, if the coalgebra $\cC$ is left Gorenstein,
then $\bcS$ is an $n$\+tilting object in the abelian category
$\A=\bcS\simodl$
by Proposition~\ref{bounded-tilting-t-structure-implies-generation}.
Similarly, the left $\bcS$\+semicontramodule $W=\bcS^*=\Hom_k(\bcS,k)$ is
an $n$\+cotilting object in the abelian category $\B=\bcS\sicntr\simeq\fR\contra$
in this case; it corresponds to a certain natural choice of
an injective cogenerator $W\in\bcS\simodl$.

\begin{rem}
The tilting equivalences for a semialgebra $\bcS$ over a left Gorenstein
coalgebra $\cC$ such that $\bcS$ is left and right injective as a $\cC$-comodule
can be also described as in~\eqref{eq:discr-tilting}
in Section~\ref{big-tilting-module-secn}.
This is since one can prove that $\bcS\simodl\simeq\fA\discrl$,
where the underlying group of $\fA$ is the group of right $\cC$\+comodule
homomorphisms $\Hom_{\cC^\rop}(\cC,\bcS)$. Then $T=\bcS$ is naturally an
$\fA$\+$\fR$\+bimodule, discrete from either side, which is tilting
in $\fA\discrl$.
%
%Finally, in the same assumptions one can prove that the category of
%discrete right $\fR$\+modules $\discr\fR$ is equivalent to
%the category of right $\bcS$\+semimodules $\simodr\bcS$, with
%the equivalence forming a commutative diagram with the forgetful
%functors to $k$\+vector spaces.
%Both these abelian categories are also equivalent to the category of
%$k$\+linear colimit-preserving functors $\fR\contra\rarrow k\modl$
%(with the forgetful functor corresponding to the right
%semimodule/discrete module~$\bcS$).
\end{rem}

Now we give a particular class of examples where the tilting equivalences
for semialgebras apply; namely, in the study locally profinite groups
(otherwise known as locally compact, totally disconnected topological groups).

\begin{ex} \label{group-ring-example}
Let $H$ be a profinite group and $k$ be a field (of possibly
finite characteristic).
Then the $k$\+vector space $\cC=k(H)$ of locally constant functions
$H\rarrow k$ has a natural structure of coassociative, counital
coalgebra over~$k$.
It can be constructed as the inductive limit $k(H)=\varinjlim_U k(H/U)$
over the open normal subgroups $U\subset H$ of the coalgebras
$k(F)=k[F]^*$ dual to the group algebras of the finite quotient groups
$F=H/U$ of the group~$H$.
The (left or right) $k(H)$\+comodules are the \emph{discrete
$H$\+modules} over~$k$, that is $k$\+vector spaces endowed with
an action of $H$ such that the stabilizer of every vector is an open
subgroup in~$H$.
	
Furthermore, let $G$ be a locally profinite group and $H\subset G$ be
a compact open subgroup.
Then the $k$\+vector space $\bcS=k(G)$ of compactly supported
locally constant functions $G\rarrow k$ has a natural structure of
semiassociative, semiunital semialgebra over~$\cC$.
The (left or right) $\bcS$\+semimodules are the \emph{smooth
$G$\+modules} over~$k$, which means, once again, $k$\+vector spaces
endowed with an action of $G$ such that the stabilizer of every vector is
an open subgroup in~$G$.
One should be careful: the vector space $\bcS$ does \emph{not} depend
on the choice of a subgroup $H$ in the given group $G$, the semialgebra
structure on $\bcS$ \emph{depends} on this choice, and the category
of $\bcS$\+semimodules again does \emph{not} depend on it
\cite[Sections~E.1.2--E.1.3]{Psemi}, \cite[Example~2.6]{Prev}.

The category of (left or right) $\bcS$\+semicontramodules does not depend
on the choice of a compact open subgroup $H$ in $G$ either.
Another name for $\bcS$\+semicontramodules is
\emph{$G$\+contramodules} over~$k$.
These can be described as $k$\+vector spaces $\fP$ endowed with
a map assigning an element of $\fP$ to every $\fP$\+valued measure of
a certain kind on the group~$G$ \cite[Section~1.8]{Prev},
\cite[Section~2]{Psm}.
Let us denote the abelian category of smooth $G$\+modules over~$k$
by $G\smooth_k=\bcS\simodl$ and the abelian category of
$G$\+contramodules over~$k$ by $G\contra_k=\bcS\sicntr$.

The special case when $k$~is a field of characteristic~$p$ and $G$ is
a $p$\+adic Lie group (such as, e.~g., the group
$\operatorname{GL}_N(\mathbb Q_p)$ of invertible $N\times N$ matrices
with rational $p$\+adic entries) is of particular interest.
In this case, for a small enough compact open subgroup $H\subset G$
(in fact, for any compact $p$\+adic Lie group $H$ without $p$\+torsion
elements), the coalgebra $\cC=k(H)$ has finite homological dimension
(equal to the dimension~$n$ of the $p$\+adic Lie group $G$ or~$H$, e.~g.,
$n=N^2$ for $G=\operatorname{GL}_N(\mathbb Q_p)$;
see~\cite[Section~3]{Koh} or~\cite[Section~0.11]{Psm} and
the references therein).
Thus, in this case the discrete $G$\+module $T=\bcS$ of compactly
supported locally constant $k$\+valued functions on $G$ is
an $n$\+tilting object in $G\smooth_k$ and the $G$\+contramodule
$W=\bcS^*$ is an $n$\+cotilting object in $G\contra_k$.
We refer to~\cite[Example~4.2]{Pmc} and the paper~\cite{Psm}
for further details.
\end{ex}

\end{document}